\documentclass[11pt,a4paper,reqno]{amsart}  

\numberwithin{equation}{section}

\usepackage{amsmath}
\usepackage{amssymb}
\usepackage{graphicx}
\usepackage{color,xcolor}
\usepackage{latexsym}
\usepackage{array, tabularx}
\usepackage[utf8]{inputenc}
\usepackage[T1]{fontenc}
\usepackage{stmaryrd}
\usepackage{comment}

\usepackage[normalem]{ ulem }
\usepackage{soul}
\usepackage{enumitem}

\usepackage{dsfont,bbm,cases}
\usepackage{subfigure}
\usepackage{geometry} 
\usepackage{float}
\newcommand{\beq}{\begin{equation}}
  \newcommand{\eeq}{\end{equation}}

\colorlet{darkgreen}{green!50!black}
\usepackage{hyperref} 
\hypersetup{	
  colorlinks=true,  
  breaklinks=true,  
  urlcolor= magenta, 
  linkcolor= blue,	
  citecolor=darkgreen,	
  pdftitle={On the stationary distribution of reflected Brownian motion in a wedge}, 
  pdfauthor={Mireille Bousquet-M\'elou, Andrew Elvey Price, Sandro Franceschi, Charlotte Hardouin and Kilian Raschel},
} 

\newtheorem{thm}{Theorem}[section]
\newtheorem{prop}[thm]{Proposition}
\newtheorem{cor}[thm]{Corollary}
\newtheorem{lem}[thm]{Lemma}
\newtheorem{defn}[thm]{Definition}

\theoremstyle{definition}
\newtheorem{rem}[thm]{Remark}

\renewcommand{\ge}{\geqslant}
\renewcommand{\geq}{\geqslant}
\renewcommand{\le}{\leqslant}
\renewcommand{\leq}{\leqslant}

\def\emm#1,{{\em #1}}

\newcommand{\s}{\mathcal{S}}
\newcommand{\tx}{\tilde x}
\newcommand{\ty}{\tilde y}

\newcommand{\G}{\mathcal G_\mathcal R}
\newcommand{\R}{\mathcal{R}}
\newcommand{\bG}{\overline{\mathcal{G}}_{\mathcal{R}}}

\newcommand{\wmI}{\widehat{\mathcal I}}
\newcommand{\wD}{\widehat \Delta}
\newcommand{\wxi}{\widehat \xi}
\newcommand{\weta}{\widehat \eta}
\newcommand{\wgamma}{\widehat \gamma}

\def\P{\mathop{\hbox{\sf P}}\nolimits}
\let\phi=\varphi
\let\vareps=\varepsilon
\let\eps=\epsilon
\let\om=\omega

\def \an {{\sf{a}}}
\def \bn {{\sf{b}}}
\def \xn {{\sf{x}}}
\def \yn {{\sf{y}}}
\def \Xn {{\sf{X}}}
\def \Yn {{\sf{Y}}}

\def\cS{\mathcal{S}} 
\def\P1{\mathbb{P}^1}
\def\C{\mathbb{C}}
\def\Q{\mathbb{Q}}
\def\Z{\mathbb{Z}}
\def\N{\mathbb{N}}
\def\RR{\mathbb{R}}

\def\hphi{\widehat \phi}
\def\hp{\widehat p}
\DeclareMathOperator{\sgn}{sgn}
\DeclareMathOperator{\erf}{erf}
\DeclareMathOperator{\cotan}{cotan}
\DeclareMathOperator{\res}{res}
\DeclareMathOperator{\Res}{Res}

\title[On the stationary distribution of the SRBM: differential properties]
{On the stationary distribution of\\
  reflected Brownian  motion in a wedge:\\
  differential properties}

\author{M.\ Bousquet-M\'elou} 
\address{Laboratoire Bordelais de Recherche en Informatique, Universit\'e de Bordeaux, CNRS, 351 cours de la Lib\'eration, 33405 Talence Cedex, France} 
\email{mireille.bousquet-melou@u-bordeaux.fr}

\author{A.\ Elvey Price} 
\address{Institut Denis Poisson, Universit\'e de Tours, CNRS, Parc de Grandmont, 37200 Tours, France}
\email{andrew.elvey@univ-tours.fr}

\author{S.\ Franceschi} 
\address{T\'el\'ecom SudParis, Institut Polytechnique de Paris, 19 place Marguerite Perey, 91120 Palaiseau, France}
\email{sandro.franceschi@telecom-sudparis.eu}

\author{C.\ Hardouin} 
\address{Institut de Math\'ematiques de Toulouse, 118 route de Narbonne, 31062 Toulouse Cedex 9, France} 
\email{charlotte.hardouin@math.univ-toulouse.fr}

\author{K.\ Raschel} \address{Laboratoire Angevin de Recherche en Mathématiques, Universit\'e d'Angers, CNRS, 2 boulevard de Lavoisier, 49000 Angers, France}
\email{raschel@math.cnrs.fr}

\thanks{This project has received funding from the European Research Council (ERC) under the European Union's Horizon 2020 research and innovation programme under the Grant Agreement No.\ 759702, and from the ANR projects DeRerumNatura (ANR-19-CE40-0018), Combiné (ANR-19-CE48-0011) and RESYST (ANR-22-CE40-0002).}

\keywords{Reflected Brownian motion in a wedge; Stationary distribution; Laplace transform; Differentially algebraic functions; $q$-Difference equations; Decoupling function; Tutte's invariants; Conformal mapping}

\makeatletter
\@namedef{subjclassname@2020}{%
  \textup{2020} Mathematics Subject Classification}
\makeatother

\subjclass[2020]{Primary 60E10, 60J65 -- Secondary 12H05, 12H10, 30D05, 30F10, 39A06, 30C20}

\catcode`\@=11
\def\section{\@startsection{section}{1}%
  \z@{.7\linespacing\@plus\linespacing}{.5\linespacing}%
  {\normalfont\bfseries\scshape\centering}}

\def\subsection{\@startsection{subsection}{2}%
  \z@{.5\linespacing\@plus\linespacing}{.5\linespacing}%
  {\normalfont\bfseries\scshape}}

\def\subsubsection{\@startsection{subsubsection}{3}%
  \z@{.5\linespacing\@plus\linespacing}{-.5em}
  {\normalfont\bfseries\itshape}}

\usepackage{marginnote}
\usepackage{pifont} 

\begin{document}

\maketitle
\date{\today}

\begin{abstract}
  We consider the classical problem of determining the stationary  distribution of the semimartingale reflected Brownian motion (SRBM) in a two-dimensional wedge. 
  Under standard assumptions on the parameters of the model (opening of the wedge,  angles of the reflections,  drift), we study the algebraic and differential nature of the Laplace transform of this stationary distribution. We derive necessary and sufficient conditions for this Laplace transform to be rational, algebraic, differentially finite or more generally differentially algebraic. 
  These conditions are explicit linear dependencies between   the angles 
  of the model.

  A complicated integral expression  for this Laplace transform has recently been obtained by two authors of this paper.  In the differentially  algebraic case, we  provide  a  simple, explicit  integral-free expression in terms of a hypergeometric function. It specializes to  earlier expressions  in several classical cases:
  the skew-symmetric case, the   orthogonal reflections case and the sum-of-exponential densities  case (corresponding to the so-called Dieker-Moriarty conditions on the parameters). This paper thus closes, in a sense,  the quest of  all ``simple'' cases.

  To prove these results, we start from a functional equation that the Laplace transform satisfies, to which we apply tools from diverse horizons.
  To establish  differential algebraicity, a key ingredient is   Tutte's \emm invariant approach,, which originates in enumerative combinatorics. It allows us to express the Laplace transform (or its square) as a rational function of a certain  \emm canonical invariant,, a hypergeometric function in our context.  
  To establish differential transcendence,  we turn the functional equation into a  difference equation and apply Galoisian results on the
  nature of the solutions to such equations.
\end{abstract}


\setcounter{tocdepth}{1}  
\tableofcontents

\section{Introduction}
\label{sec:introduction}

We consider  an obliquely reflected Brownian motion in a two-dimensional convex wedge  with opening angle $\beta\in (0, \pi)$, defined by its drift $\widetilde \mu$, and two reflections angles $\delta$ and $\varepsilon$ in $(0, \pi)$ (Figure~\ref{fig:parameters}). The covariance matrix is taken to be the identity.

\begin{figure}[hbtp]
  \centering
  \includegraphics[scale=0.1]{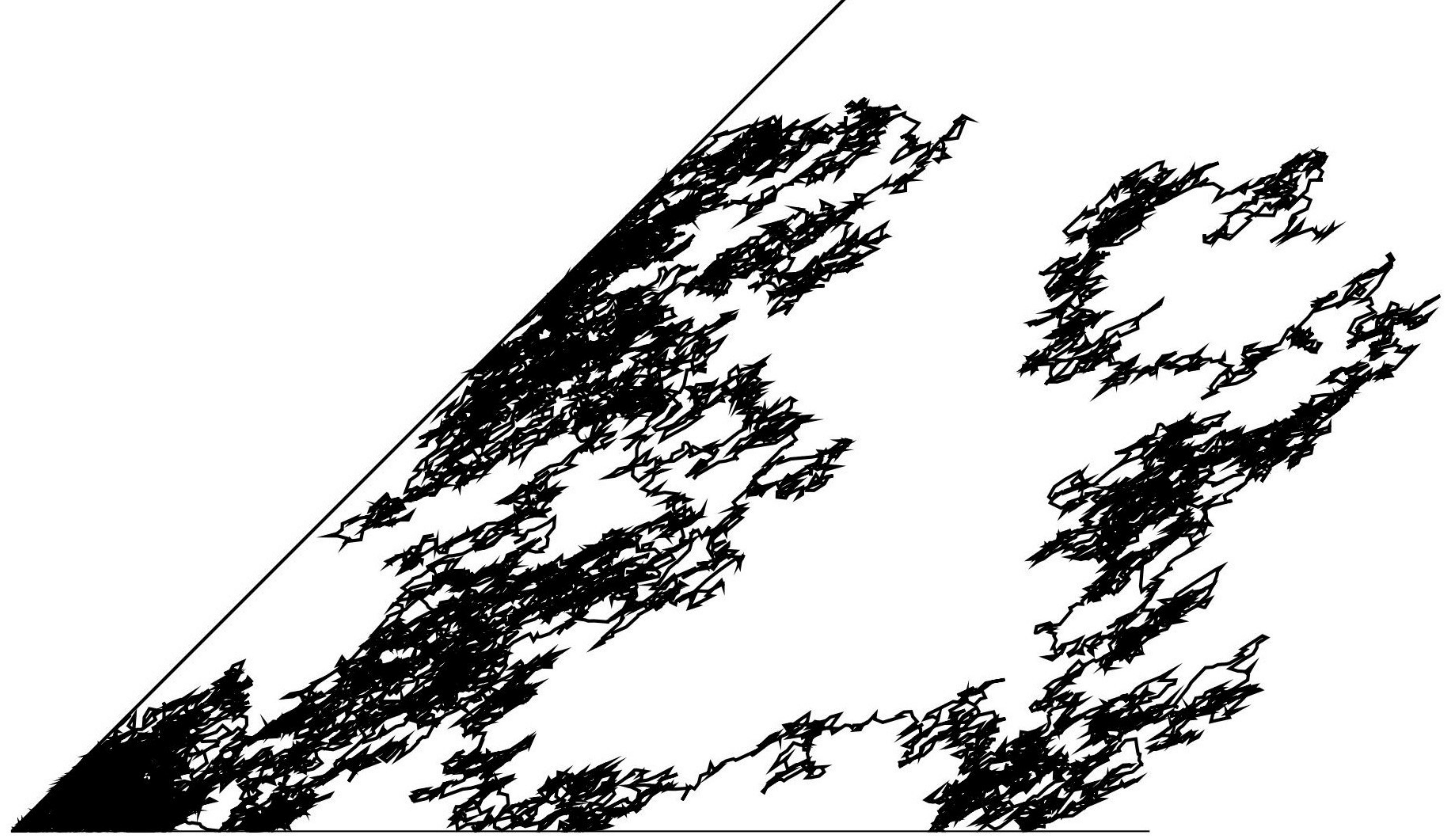}
  \hskip 10mm
  \includegraphics[trim=1.3cm 0.87cm 1cm 0.7cm, clip,scale=0.85]{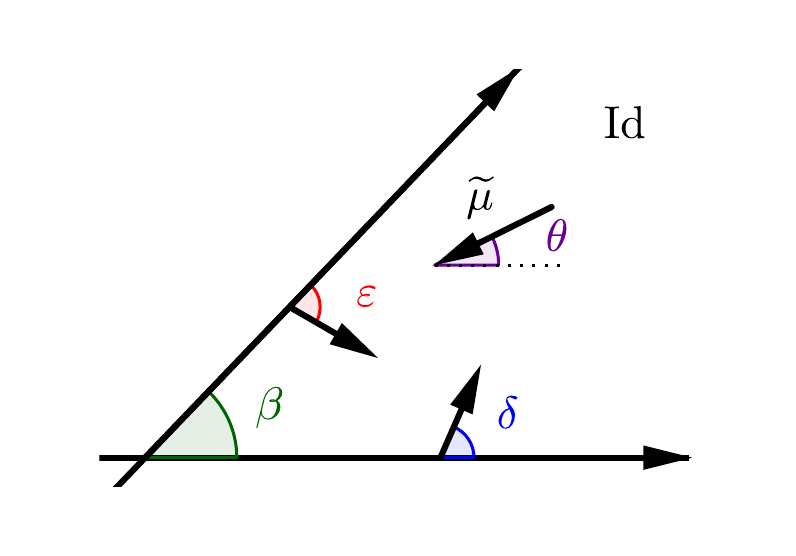}
  \caption{A trajectory of the reflected Brownian motion in a wedge, and the parameters $\beta$,  $\widetilde\mu$, $\theta$, $\delta$ and $\varepsilon$.}
  \label{fig:parameters}
\end{figure}

Since the introduction of reflected Brownian motion in the eighties~\cite{HaRe-81,HaRe-81b,varadhan_brownian_1985}, the mathematical community has shown a constant interest in this topic. Typical questions deal with  the recurrence of the process, the absorption at the corner of the wedge, the existence of stationary distributions... We refer  for more details to the introduction of \cite{franceschi_explicit_2017} and to
Figure~\ref{fig:frisealpha}.  The parameter $\alpha$ occurring there, also central in this paper, is:
\beq
\label{eq:def_alpha}
\alpha=\frac{\delta+\varepsilon-\pi}{\beta}.
\eeq
We further introduce the following refinement of $\alpha$:
\beq
\label{eq:def_alpha_i}
\alpha_1= \frac{2\varepsilon+\theta-\beta-\pi}\beta \qquad \text{and} \qquad
\alpha_2= \frac{2\delta-\theta-\pi}\beta,
\eeq
where  $\theta =\text{arg} (-\widetilde{\mu})\in (-\pi,\pi]$ as shown in Figure~\ref{fig:parameters}.  Note that $\alpha_1+\alpha_2=2\alpha-1$. These two numbers also play a key role in this paper, and it seems to be the first time that their importance is acknowledged.

\begin{figure}[hbtp]
  \centering
  \includegraphics[scale=0.13]{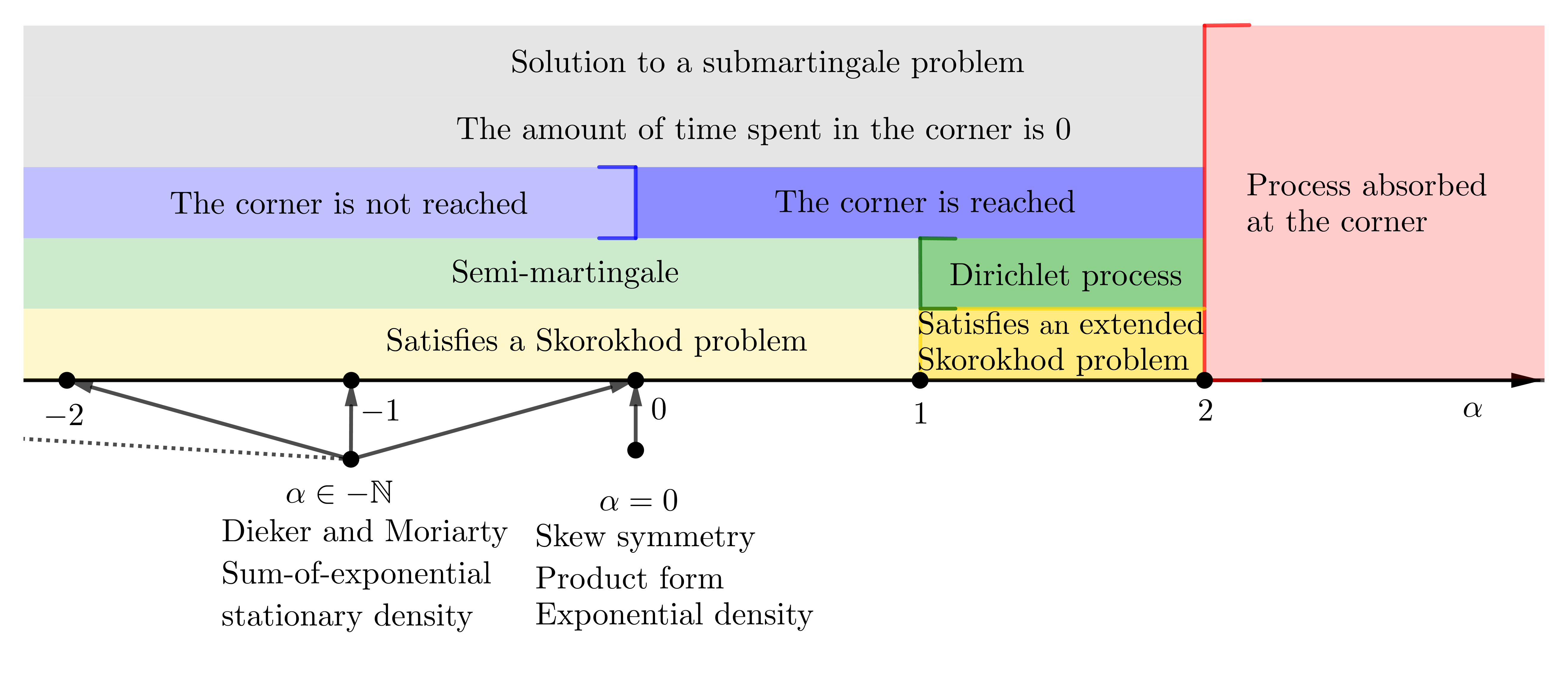}
  \caption{Properties of obliquely reflected Brownian motion in terms of $\alpha=\frac{\delta+\varepsilon-\pi}{\beta}$. Here are some references:  
    semimartingale property~\cite{Williams-85,reiman_boundary_1988,      taylor_existence_1993}; Skorokhod problem~\cite{HaRe-81b,      williams_semimartingale_1995}; submartingale problem~\cite{varadhan_brownian_1985};
    amount of time spent at the corner, accessibility of the corner and absorption~\cite{varadhan_brownian_1985};
    Dirichlet process and    extended Skorokhod problem~\cite{lakner_dirichlet_2016,      kang_dirichlet_2010};
    skew symmetry~\cite{HaRe-81,      harrison_multidimensional_1987}; sum-of-exponential stationary    density~\cite{DiMo-09}.}
  \label{fig:frisealpha}
\end{figure}

It is known~\cite{Williams-85} that the process is  a semimartingale (called \emph{semimartingale reflected Brownian motion}, SRBM for short) if and only if 
\beq
\delta+\varepsilon -\pi<\beta, \qquad \hbox{or equivalently}  \qquad
\alpha<1.
\label{eq:conditionRsemimart1}
\eeq
We assume this to hold in this paper.  We also assume that
\beq\label{theta-cond1}
0< \theta< \beta.
\eeq
The meaning of this condition will be clarified in Section~\ref{sec:normalize}; see~\eqref{eq:drift_negative0}.
Under the assumptions~\eqref{eq:conditionRsemimart1} and~\eqref{theta-cond1}, a stationary distribution exists if and only if
\beq
\beta-\varepsilon < \theta <\delta,
\label{eq:stationary_distribution_CNS1}
\eeq
and it is then unique; see~\cite[\S 3]{DiMo-09} and \cite{hobson_recurrence_1993}. We also assume this to hold.

\subsubsection*{Nature of the Laplace transform}
The main object of study in this paper is the Laplace transform $\Phi(x,y)$  of this two-dimensional stationary distribution. In a recent paper~\cite{franceschi_explicit_2017}, two of the authors gave a (complicated) closed form expression
for it, which involves integrals and various trigonometric and algebraic functions. However, it is known that when the parameter $\alpha$ is a non-positive integer (with an additional non-degeneracy condition),  the stationary density is a finite sum of exponentials, of the form
$\sum_{i} c_i e^{-a_i x- b_i y}$,
which implies that $\Phi(x,y)$ is a rational function in $x$ and $y$~\cite{DiMo-09}. This drastic simplification raises the following natural  question:
for which values of the parameters $\beta, \widetilde \mu, \delta$ and~$\varepsilon$ does the Laplace transform simplify? The case when it is rational being (mostly)  elucidated by~\cite{DiMo-09}, when is it an \emm algebraic, function of $x$ and $y$ (meaning that it satisfies a polynomial equation with coefficients in the field $\RR(x,y)$ of rational functions in $x$ and $y$)?
When is it \emm D-finite, (DF)? By this, we mean that it satisfies two linear differential equations with coefficients in $\RR(x,y)$, one in $x$ and one in~$y$. More generally, when is it \emm D-algebraic, (DA), that is, when does it satisfy a polynomial differential equation in $x$, and another in $y$? In other words, we want to \emm classify, the parameters of the semimartingale  reflected Brownian motion depending on whether, and where, the associated Laplace transform fits in the following  natural hierarchy of functions:
\beq\label{hierarchy}
\hbox{ rational } \subset \hbox{ algebraic } \subset  \hbox{ D-finite }
\subset \hbox{ D-algebraic}.
\eeq
A function that does not fit in  this hierarchy, that is, is not D-algebraic, is said to be \emm differentially transcendental, (or D-transcendental for short).

\renewcommand{\arraystretch}{1.5}
\begin{table}[h!]
  \begin{tabular}{|c|c|c|c|c|}
    \hline 
    & D-algebraic & D-finite &  Algebraic & Rational  \\ 
    \hline 
    $\beta/\pi \notin \Q$ & $\alpha \in \Z +\frac\pi \beta\Z$, {or} 
                  & $\alpha \in -\N_0 +\frac\pi \beta\Z$, {or} 
                             &  $\alpha \in -\N_0$, {or}  & $\alpha \in -\N_0$ 
    \\
    &{ $\{\alpha_1, \alpha_2 \} \subset \Z +\frac\pi \beta\Z$ }
                  &{ $\{\alpha_1, \alpha_2 \} \subset   \Z \cup \left(-\N+ \frac \pi\beta \Z\right)$}
                             &{ $\{\alpha_1, \alpha_2 \} \subset \Z $ } &\\
    \hline
    $\beta/\pi \in \Q$  & 
                          always & $\alpha \in \Z +\frac\pi \beta\Z$, {or} 
                             & $\alpha \in \Z +\frac\pi \beta\Z$, {or}  
                                          & $\alpha  \in  -\N_0$
    \\
    &&{ $\{\alpha_1, \alpha_2 \} \subset \Z +\frac\pi \beta\Z$ }
                  &{ $\{\alpha_1, \alpha_2\}\subset \Z +\frac\pi \beta\Z$ }&\\
    \hline      
  \end{tabular}
  \medskip
  \label{table:characterisation}
  \caption{Nature of the Laplace transform in terms    of $\alpha$, $\alpha_1$ and $\alpha_2$. We denote $\N_0:=\{0, 1, 2, \ldots\}$ and $\N:=\N_0\setminus\{0\}$. The condition $\alpha \in -\N_0$ can be rewritten as $\alpha\in \Z$ since we assume $\alpha <1$.}
\end{table}

\subsubsection*{Main results}
In this paper, we answer the above questions completely.
The necessary and sufficient conditions that we establish are summarized in Table~\ref{table:characterisation}. Note that they are remarkably compact, and geometric.
Observe the key role played by the parameters $\alpha$, $\alpha_1$ and $\alpha_2$ of~\eqref{eq:def_alpha},~\eqref{eq:def_alpha_i}, and in particular by the conditions
\beq\label{eq:CNS0}
\alpha \in \Z+ \frac \pi\beta \Z,\quad \hbox{or equivalently} \quad \delta+ \varepsilon \in \beta\Z+\pi\Z
\eeq
and
\beq\label{eq:CNS0double}
\{\alpha_1, \alpha_2 \} \subset \Z+ \frac \pi\beta \Z,\quad \hbox{or
  equivalently} \quad \{2\varepsilon +\theta, 2\delta-\theta\} \subset \beta\Z+\pi\Z.
\eeq 
We call them the \emm simple angle condition, and the \emm double angle condition,, respectively. When one of them holds, we give for $\Phi(x,y)$ a new, integral-free expression  in terms of the classical hypergeometric function ${_2F}_1$, from which D-algebraicity follows via  classical closure properties of DA functions.
Several explicit examples are given in Sections~\ref{sec:examples} and~\ref{sec:examples-double}.
In a sense,  this article closes the quest of  ``simple'' cases by finding and listing them all, and  providing
for them  unified and simple explicit expressions for the Laplace transform.

The algebraic and differential properties of the Laplace transform are reflected in various ways on the stationary distribution itself. Let us give two
examples, focussing, for simplicity, on the one-dimensional transform $\Phi(x,0)$ and the corresponding distribution, denoted by  $\nu$ here.
\begin{itemize}
\item {\em Moments.} If $\Phi(x,0)$ is DA, then the differential equation that it satisfies translates into a  recurrence relation for the moments $M_n$ of $\nu$. In general this relation has infinite order and its coefficients are polynomials in $n$; it becomes linear, and of finite order, as soon as $\Phi(x,0)$ is DF. An explicit example is worked out in Section~\ref{sec:DF-example}.
\item {\em Density.} If $\Phi(x,0)$ is DF, then the density of $\nu$ is DF as well. If $\Phi(x,0)$ is even rational, the density is a linear combination of terms $x^k e^{-ax}$, with $k\in \N_0$.
\end{itemize}

For the corresponding \emm discrete, problem, namely stationary distributions of discrete random walks in a wedge, a number of cases where similar simplifications occur are known: let us cite for instance the famous Jackson networks and their product form  distributions~\cite{Ja-57}, works of Latouche and Miyazawa \cite{LaMi-14} and Chen, Boucherie and Goseling~\cite{ChBoGo-15}, who obtain
necessary and sufficient conditions for the stationary distribution of random walks in the quadrant to be sums of geometric terms, and the results of Fayolle, Iasnogorodski and Malyshev in~\cite[Chap.~4]{FIM17}.
Nonetheless, it remains a challenge to find uniform criteria analogous to those that we obtain in the continuous setting. The same is true for the associated enumerative   problem, namely when one tries to understand the algebraic/differential nature of the generating function that \emm counts,  discrete walks in the quadrant~\cite{BeBMRa-16,BeBMRa-17,DHRS1}.

\subsubsection*{Tools} Let us now describe the ingredients of our proofs.
We find them to be surprisingly diverse, and we believe that one merit of this paper is to enrich the classical study of  reflected Brownian motion with two important new tools, namely Tutte's \emm invariant theory, and \emm difference Galois theory,. Let us give a few details. Our starting point is a linear functional equation,  established in~\cite{dai_reflecting_2011}, that characterizes the function $\Phi(x,y)$.
The proof of D-algebraicity (when~\eqref{eq:CNS0} or~\eqref{eq:CNS0double} holds) relies on Tutte's  invariant approach. Between 1973 and 1984, Tutte studied a functional equation that arises in the enumeration of properly colored triangulations~\cite{Tutte-95}, and has similarities with the equation defining $\Phi$. In order to solve it  (and  prove that its solution is D-algebraic), Tutte developed  an algebraic approach based on the construction of certain \emm invariants,. This approach has recently been fruitfully applied, first to other map enumeration problems~\cite{BeBM-11,BeBM-17}, and then in other contexts, such as the enumeration of walks confined to the first quadrant~\cite{BeBMRa-16,BeBMRa-17}, or avoiding a quadrant~\cite{mbm-kreweras-3-4}. A first application to
reflected Brownian motion is presented in~\cite{franceschi_tuttes_2016} in the case where $\beta=\delta=\varepsilon$. This is clearly a special  case of~\eqref{eq:CNS0}, corresponding to  orthogonal reflections on the boundaries once the wedge is deformed into a quadrant (see Section~\ref{sec:normalize}).
The present  paper goes much further than~\cite{franceschi_tuttes_2016} by  finding \emm the exact applicability of the invariant method, in the determination of the stationary distribution of the reflected Brownian motion.
This approach might  be applicable to other related problems,
such as computation of the Green function and the Martin boundary in the transient case.

The differential transcendence  result, proving that $\Phi(x,y)$ is \emm not, D-algebraic if $\beta/\pi \not \in \Q$ and neither~\eqref{eq:CNS0} nor~\eqref{eq:CNS0double} holds, also starts from the functional equation defining $\Phi(x,y)$, but relies on a completely different tool, namely \emm difference Galois theory,. Analogously to classical Galois theory, difference Galois theory builds a correspondence between the algebraic relations satisfied by the solutions of a linear functional equation and the algebraic dependencies between
the coefficients of this equation. Using this theory, one can reduce  the question of the D-transcendence to the study of the zeroes and poles of  an explicit rational function. Difference Galois  theory has recently been applied to the enumeration of {\em discrete} walks in the quadrant~\cite{DHRS1,DHRS0,DreyfHardtderiv}.  To our knowledge, this is the first time that it is applied to a {\em  continuous} random process such as SRBM.

\subsubsection*{Outline of the paper}
In Section~\ref{sec:prelim}, we define precisely the process under study and
its normalization to a quadrant.
We also give the functional equation that characterizes the Laplace transform $\Phi(x,y)$ (or more  precisely, the corresponding transform $\phi(x,y)$ on the  quadrant). We finally state our results in detail. In Section~\ref{sec:kernel} we study a bivariate polynomial, called the \emm kernel,,  involved in the functional equation. In Section~\ref{sec:invariants} we introduce the notion of \emm invariant,, and relate it to a boundary value problem satisfied by $\phi$.
In Section~\ref{sec:gluing}  we exhibit a simple invariant $w$, which is D-finite, explicit, and exists for all values of the parameters. Moreover, we prove that $w$ is  \emm canonical, in the sense that any invariant is   a rational function in $w$.
In Section~\ref{sec:existence-decoupling} we show how to construct an invariant involving $\phi$, provided a certain \emm decoupling function, exists. We then  show that such a function  exists if and only if one of the angle conditions~\eqref{eq:CNS0} or~\eqref{eq:CNS0double} holds. These two cases are then detailed, respectively, in Sections~\ref{sec:expression} and~\ref{sec:expression-double}. In particular, we obtain an expression of $\phi$ (and  $\Phi$) in terms of $w$, from which D-algebraicity follows. The {D-transcendence} condition is established in Section~\ref{sec:DT}.
Section~\ref{sec:rat} is devoted to the case $\beta/\pi  \in \Q$. 

A {\sc Maple} session, available on the first author's \href{https://www.labri.fr/perso/bousquet/publis.html}{webpage}~\cite{mbm-web}, supports most calculations of the paper.

\section{Preliminaries and main results}
\label{sec:prelim}

Let us begin with some basic notation. Recall that we denote by $\N_0:=\{0, 1, 2, \ldots\}$ the set of natural integers; and by $\N:=\N_0\setminus\{0\}$ the set of positive integers. We denote by $\RR_+$ (resp.\  $\RR_-$) the set of positive (resp.\ negative) real numbers.

\subsection{Semimartingale reflected Brownian motion (SRBM) in    a wedge}
\label{sec:normalize}

A simple linear transformation maps the reflected Brownian motion 
discussed in the introduction  (with covariance matrix the identity) 
onto a reflected Brownian motion in the first (non-negative) quadrant with
non-trivial covariance matrix. Most of the time we will work in the quadrant,  but  it will  sometimes be important to switch between these two representations, as some quantities are more simply computed or understood in one or the other of the two frameworks. To describe the quadrant normalization explicitly, we first need to give a precise definition of SRBM in the quadrant. 

We consider $(Z_t)_{t\geq0}$, an obliquely reflected Brownian motion in the first quadrant, of covariance~$\Sigma$, drift $\mu$ and reflection matrix $R$, where
\[
  \Sigma=\small\left(  \begin{array}{ll} \sigma_{11} & \sigma_{12} \\ \sigma_{12} & \sigma_{22} \end{array} \right),
  \quad
  \mu= \small\left(  \begin{array}{l} \mu_1 \\  \mu_2  \end{array} \right),
  \quad
  R=(R^1,R^2) =\small \left(  \begin{array}{cc} r_{11} & r_{12} \\ r_{21} & r_{22} \end{array} \right),
\] 
with $r_{11}>0$ and $r_{22}>0$.
The  columns $R^1$ and $R^2$ of the matrix $R$ represent the
directions in which the Brownian motion is reflected on the
boundaries; see Figure~\ref{fig:linear_transformation}, left. The so-called
\emm orthogonal reflections case, corresponds to $r_{12}=r_{21}=0$.

The process $Z_t$ exists as a semimartingale if and only if 
\beq
\det R >0 \quad \text{or} \quad ( r_{12}>0\text{ and } r_{21}>0).
\label{eq:conditionRsemimart}
\eeq
See~\cite{taylor_existence_1993,reiman_boundary_1988} for a proof of a
multidimensional version of this statement, and
\cite{williams_semimartingale_1995} for a general survey of the SRBM
in an orthant.
The reflected Brownian motion  may then be written as 
\[
  Z_t=Z_0 + B_t + \mu\cdot t + R\cdot\left(
    \begin{array}{l} L_t^1 \\  L^2_t
    \end{array} \right),\qquad \forall t\geq 0,
\] 
where $Z_0$ is an inner starting point,
$(B_t)_{t\geq0}$ is a Brownian motion with covariance $\Sigma$
starting from the origin, and $(L^1_t)_{t\geq0}$ (resp.\ $(L^2_t)_{t\geq0}$) is (up to a multiplicative constant) the local time on the $y$-axis (resp.\ $x$-axis).
The  process $(L^1_t)_{t\geq0}$ is  continuous and non-decreasing, starts
from~$0$, and increases only when the process $Z_t$ touches
the vertical boundary, which implies that for all $t\geq 0$, $\int_{0}^t \mathds{1}_{\{Z^1_s \neq 0  \}} \mathrm{d} L^1_s=0$. Of course, a similar statement holds for  $L^2_t$.

\begin{figure}[hbtp]
  \vspace{-5mm}
  \centering
  \includegraphics[scale=0.7]{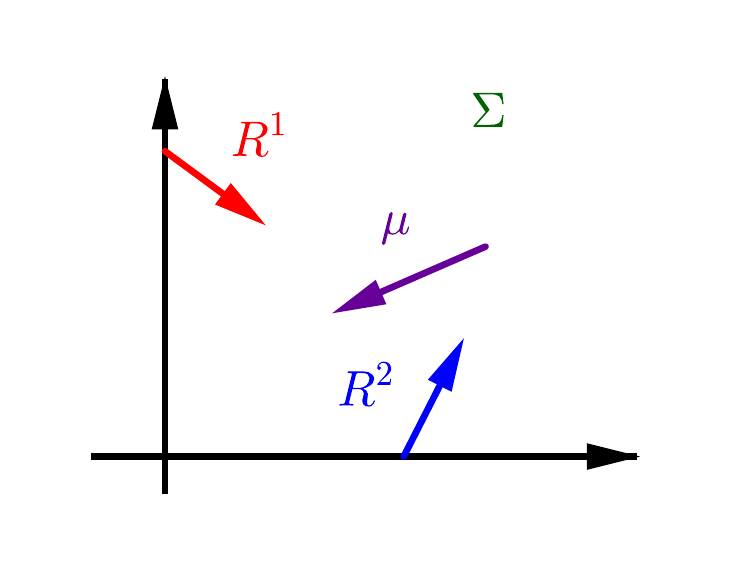} \hskip 10mm
  \includegraphics[scale=0.7]{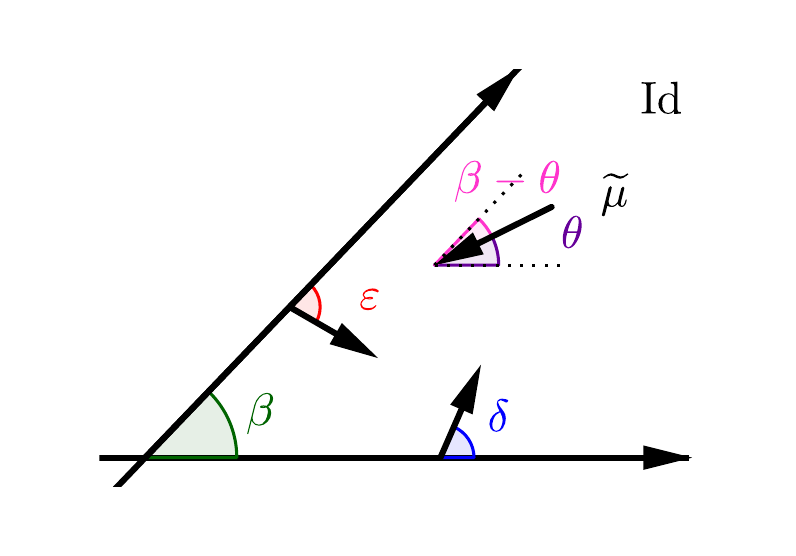}
  \vspace{-5mm}
  \caption{Transformation of the quadrant     into a wedge   of opening angle $\beta$. The new parameters    $\beta$, $\widetilde\mu$,  $\delta$ and $\varepsilon$ are     given by~\eqref{eq:beta},~\eqref{eq:mu} and~\eqref{eq:expression_delta_epsilon}, respectively.}
  \label{fig:linear_transformation}
\end{figure}

We now describe the linear transform that maps the Brownian motion in
the quadrant with   covariance matrix $\Sigma$
to a Brownian motion with covariance matrix the identity, confined to
a wedge  of opening $\beta$ (see Figure~\ref{fig:linear_transformation} and~\cite[App.~A]{franceschi_explicit_2017}). We take
\beq
\label{eq:beta}
\beta = \arccos
\left(-\frac{\sigma_{12}}{\sqrt{\sigma_{11}\sigma_{22}}}\right) \in
(0, \pi),
\eeq
so that   
\beq\label{sin-beta}
\sin \beta= \sqrt{\frac{\det\Sigma}{\sigma_{11}\sigma_{22}}}.
\eeq
Then we define a linear transformation $T$, which depends only on $\Sigma$,
\beq\label{T-def}
T= \left(
  \begin{array}{cc}
    \displaystyle     \frac 1 {\sin \beta} & \cot \beta\\
    0 & 1
  \end{array}
\right)
\left(
  \begin{array}{cc}
    \displaystyle    \frac 1 {\sqrt{\sigma_{11}}} & 0\\
    0 & \displaystyle  \frac 1 {\sqrt{\sigma_{22}}}
  \end{array}\right)
= \left(
  \begin{array}{cc}
    \displaystyle     \sqrt{\frac {\sigma_{22}}{\det\Sigma}} & -
                                                               \displaystyle          \frac{\sigma_{12}}{\sqrt{\sigma_{22}
                                                               \det\Sigma}}\medskip\\
    0 &\displaystyle \frac 1 {\sqrt{\sigma_{22}}}
  \end{array}
\right).
\eeq
This is easily inverted into
\beq\label{eq:Tinverse}
T^{-1} = 
\left(
  \begin{array}{cc}
    \displaystyle       {\sqrt{\sigma_{11}}} & 0\\
    0 &  \displaystyle {\sqrt{\sigma_{22}}}
  \end{array}\right)
\left(
  \begin{array}{cc}
    {\sin \beta} & -\cos \beta\\
    0 & 1
  \end{array}
\right)  =
\left(
  \begin{array}{cc}
    \displaystyle  \sqrt{\frac{\det\Sigma}{\sigma_{22}}} &\displaystyle \frac{\sigma_{12}}{\sqrt{\sigma_{22}}}\\
    0 &\displaystyle \sqrt{\sigma_{22}}
  \end{array}
\right) 
.
\eeq
Under the transformation $T$, the reflected Brownian motion $Z_t$  associated to $(\Sigma, \mu, R)$ becomes a Brownian motion with covariance matrix the identity in a wedge of angle $\beta$  with parameters $(\text{Id}, \widetilde \mu , \delta, \varepsilon)$. 
More explicitly, we first have $\widetilde\mu=T\mu$, that is:
\beq
\label{eq:mu}
\widetilde{\mu}_1=
\frac{\mu_1\sigma_{22}-\mu_2\sigma_{12}}{\sqrt{\sigma_{22}
    \det \Sigma}} 
\qquad
\text{and}
\qquad
\widetilde{\mu}_2= \frac{\mu_2}{\sqrt{\sigma_{22}}},
\eeq
so that, upon defining $ \theta:= \arg(-\widetilde \mu)\in (-\pi, \pi]$ (here we assume that $\mu \not = (0,0)$), 
\beq\label{eq:tan_theta}
\tan  \theta=\frac{\widetilde{\mu}_2}{\widetilde{\mu}_1}=\frac{\mu_2    \sqrt{\det\Sigma}}{\mu_1\sigma_{22}-\mu_2\sigma_{12}}    = \frac{\sin \beta}{\frac{\mu_1}{\mu_2}\sqrt{\frac{\sigma_{22}}{\sigma_{11}}} +\cos \beta}.
\eeq
More precisely,
\beq \label{theta_precise}
\theta 
=-\sgn (\widetilde{\mu}_2) \arccos \left(
  \frac{-\widetilde{\mu}_1}{\sqrt{\widetilde{\mu}_1^2+\widetilde{\mu}_2^2}}
\right)
=-\sgn ({\mu}_2) \arccos \left(\frac{\mu_2\sigma_{12}-\mu_1\sigma_{22}}{\sqrt{\sigma_{22}\Delta}}\right),
\eeq
where 
\beq\label{Delta-def}
\Delta:= |\widetilde \mu|^2 \det \Sigma= \mu_1^2\sigma_{22}-2\mu_1\mu_2\sigma_{12}+\mu_2^{2}\sigma_{11}.
\eeq
Observe that $\Delta$ is left invariant under diagonal  reflection of the quadrant model in the first diagonal. From the values of the trigonometric functions of $\beta$ and $\theta$, we also derive
\beq \label{eq:cosbetaminustheta}
\cos(\beta- \theta)=\frac{ \mu_1\sigma_{12}-\mu_2\sigma_{11} }{\sqrt{\sigma_{11}\Delta}}.
\eeq
Later we will assume that $\mu_1<0$ and $\mu_2 <0$, so that by~\eqref{theta_precise}, $\theta\in [0,\pi]$, and more precisely $\theta\in [0,\beta)$ by~\eqref{eq:tan_theta} (the function $\cotan$ is decreasing on $[0, \pi]$).  In this case, we see from~\eqref{theta_precise} and~\eqref{eq:cosbetaminustheta} that $\theta$ and $\beta-\theta$ play symmetric roles, and are exchanged under diagonal  reflection (while $\beta$ is unchanged).

The new reflection angles $\delta, \varepsilon\in(0,\pi)$ are given by:
\beq
\label{eq:expression_delta_epsilon}
\tan\delta=
\frac{\sin\beta}{\frac{r_{12}}{r_{22}}\sqrt{\frac{\sigma_{22}}{\sigma_{11}}}+\cos\beta} 
\qquad
\text{and}
\qquad
\tan\varepsilon=
\frac{\sin\beta}{\frac{r_{21}}{r_{11}}\sqrt{\frac{\sigma_{11}}{\sigma_{22}}}+\cos\beta},
\eeq
and are exchanged under  diagonal reflection; see~\cite[App.~A]{franceschi_explicit_2017}. 
Then one can prove that the semimartingale  conditions~\eqref{eq:conditionRsemimart} for the quadrant translate into   Conditions~\eqref{eq:conditionRsemimart1} for  the $\beta$-wedge; see Lemma~\ref{lem:equivcondition}~\ref{item:i}. The second result of Lemma~\ref{lem:equivcondition} states that Condition~\eqref{theta-cond1} is equivalent to the drift $\mu$ being negative:
\beq
\label{eq:drift_negative0}
\mu_1<0,\qquad \mu_2<0.
\eeq
This assumption is standard and appears for instance in~\cite{Foschini,foddy_1984,DiMo-09,franceschi_asymptotic_2016,franceschi_explicit_2017}.
We believe that it is possible to achieve a similar classification when this assumption does not hold -- but of course this would increase the number of cases.

\subsection{Invariant measures and Laplace transforms}
Assuming that Condition~\eqref{eq:conditionRsemimart} holds, the reflected Brownian motion in the quadrant has a stationary distribution if and only if~\cite{hobson_recurrence_1993}:
\beq
\label{eq:stationary_distribution_CNS}
\det R >0,
\quad r_{22} \mu_1 - r_{12}  \mu_2 < 0, \quad r_{11} \mu_2 - r_{21}  \mu_1 < 0,
\eeq
which strengthens the first part of~\eqref{eq:conditionRsemimart}. Assuming~\eqref{theta-cond1}, or equivalently~\eqref{eq:drift_negative0}, Condition~\eqref{eq:stationary_distribution_CNS} can be seen  to be equivalent to the $\beta$-wedge conditions~\eqref{eq:conditionRsemimart1} and~\eqref{eq:stationary_distribution_CNS1} combined; see Lemma~\ref{lem:equivcondition}~\ref{item:iii}.
From now on, we assume that~\eqref{eq:stationary_distribution_CNS} is
satisfied and we denote by $\Pi$ the stationary distribution, which is an invariant probability measure~\cite{harrison_brownian_1987}. Then it has a density $p_0$ relative to the Lebesgue measure on~$\RR_+^2$ \cite[Lem.~3.1]{williams_semimartingale_1995}. 
Moreover,  there exist two finite boundary measures $\nu_1$ and $\nu_2$
on the coordinate axes, defined, for $i=1,2$, by
\[
  \nu_i(\cdot)=\mathbb{E}_{\Pi}
  \left[ 
    \int_0^1 \mathbf{1}_{\{ Z_t\in\ \cdot\}} \mathrm{d}L^i_t
  \right].
\]
These measures may be considered as invariant measures (or stationary distributions) on the axes, see~\cite{harrison_brownian_1987}.
The measure $\nu_1$ (resp.\ $\nu_2$) has its support on the vertical (resp.\ horizontal) axis, where $z_1=0$ (resp.\ $z_2=0$). It has a density $p_1$ (resp.\ $p_2$) relative to the Lebesgue measure on $\RR_+$ \cite[Lem.~3.1]{williams_semimartingale_1995}.
Let $\varphi$ denote the Laplace transform of $\Pi$:
\[ 
  \varphi (x,y) = \mathbb{E}_{\Pi} [\exp{ ((x,y) \cdot  Z_t)}]
  = \iint_{{\RR}_+^2} e^{xz_1+yz_2}\, p_0(z_1,z_2)\mathrm{d} z_1\mathrm{d} z_2,
\] 
and let $\varphi_{1}$ and $\varphi_{2}$ the Laplace transforms  of $\nu_1$ and $\nu_2$:
\beq
\label{eq:Laplace_transform_boundary}
\varphi_1 (y) =\int_{{\RR}_+} e^{yz_2} 
{p_1(z_2)}\mathrm{d} z_2\quad\text{and}\quad
\varphi_2 (x) =\int_{{\RR}_+} e^{xz_1} 
{p_2(z_1)}\mathrm{d} z_1 .
\eeq
The measures $\nu_1$, $\nu_2$  are also bounded~\cite{harrison_brownian_1987} and thus these three Laplace transforms exist and are finite at least when $x$ and $y$ have non-positive real  parts.

It is known that  for all values of $x,y$ for which $\phi(x,y)$ 
is finite, the transforms $\phi_1(y)$ and $\phi_2(x)$ are finite as well, and
that $\phi(x,y)$ is a  linear combination of $\phi_1(y)$ and $\phi_2(x)$ with rational coefficients~\cite[Lem.~4.1]{dai_reflecting_2011}:
\beq  
\label{eq:functional_equation}
-\gamma (x,y) \varphi (x,y) =\gamma_1 (x,y) \varphi_1 (y) + \gamma_2 (x,y) \varphi_2 (x),
\eeq
where
\beq
\label{eq:def_gamma_gamma1_gamma2}
\left\{
  \begin{array}{rl}
    \gamma (x,y) =&\hspace{-2mm}
                    \displaystyle \frac{1}{2}  ( (x,y) \Sigma ) \cdot  (x,y)  +  (x,y) \cdot  \mu
                    =\frac{1}{2}(\sigma_{11} x^2+2\sigma_{12}xy + \sigma_{22} y^2 )
                    +
                    \mu_1x+\mu_2y, \smallskip
    \\
    \gamma_1 (x,y)=&\hspace{-2mm}\displaystyle  (x,y)    R^1 =r_{11} x + r_{21} y, \smallskip\\
    \gamma_2 (x,y)=&\hspace{-2mm} \displaystyle(x,y)  R^2    =r_{12} x + r_{22} y.
  \end{array}
\right.
\eeq
The polynomial $\gamma(x,y)$ is called the \emm kernel, of Equation~\eqref{eq:functional_equation}. By letting $x$ and/or $y$ tend to zero and noticing that $\phi(0,0)=1$,  we can conversely  express $\phi_1$ and $\phi_2$ in terms of $\phi$:
\beq
\label{eq:valuephi1O}
\phi_1(0)= \frac{\mu_1 r_{22}
  -\mu_2r_{12}}{r_{12}r_{21}-r_{11}r_{22}}, \qquad
\phi_2(0)= \frac{\mu_2 r_{11}
  -\mu_1r_{21}}{r_{12}r_{21}-r_{11}r_{22}},  
\eeq
and more generally,
\[
  r_{21}\phi_1(y)= -\left(\mu_2 + \sigma_{22}\,y/2\right) \phi(0,y)
  -r_{22}\frac{\mu_2 r_{11}
    -\mu_1r_{21}}{r_{12}r_{21}-r_{11}r_{22}},  
\]
and symmetrically for $\phi_2(x)$. Hence~\eqref{eq:functional_equation} can also be seen as a functional equation in $\phi$ only.

We can also relate the densities $p_1(z_2)$ and $p_2(z_1)$ to $p_0(z_1,z_2)$ as follows. First, the classical initial value formula gives
\[
  \lim_{x\to-\infty} x \phi(x,y)= -\int_{\RR_+} e^{y z_2} p_0(0,z_2)\mathrm{d}z_2.
\]
Then, by dividing~\eqref{eq:functional_equation} by $x$, we also obtain 
\[
  \frac{\sigma_{11}}{2} \lim_{x\to-\infty} x \phi(x,y)=-r_{11} \phi_1(y) =-r_{11} \int_{{\RR}_+} e^{yz_2} 
  {p_1(z_2)}\mathrm{d} z_2.
\]
By comparing the two limits, we find
\[
  r_{11}p_1(z_2)=  \frac{\sigma_{11}}{2} p_0(0,z_2),
\]
and analogously for $p_2(z_1)$.

Finally, the  Laplace transform $\phi(x,y)$ of $\Pi$ and the Laplace transform $\Phi(x,y)$ of the corresponding stationary distribution for the Brownian motion in the $\beta$-wedge   $T(\RR_+^2)$ (with $T$ given by~\eqref{T-def}) are  related by a linear change of variables: 
\beq\label{phi-Phi}
\phi (x,y)= \Phi((x,y)T^{-1}).
\eeq
This is proved in~\cite[Cor.~2]{franceschi_explicit_2017} when $\sigma_{11}=\sigma_{22}=1$ and still holds in our more general setting.  From this, and the above relations between $\phi$, $\phi_1$ and $\phi_2$, we see that determining the differential and algebraic nature of $\phi$ and $\Phi$ boils down to studying the nature of $\phi_1$ and $\phi_2$.

\begin{prop}\label{prop:nature-phi}
  The Laplace transform $\phi(x,y)$ is rational (resp.\ algebraic,
  D-finite, D-algebraic) if and only if $\phi_1$ and $\phi_2$ are
  rational  (resp.\ algebraic, D-finite, D-algebraic).  The same holds
  for the Laplace transform $\Phi(x,y)$.
\end{prop}

This proposition relies on various  properties of rational/algebraic/D-finite/D-algebraic functions: they include rational functions, form a ring, are closed by specialization of variables, by composition with an affine function... We refer to~\cite{lipshitz-diag,lipshitz-df,stanleyDF} for classical articles on D-finite functions, and to~\cite[Sec.~6.1]{BeBMRa-17} for a reference on D-algebraic functions.

\subsection{Main results}
\label{sec:main}

Recall  that $ r_{11} > 0$ and $r_{22} > 0$, and that  we work under the following additional assumptions:
\begin{itemize}
\item existence of a stationary distribution:
  \[ 
    \det R >0,
    \quad r_{22} \mu_1 - r_{12}  \mu_2 < 0, \quad r_{11} \mu_2 - r_{21}  \mu_1 < 0,
  \] 
\item negative drift in the quadrant model:
  \[ 
    \mu_1<0,\qquad \mu_2<0.
  \] 
\end{itemize}
In terms of the $\beta$-wedge, the angles $\beta$, $\delta$ and $\varepsilon$ are taken in $(0 ,\pi)$, and  the above combined conditions translate into: 
\beq\label{assumptions}
\delta - \pi < \beta-\varepsilon <\theta <\delta, \qquad  0  <\theta<\beta ,
\eeq
see Lemma~\ref{lem:equivcondition}~\ref{item:iii}. It seems that these equivalences were never formerly established in the SRBM literature.

We focus on $\phi_1(y)$, since the study of $\phi_2(x)$ is obviously symmetric. We distinguish two cases, depending on whether the angle
$\beta$ is a rational multiple of $\pi$, or not.

\begin{thm}
  \label{thm:main_diffalg}
  Assume that $\beta/\pi\notin \Q$. Then $\phi_1(y)$ is differentially algebraic over $\RR(y)$ if and only if one of the angle conditions~\eqref{eq:CNS0} or~\eqref{eq:CNS0double}  holds. The necessary and sufficient conditions for $\phi_1(y)$ to be D-finite, algebraic (over $\RR(y)$) or rational (over $\RR$) are those stated in the first line of Table~\ref{table:characterisation}.
\end{thm}

We now move to the case where $\beta$ is a rational multiple of $\pi$.
\begin{thm}
  \label{thm:main_alg}
  If $\beta/\pi\in \Q$, then $\frac{1}{\phi_1}\frac{d 
    \phi_1}{d y}$, the logarithmic derivative of $\phi_1$, is
  D-finite over $\RR(y)$. In particular, $\phi_1$ is  differentially algebraic.  Moreover, $\phi_1$ is algebraic if and only if~\eqref{eq:CNS0} or~\eqref{eq:CNS0double} holds, and this is the only case where it is D-finite. Finally, $\phi_1$ is rational if and only if $\alpha\in - \N_0$.
\end{thm}

When one of the angle conditions~\eqref{eq:CNS0} or~\eqref{eq:CNS0double} holds, we obtain an explicit expression of $\phi_1$.

\begin{thm}
  When~\eqref{eq:CNS0} or~\eqref{eq:CNS0double} holds, there exist an integer $m \in \{1,2\}$, and four polynomials $P(y)$, $Q(y)$, $R(z)$ and  $S(z)$ with real  coefficients such that
  \[
    \phi_1^m(y) =   \frac {Q(y)}{P(y)} \frac{S(w(y))}{R(w(y))},
  \]
  where $w(y)$ is an explicit D-finite function, defined in Section~\ref{sec:gluing} in terms of Gauss' hypergeometric function ${_2F}_1$ (see~\eqref{eq:definition_w-bis}).
\end{thm}
We refer to Theorems~\ref{thm:main} and~\ref{thm:main-double} for an explicit description of these four polynomials.

\medskip

It follows from the above theorems that the Laplace transforms $\phi_1$ and $\phi_2$ are always of  the same nature, in the sense of the hierarchy~\eqref{hierarchy} -- and then of the same nature as $\phi(x,y)$ and $\Phi(x,y)$, by Proposition~\ref{prop:nature-phi}. Indeed, as already observed in Section~\ref{sec:normalize}, applying an $x/y$-symmetry to the quadrant model exchanges $\phi_1$ and $\phi_2$, leaves  the angle $\beta$  unchanged, and exchanges  $\delta$ and~$\varepsilon$, as well as $\theta$ and $\beta-\theta$.
This implies that the parameter $\alpha$ defined by~\eqref{eq:def_alpha} is unchanged, while the parameters $\alpha_1$ and $\alpha_2$ defined by~\eqref{eq:def_alpha_i} are exchanged. Since the angle conditions 
of  Table~\ref{table:characterisation} are expressed in terms of $\alpha$, $\alpha_1$ and $\alpha_2$ only, and are symmetric in $\alpha_1$ and $\alpha_2$,
the transforms  $\phi_1$ and $\phi_2$ will  always be  of the same nature.

Our results apply in particular to  three cases  in which Condition~\eqref{eq:CNS0} holds and the Laplace transform is known to take a particularly simple form (see Figure~\ref{fig:three_examples} for an illustration):
\begin{itemize}
\item The skew symmetric case   $\delta+\varepsilon=\pi$, or equivalently  
  $\alpha=0$, studied for instance in~\cite[\S 10]{HaRe-81},~\cite{harrison_multidimensional_1987} or~\cite{DaMi-13}, and in~\cite{kang_2014} for a class of problems with state-dependent drifts.
\item The (more general) Dieker and Moriarty case
  $     \alpha\in -\N_0$ 
  (see~\cite{DiMo-09}).
\item Orthogonal reflections in the quadrant model~\cite{franceschi_tuttes_2016},  corresponding to $r_{12}=r_{21}=0$, or equivalently to
  $\delta=\varepsilon=\beta$ (see~\eqref{eq:expression_delta_epsilon}).
\end{itemize}
The transform $\phi_1$ is rational in the first two cases, and we will see that $1/\phi_1$ is D-finite in the third one (Theorem~\ref{thm:main}). We review these cases in Section~\ref{sec:examples}, together with an algebraic example where $\beta=2\pi/3$, and finally a D-finite one, $\delta+\varepsilon+\beta=2\pi$,  for which we work out explicitly the recurrence relation satisfied by the moments of $\nu_1$. In Section~\ref{sec:examples-double} we present additional interesting cases, this time where the double angle condition holds, with an emphasis on models where $\phi_1$ is algebraic while the angle $\beta$ is not necessarily in $\pi\Q$. 
This happens in particular when $\alpha_1=\alpha_2=0$ (so that $\alpha=1/2$), and in this case we prove that  the density of the stationary distribution in the $\beta$-wedge, expressed in polar coordinates $(\rho,a)$, is 
\[
  \kappa '   \,
  \frac{\cos(\frac{\theta-a}2)}{\sqrt{\rho}}
  \exp\left(-2 | \widetilde \mu |
    \, \rho \cos^2\left(\frac{\theta-a}2\right)\right),
\]
where $|\tilde \mu|$ is given by~\eqref{Delta-def}, and $\kappa'$ is an explicit constant; see~\eqref{q0}.
A similar density has already been established by Harrison~\cite[Sec.~9]{harrison_78_diffusion} in a limit case.

\begin{figure}[hbtp]
  \centering
  \includegraphics[scale=1.2]{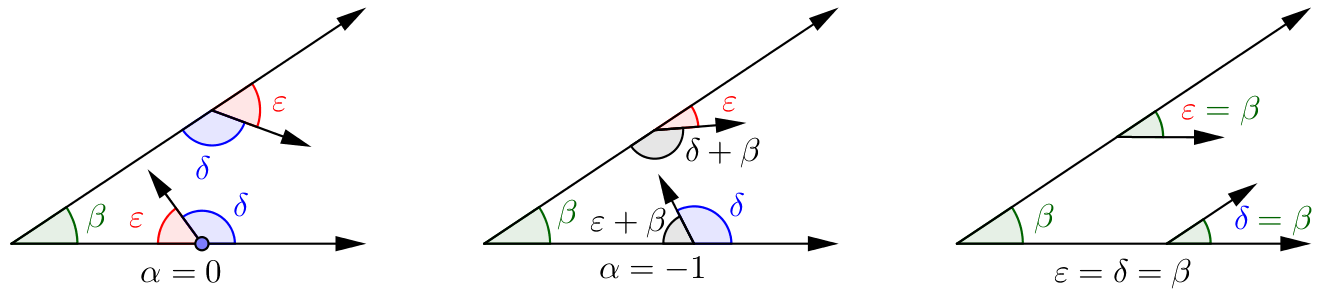}
  \caption{Three    interesting cases where the Laplace transform $\phi_1$ is D-algebraic. From left to right: skew symmetry, Dieker and    Moriarty condition (for $\alpha=-1$) and orthogonal reflections.}
  \label{fig:three_examples}
\end{figure}

\subsection{Homogeneities and normal forms}
\label{sec:normal}

The SRBM defined in the previous section involves nine real parameters (the $\sigma_{ij}$, $\mu_j$ and $r_{ij}$), but there are homogeneities between them that become visible when we move to the variables of the $\beta$-wedge. For instance, $\beta$ is unchanged if we multiply the $\sigma_{ij}$ by a positive scalar. Moreover, most quantities can now be written in many different ways, by mixing parameters from the quadrant and from the $\beta$-wedge. It will be convenient to use the $\beta$-parameters as much as possible,
keeping the quadrant parameters as prefactors only. We call \emm normal forms, such expressions. For instance, in the following identities, derived from Section~\ref{sec:normalize}, the right-hand sides are in normal form; The first two identities involve only $\Sigma$, the next three $\mu$ and $\Sigma$, and the final ones $R$ and~$\Sigma$:
\beq\label{normal-sigma}
  -\frac{\sigma_{12}}{\sqrt{\sigma_{11}\sigma_{22}}}= \cos\beta, \qquad
  \sqrt{\frac{\det \Sigma}{\sigma_{11}\sigma_{22}}}= \sin \beta,
\eeq
\beq\label{normal-mu-sigma}
  {\mu_1}\sqrt{\frac{\det \Sigma}{\sigma_{11}\Delta}}=-\sin(\beta-\theta), \qquad
  {\mu_2}\sqrt{\frac{\det \Sigma}{\sigma_{22}\Delta}}=-\sin(\theta),
  \qquad     \frac{\mu_1}{\mu_2} \sqrt{\frac{\sigma_{22}}{\sigma_{11}}}= \frac{\sin(\beta-\theta)}{\sin \theta}, 
\eeq
\beq\label{normal-R-sigma}
  \frac{r_{12}}{r_{22}} \sqrt{\frac{\sigma_{22}}{\sigma_{11}}}= \frac{\sin(\beta-\delta)}{\sin \delta}, \qquad
  \frac{r_{21}}{r_{11}} \sqrt{\frac{\sigma_{11}}{\sigma_{22}}}= \frac{\sin(\beta-\vareps)}{\sin \vareps}.
\eeq
The masses~\eqref{eq:valuephi1O}  of the measures $\nu_1$ and $\nu_2$ can now be written as
\beq\label{masses}
\phi_1(0)=  \frac 1 {r_{11}}\sqrt{\frac{\Delta}{\sigma_{22}}}
\frac{ \sin(\theta-\delta)\sin \vareps}{\sin(\beta-\delta-\varepsilon)\sin \beta },
\qquad
\phi_2(0)=  \frac 1 {r_{22}}\sqrt{\frac{\Delta}{\sigma_{11}}}
\frac{ \sin(\beta-\theta-\vareps)\sin \delta }{\sin(\beta-\delta-\varepsilon) \sin \beta}.
\eeq
This strategy, and the definition~\eqref{eq:def_gamma_gamma1_gamma2} of the kernel $\gamma(x,y)$, lead us to introduce normalized versions $\xn$ and $\yn$ of the variables $x$ and $y$ as well: we define them by
\beq\label{xy:normal}
x= \frac{\sqrt{\Delta \sigma_{22}}}{\det \Sigma}\,  \xn,  \qquad
y=  \frac{\sqrt{\Delta \sigma_{11}}}{\det \Sigma}\,  \yn.
\eeq
Then the kernel can be rewritten in normal form
\beq\label{gamma-normal}
  \gamma(x,y)=\frac \Delta{2\sin^2\beta \det \Sigma} \left( \xn^2+\yn^2-2\xn \yn \cos \beta - 2\xn \sin \beta \sin(\beta-\theta) - 2\yn \sin \beta\sin \theta\right).
\eeq
This gives the following normal forms for the other two polynomials of~\eqref{eq:def_gamma_gamma1_gamma2}:
\beq\label{gamma_i:normal}
\gamma_1(x,y)=r_{11} \frac{\sqrt{\Delta \sigma_{22}}}{\det \Sigma} \left( \xn+ \yn \, \frac{\sin (\beta-\vareps)}{\sin \vareps}\right),
\qquad
\gamma_2(x,y)=r_{22}\frac{\sqrt{\Delta \sigma_{11}}}{\det \Sigma} \left( \xn  \, \frac{\sin (\beta-\delta)}{\sin \delta}+ \yn\right).
\eeq
All identities of this subsection are implemented in our {\sc Maple} session.

\section{The kernel}
\label{sec:kernel}

In this section we  study the curve $\gamma(x,y)=0$, where $\gamma(x,y)$ is the quadratic polynomial defined in~\eqref{eq:def_gamma_gamma1_gamma2}.

\subsection{The kernel and its roots}

The roots of the kernel  $\gamma(x,y)$ (when solved for $x$, or for $y$) are algebraic functions $ X^\pm(y)$ and $ Y^\pm(x)$ defined by
\[
  \gamma( X^\pm(y), y)=\gamma(x,  Y^\pm(x))= 0.
\]
They can be expressed in closed form: 
\beq
\label{eq:definition_Theta_pm}
\left\{
  \begin{array}{l}
    X^\pm(y)=\dfrac{-(\sigma_{12}y+\mu_1)\pm\sqrt{y^2(\sigma_{12}^2-\sigma_{11}\sigma_{22})+2y(\mu_1\sigma_{12}-\mu_2\sigma_{11})+\mu_1^2}}{\sigma_{11}},
    \smallskip\smallskip
    \\
    Y^\pm(x)=\dfrac{-(\sigma_{12}x+\mu_2)\pm\sqrt{x^2(\sigma_{12}^2-\sigma_{11}\sigma_{22})+2x(\mu_2\sigma_{12}-\mu_1\sigma_{22})+\mu_2^2}}{\sigma_{22}},
  \end{array}\right.
\eeq
where we take the principal value of the square root on $\C\setminus(-\infty, 0]$.
Each of the discriminants (that is, the polynomials under the square roots) in~\eqref{eq:definition_Theta_pm} admits two roots, which are the branch points of the functions $ X^\pm$ and $ Y^\pm$. They are respectively given~by 
\beq
\label{eq:definition_theta_pm}
\left\{
  \begin{array}{l}
    \displaystyle y^\pm = \frac{(\mu_1\sigma_{12}-\mu_2\sigma_{11}) \pm
    \sqrt{
    \sigma_{11} \Delta       }}{\det\Sigma},\smallskip\smallskip\\
    \displaystyle x^\pm= \frac{(\mu_2\sigma_{12}-\mu_1\sigma_{22}) \pm \sqrt{
    \sigma_{22} \Delta           }}{\det\Sigma},
  \end{array}\right.
\eeq
with $\Delta$ defined by~\eqref{Delta-def}.
These expressions have more structure when we use the normal forms and normal variables introduced in Section~\ref{sec:normal}. If we 
write    the roots of the kernel  as:
\[
  X^\pm (y) = \frac{\sqrt{\Delta \sigma_{22}}}{\det \Sigma} \, \Xn^\pm(\yn), \qquad
  Y^\pm (x) = \frac{\sqrt{\Delta \sigma_{11}}}{\det \Sigma} \, \Yn^\pm(\xn),
\]
then  we have
\begin{alignat}{4}
  \Xn^\pm(\yn)&=\sin\beta \sin(\beta-\theta) & + \yn \cos\beta &\pm  \sin \beta \sqrt{(\yn-\yn^-)(\yn^+-\yn)}, \\
  \Yn^\pm(\yn)&=\sin \beta\sin\theta&+ \xn \cos\beta &\pm \sin \beta \sqrt{(\xn-\xn^-)(\xn^+-\xn)},
\end{alignat}
with
\beq\label{xyn-pm}
\xn^\pm =\frac{\det \Sigma}{\sqrt{\Delta \sigma_{22}}}\, {x^\pm} 
= {\cos \theta\pm 1}, \qquad
\yn^\pm =\frac{\det \Sigma}{\sqrt{\Delta \sigma_{11}}}\, {y^\pm} 
= \cos (\beta-\theta)\pm 1.
\eeq
Clearly, $\yn^+$ (and $y^+$) is positive and  $\yn^-$ (and $y^-$) is negative.
The branches $ X^\pm$ are thus analytic on $\C\setminus ((-\infty,y^-]\cup [y^+,\infty))$.  Similarly, the branches $ Y^\pm$ are analytic on $\C\setminus ((-\infty,x^-]\cup [x^+,\infty))$.

\begin{rem}\label{rem:complexconjugateroots}
  For   $x \in (-\infty, x^-] \cup [x^+, \infty)$,  the roots of $\gamma(x,y)=0$, solved for $y$, are complex conjugate. We still denote them $Y^\pm(x)$, but they are only defined up to conjugacy.
\end{rem}

\subsection{Parametrization of the curve $\boldsymbol{\gamma(x,y)=0}$}
\label{subsec:kernel}

It  will be very convenient to work with a rational uniformization (or parametrization) of the kernel, rather than 
with the variables $x$ and $y$. Let us introduce the curve
\[
  \mathcal S :=\{(x,y)\in (\C\cup\{\infty\})^2  : \gamma(x,y)=0 \},
\] 
which is a Riemann surface of genus $0$, see~\cite{franceschi_asymptotic_2016}. The following uniformization of $\mathcal S$ is  established in~\cite[Prop.~5]{franceschi_asymptotic_2016}:
\beq
\label{eq:definition_Riemann_sphere_2}
\mathcal S =\{(x(s),y(s)): s\in\C\cup\{\infty\}\}, 
\eeq
where
\beq
\label{eq:uniformization}
\left\{
  \begin{array}{l}
    x(s) =\displaystyle \frac{x^++x^-}{2}+\frac{x^+-x^-}{4}\left(s+\frac{1}{s}\right),\smallskip\\
    y(s) = \displaystyle\frac{y^++y^-}{2}+\frac{y^+-y^-}{4}\left(\frac{s}{e^{i\beta}}+\frac{e^{i\beta}}{s}\right).
  \end{array}\right.
\eeq
Recall that $x^\pm$ and $y^\pm$ are the branch points given by~\eqref{eq:definition_theta_pm}. 
In normal form,
\[
  x(s)=\frac{\sqrt{\Delta \sigma_{22}}}{\det \Sigma}\,  \xn(s),  \qquad
  y(s)=\frac{\sqrt{\Delta \sigma_{11}}}{\det \Sigma}\,  \yn(s), 
\]
with
\beq \label{param:normal}
\left\{
  \begin{array}{rll}
    \xn(s)&=\frac 1{2 
            }\left(s+e^{i\theta}\right)\left(1+ e^{-i\theta}/s\right)
    &=\frac 1{2
      }\left(2\cos \theta + s + 1 /s\right),\\
    \yn(s)&=\frac 1{2
            }\left(s+e^{i\theta}\right) \left(e^{-i\beta}+ e^{i(\beta-\theta)}/s\right)
    &= \frac 1{2
      }\left(2\cos (\beta-\theta) +e^{-i\beta} s + {e^{i\beta}}/ s\right). 
  \end{array}\right. 
\eeq
We will use repeatedly, and without mention, the fact that $x$ and $\xn$, or $x(s)$ and $\xn(s)$, and so on, only differ by a \emm positive, multiplicative factor. 
\begin{rem}\label{lem:realpointsmathcalS}
  Since  $\beta$ is real, the unit circle $|s|=1$ corresponds via the parametrization~\eqref{eq:uniformization} to the real points of $\cS$.
\end{rem}
The points $s=0$ and $s=\infty$ are sent to the unique  point at infinity of the surface $\mathcal S$.
We now introduce the transformations
\beq
\label{eq:elements_group}
\xi(s)=\frac{1}{s},\qquad \eta(s)=\frac{e^{2i\beta}}{s},\qquad \zeta(s)=\eta\xi(s)=e^{2i\beta}s.
\eeq
By construction, $\xi$ (resp.\ $\eta$) leaves $x(s)$ (resp.\
$y(s)$) invariant. By analogy with discrete models~\cite{BMM-10}, the
group $\langle \xi,\eta\rangle$ generated by $\xi$ and $\eta$ may
be called the \emm group of the model,. It is finite if and only if
$\zeta$ has finite order, \emm i.e.,,   if and only if $\beta/\pi\in\Q$.

Observe that for any $s$,  we have the following equality of  sets:
\beq \label{Ys}
\bigl\{ Y^+(x(s)), Y^-(x(s))\bigr\} = \bigl\{y(s), y(1/s)\bigr\}
\eeq
and analogously, 
\beq\label{Xs}
\bigl\{ X^+(y(s)), X^-(y(s))\bigr\} = \bigl\{x(s), x(e^{2i\beta}/s)\bigr\}.
\eeq
Also, it follows easily from~\eqref{eq:uniformization} that
\[
  x(1)=x^+,\quad
  x(-1)=x^-,\quad
  y(e^{i\beta})=y^+,\quad
  y(-e^{i\beta})=y^-.
\]
Having in mind that the index $1$ (resp.\ $2$) refers to $x$ (resp.\ $y$), we will denote accordingly
\beq\label{encombre}
s_1^+=1, \quad s_1^-=-1, \quad s_2^+=e^{i\beta}, \quad s_2^-=-e^{i\beta}.
\eeq
These    special points are shown  in  Figure~\ref{fig:ellipse_uniformization}.  The map $s\mapsto x(s)$ is $2$-to-$1$ from  $(-\infty, 0)$  onto $(-\infty, x^-]$, except at $s=-1$. It is $2$-to-$1$ from  $(0,
+\infty)$ onto $[x^+, +\infty)$, except at $s=1$. Similarly,
the map $s\mapsto y(s)$ is $2$-to-$1$ from  $e^{i\beta} \RR_{-}$
(resp.\ $e^{i\beta }\RR_{+}$)  onto $(-\infty, y^-]$ (resp.\ $[y^+,
+\infty)$), except at the point $-e^{i\beta}$ (resp.\  $e^{i\beta}$).

\begin{figure}[t]
  \vspace{-5mm}
  \centering
  \includegraphics[scale=0.8]{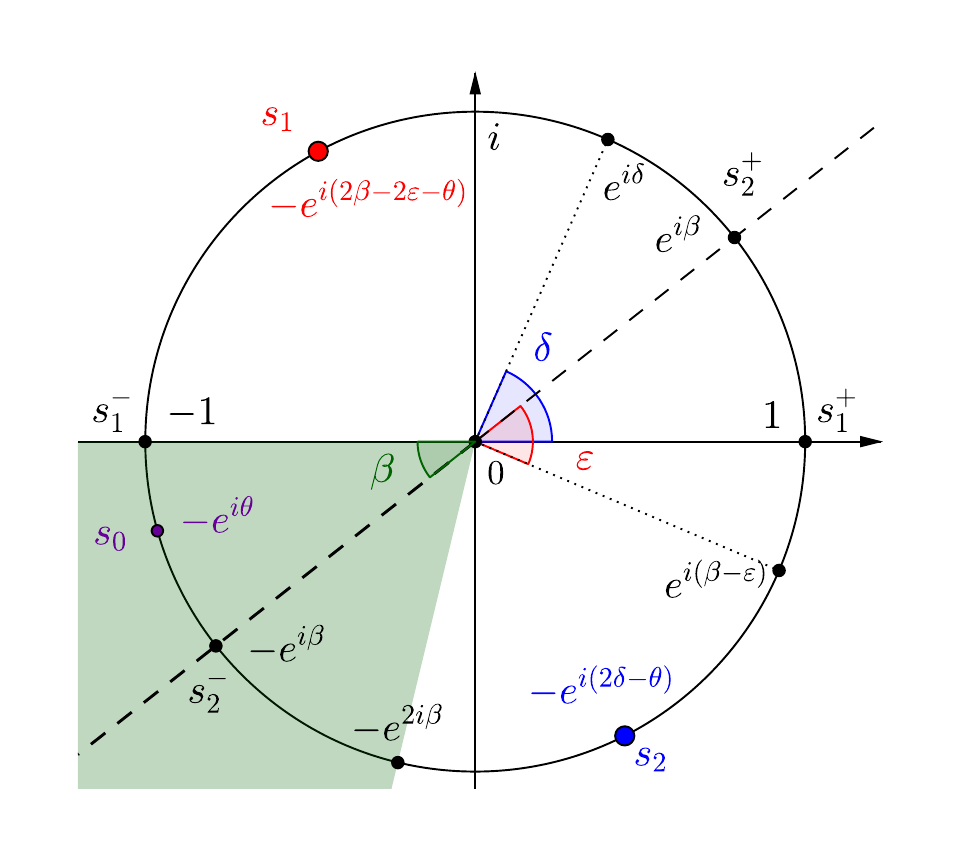}
  \vspace{-5mm}
  \caption{The complex $s$-plane, from which the uniformization~\eqref{eq:uniformization} of $\cS$ is expressed.}
  \label{fig:ellipse_uniformization}
\end{figure}

\medskip

We now establish a series of four basic lemmas that will be used later. The  first one follows directly from the normal forms~\eqref{param:normal}.
\begin{lem}
  \label{lem:s0}
  The pair of equations $x(s_0)=y(s_0)=0$ has a  unique solution, which is: 
  \[ 
    s_0=     -e^{i \theta},
  \] 
  where we recall that $\theta=\arg (-\widetilde \mu)$ is given by~\eqref{theta_precise}.
\end{lem}

Consider now, for $i\in\{1,2\}$, the following equation in $s$:
\beq
\label{eq:gamma_theta_rho_0}
\gamma_i(x(s),y(s))=0,
\eeq
where $\gamma_i(x,y)$ is the bilinear function given by~\eqref{eq:def_gamma_gamma1_gamma2}. Since $x(s)$ and $y(s)$ have degree $2$ in $s$, the above equation
is quadratic in $s$ and thus has two solutions.
One of them has to be $s_0$ (because $(x(s_0),y(s_0))=(0,0)$ obviously cancels $\gamma_i(x,y)$).  We denote the other solution by $s_i$.

\begin{lem}
  \label{lem:angles}
  The solutions $s_1$ and $s_2$ of~\eqref{eq:gamma_theta_rho_0} (distinct from $s_0$) satisfy:
  \[
    s_0s_1=e^{2i (\beta-\varepsilon)}
    \quad\text{and}\quad s_0s_2=e^{2i\delta},
  \]
  where the angles $\delta$ and $\varepsilon$ are defined by~\eqref{eq:expression_delta_epsilon}. That is,
  \[
    s_1= -e^{i (2\beta-2\varepsilon-\theta)} =e^{i\beta(1-\alpha_1)} \quad\text{and}\quad s_2=-e^{i(2\delta-\theta)} =e^{i\beta\alpha_2} ,
  \]
  where $\alpha_1$ and $\alpha_2$ are defined by~\eqref{eq:def_alpha_i}. In particular,
  \[
    \frac{s_1}{s_2}=e^{2i\beta(1-\alpha)}.
  \]  
  Moreover, $s_2 \neq -1$, and $s_1=-1$ if and only if $2\beta -2 \varepsilon-\theta=0$.
\end{lem}
The points $s_0$, $s_1$ and $s_2$ are shown in Figure~\ref{fig:ellipse_uniformization}.
\begin{proof}
  When $i=1$,    Equation~\eqref{eq:gamma_theta_rho_0} reads, with the normal form~\eqref{gamma_i:normal},
  \[
    \xn(s)+ \yn(s) \frac{\sin(\beta-\vareps)}{\sin \vareps}=0,
  \]
  and the result follows using the expressions~\eqref{param:normal}    of $\xn(s)$ and $\yn(s)$, and basic trigonometry (see our {\sc Maple} session). The expression of $s_2$ is obtained similarly.

  Then, $s_2=-1$ would mean that $\theta=2\delta$ modulo $2\pi$, hence $\theta=2\delta$ because both angles $\theta$ and $\delta$ are in $(0, \pi)$. But this is not compatible with the condition $\delta>\theta$ coming from~\eqref{assumptions}. Similarly, $s_1=-1$ means that $2\beta-2\vareps-\theta=0$ modulo $2\pi$. Since the three angles are in $(0, \pi)$, and $\beta>\theta$ by~\eqref{assumptions}, this means that $2\beta-2\vareps-\theta=0$.
\end{proof}

We go on  with a simple property of the values $y(-1)$ and $y^+$.

\begin{lem}\label{lem:Ypositive}
  The value $ Y^\pm(x^-) =y(-1)$ lies in $(  0, y^{+})$.
\end{lem}
\begin{proof}
  Thanks to the normal forms~\eqref{param:normal} and~\eqref{xyn-pm}, what we want to prove reads:
  \[
    0<    \yn(-1)=      \cos(\beta-\theta) - \cos\beta < \yn^+= \cos(\beta-\theta) +1.
  \]
  But this  is obvious since $0<\theta<\beta<\pi$.
\end{proof}

We finish with the  reformulation of  a key condition occurring in~\cite[Thm.~1]{franceschi_explicit_2017}.

\begin{lem}\label{lem:pole}
  We have $\gamma_1(x^-, Y^\pm (x^-))\geq 0$ (resp.\ $=0$) if and only if $2\beta -2\varepsilon -\theta\geq 0$ (resp.\ $=0$).
\end{lem}
\begin{proof}
  The condition reads $\gamma_1(x(-1), y(-1)) \ge 0$, or, using the normal forms~\eqref{gamma_i:normal} and~\eqref{param:normal},
  \[
    \left(\cos(\theta)-1\right)  + \frac{\sin(\beta-\vareps)}{\sin \vareps} \left( \cos(\theta-\beta)- \cos \beta\right) \ge 0.
  \]
  The left-hand sides rewrites as
  \[
    2\, \frac{\sin \beta}{\sin \vareps} \sin \frac \theta 2  \sin \frac{2\beta-2\vareps-\theta}2,
  \]
  and the result follows, because $\beta,\vareps,\theta\in (0, \pi)$ and $2\beta-2\vareps-\theta\in (-2\pi, 2\pi)$, as we have already used in the proof of  Lemma~\ref{lem:angles}.
\end{proof}

\subsection{An important curve}

Let us consider the following curve in $\C$:
\beq \label{eq:def_R}
\R =\{y\in\C : \gamma(x,y)=0   \text{ for some } x\in (-\infty,x^-]\}.
\eeq
By Remark~\ref{rem:complexconjugateroots}, the curve $\R$ is   symmetric
with respect to the real axis (Figure~\ref{fig:bvp}). Moreover, as shown in~\cite[Lem.~9]{baccelli_analysis_1987},    it is a branch of a hyperbola, which intersects the real axis at   the point $Y^-(x^-)=Y^+(x^-) \in (0, y^+)$
(see   Lemma~\ref{lem:Ypositive}). We further   introduce the domain~$\G$, which is the (open)   domain of  $\C$ containing $0$ and bounded by $\R$.
Finally, we  denote by $\bG=\G \cup \R$ the  closure of $\G$.

\begin{figure}[hbtp]
  \centering
  \vspace{-8mm}
  \includegraphics[scale=1.6]{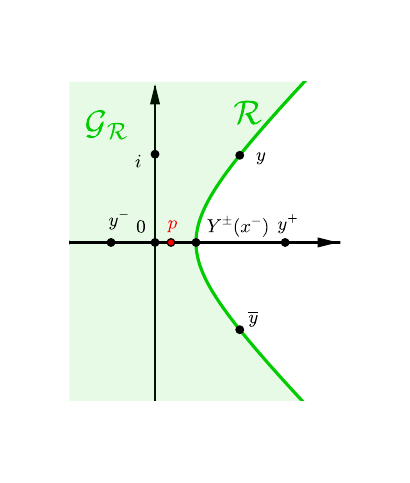}
  \includegraphics[scale=1.6]{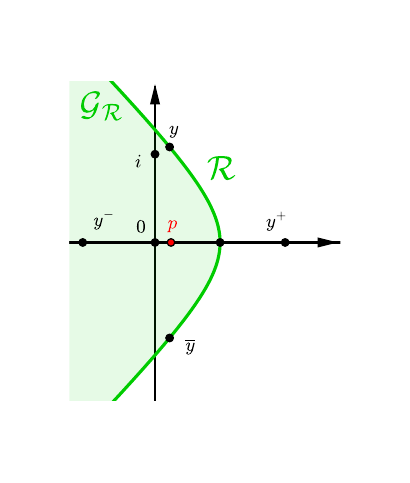}
  \vspace{-8mm}
  \caption{The curve $\R$, the domain $\G$ and the possible pole $p$ of $\phi_1$. The branch point $y^+$ (resp.\ $y^-$) lies outside (resp.\ inside) $\G$. On the left, $\beta>\pi/2$, while $\beta<\pi/2$ on the right.}
  \label{fig:bvp}
\end{figure}

The following lemma describes the links between the curve $\R$ and the parametrization~\eqref{eq:uniformization}.

\begin{lem}\label{lem:R-param}
  The curve $\R$ coincides with the set $\{ y(s) : s \in \RR_{-} \}$.  For $y=y(s)\in \R$, with $s\in \RR_-$, we have $\bar y=y(1/s)$. Moreover, $x(s)$ is the unique value $x\in (-\infty, x^-]$ such that $\{y, \bar y\}=\{Y^+(x),Y^-(x)\}$. Finally, $x(s)=X^-(y)=X^-(\bar y)$.
\end{lem}
\begin{proof}
  By definition of $\R$, the point $y$ lies on $\R$ if and only if there exists $x\in(-\infty,x^-]$ such that $\gamma(x,y)=0$, or equivalently such that $\{y, \bar y\}=\{Y^+(x),Y^-(x)\}$. As already observed in Section~\ref{subsec:kernel}, the map $s\mapsto x(s)$ sends $\RR_-$ surjectively to $(-\infty,x^-]$. Hence it is equivalent to say that there exists $s\in \RR_-$ such that $\{y, \bar y\}=\{Y^+(x(s)),Y^-(x(s))\}$, or equivalently that $\{y, \bar y\}= \{y(s), \bar y(1/s)\}$, by~\eqref{Ys}. This proves the first point of the lemma.

  Let us now take $y=y(s)\in \R$, with $s\in \RR_-$. The above argument shows that $\bar y=y(1/s)$. Moreover, $x(s)$ is one of the values $x\in (-\infty, x^-]$ such that $\gamma(x,y)=0$, or equivalently $\{y, \bar y\}=\{Y^+(x),Y^-(x)\}$. It remains to prove that any such value $x$ must be $X^-(y)$. For any such $x$, we have $\gamma(x,y)=0$, so that  $x=X^-(y)$ or $x=X^+(y)$.  That is to say, by~\eqref{Xs}, $x=x(s)$ or $x=x(e^{2i\beta}/s)$. Since $e^{2i\beta}\not \in \RR_+$, the value $x(e^{2i\beta}/s)$ cannot be in $(-\infty, x^-]$, hence $x=x(s)$. We still need to decide whether $x(s)$ is $X^-(y)$ or $X^+(y)$. As already mentioned, $\{X^-(y),X^+(y)\}=\{x(s),x(e^{2i\beta}/s)\}$. Moreover, it follows from~\eqref{eq:definition_Theta_pm} that $X^+(y)-X^-(y)$ is a square root, and hence has a non-negative real part. It thus suffices to show that $x(e^{2i\beta}/s)-x(s)$ has a positive real part to conclude.  By~\eqref{eq:uniformization},
  \[
    \Re\bigl(  x(e^{2i\beta}/s)-x(s)\bigr) =\frac{x^+-x^-}4 ( \cos(2\beta) -1) ( s+1/s) >0,
  \]
  which concludes the proof.
\end{proof}

\section{A boundary value problem -- Invariants}
\label{sec:invariants}

In this section  we introduce the notion of \emm invariant,, which is motivated by a boundary value problem satisfied by the function $\phi_1$, established in~\cite{franceschi_explicit_2017}. Recall the definitions of the curve $\R$ and  the domain $\G$ in the previous subsection.

\begin{prop}
  \label{prop:BVP_Carleman_sketch}
  The Laplace transform $\phi_1$ is meromorphic in 
  an open domain containing $\bG$. It  satisfies the boundary condition
  \beq 
  \label{eq:bound_cond_gen}
  \phi_1(\overline{y})=G(y)\phi_1({y}), \qquad\forall y\in \R,
  \eeq
  with
  \beq
  \label{def:G}
  G(y)=\frac{\gamma_1}{\gamma_2}( X^-(y),y)\frac{\gamma_2}{\gamma_1}( X^-(\overline{y}),\overline{y}),
  \eeq
  where $\gamma_1$ and $\gamma_2$ are the bivariate polynomials of~\eqref{eq:def_gamma_gamma1_gamma2}.
  
  The function $\phi_1$  has at most one  pole in $\bG$.
  This pole exists if and only if $2\beta -2\varepsilon -\theta\ge 0$. It is then  simple, and   its value  is
  \[
    p=y(s_1)=\frac{2r_{11}(\mu_1r_{21}-\mu_2r_{11})}{r_{11}^2\sigma_{22}-2r_{11}r_{21}\sigma_{12}+r_{21}^2\sigma_{11}}.
  \]
  The  pole coincides with  $Y^\pm(x^-)=y(-1)$ if and only if $2\beta -2\varepsilon -\theta= 0$, or equivalently $s_1=-1$.

  There exists a constant $\kappa \neq0$ such that, as $y\rightarrow \infty$ in $\G$, 

  \beq\label{phi-asympt}
  \phi_1(y)\underset{y\to \infty}{\sim} \kappa y^{\alpha-1},
  \eeq
  where $\alpha=(\delta+\varepsilon-\pi)/\beta$ is the key parameter introduced  in~\eqref{eq:def_alpha}.
\end{prop}

Note that   Condition~\eqref{eq:bound_cond_gen} is consistent with $G(\overline y)=1/G(y)$.

\begin{proof}
  The Laplace transform $\phi_1$ is clearly holomorphic on the domain $D_1:=\{y \in \C: \Re y<0\}$, with continuous limits on the boundary $i\RR$. Moreover, it is proved in~\cite[Lem.~5]{franceschi_explicit_2017} that for $y \in i\RR \cup \left(\bG\cap \{y : \Re y \ge 0\}\right)$, one has $\Re X^-(y)<0$. Hence this is true as well on a neighbourhood $D_2$ of this set, which we choose to be simply connected. Note that $D_1$ and $D_2$ intersect on some open set to the left of the line $i\RR$ (see Figure~\ref{fig:bvp}). Now on the simply connected domain $D_1\cup D_2$, which contains $\bG$ by construction, one can define $\phi_1$ meromorphically using
  \[
    \phi_1(y) = -\frac{\gamma_2(X^-(y),y)}{\gamma_1(X^-(y),y)}\phi_2(X^-(y))
  \]
  (see~\cite[Lem.~3]{franceschi_explicit_2017}).
  This proves the first statement of the proposition\footnote{We have reorganized slightly the arguments of~\cite{franceschi_explicit_2017} because the set defined in Lemma~3 of this reference is not open if $\beta <\pi/2$.}.
  
  The boundary identity~\eqref{eq:bound_cond_gen}  is established in~\cite[Prop.~6]{franceschi_explicit_2017}. In this proposition it is also stated that $\phi_1$ has a pole in $\bG$ if and only if $\gamma_1(x^-, Y^\pm (x^-))\geq 0$, and that this pole coincides with $Y^\pm(x^-)$ if and only if  $\gamma_1(x^-, Y^\pm (x^-))=0$. By Lemmas~\ref{lem:pole} and~\ref{lem:angles}, this gives the conditions stated in the proposition. The fact that the pole is simple comes again from~\cite[Prop.~6]{franceschi_explicit_2017}, and its value is given by~\cite[Eqs.~(18)--(19)]{franceschi_explicit_2017}. Finally, the asymptotic behaviour of $\phi_1$ comes from~\cite[Prop.~19]{franceschi_explicit_2017}.
\end{proof}

A boundary value problem like the one of Proposition~\ref{prop:BVP_Carleman_sketch} is said to be \emm homogeneous, if  the function $G$ occurring in~\eqref{eq:bound_cond_gen} is simply~$1$.  A solution is then called an \emm invariant,.

\begin{defn}
  \label{def:inv}
  A function $I$ from $\G$ to $\C$ is called  an \emm invariant, if it is meromorphic on a domain containing $\bG$ and satisfies the boundary (or invariant) condition
  \[
    I(y)=I(\overline{y}), \quad \forall y \in\mathcal{R}.
  \]
\end{defn}

There exists in the literature a stronger notion of invariant~\cite[Sec.~5]{BeBMRa-17}, where one requires  that $I(Y^+(x))=I(Y^-(x))$ for all $x$. This implies the  above boundary condition by taking $x\in (-\infty, x^-]$. 

In the following section we exhibit a simple, explicit and D-finite invariant $w$. In the next one, we show how to construct another invariant, this time involving the Laplace transform~$\phi_1$, provided one of the angle conditions~\eqref{eq:CNS0} or~\eqref{eq:CNS0double} holds.

\section{A canonical invariant}
\label{sec:gluing}

In this section we introduce a key invariant, denoted $w(y)$, and study its algebraic and differential properties. It is expressed in terms of an explicit hypergeometric function $T_a$, which generalizes the Chebyshev polynomial of the first kind, obtained for $a\in  \N_0$.

Below,  we use the words ``rational'', ``algebraic'', ``D-finite'' and ``D-algebraic'' without specifying whether we request the coefficients of the corresponding algebraic/differential equations to be  real or complex. The reason  is that  a function which, like $\phi_1(y)$, $T_a(y)$ or $w(y)$, is analytic in the neighborhood of a real segment, and takes real values on this segment, is, say, D-finite on $\RR(y)$ if and only if it is D-finite on $\C(y)$ (analogous statements hold for the other three classes of functions).

\subsection{A generalization of Chebyshev polynomials}
For $x\in\C\setminus (-\infty,-1]$ and $a\in \RR$, let us define
\[ 
  T_a(x)  ={_2F}_1 \Bigl(-a,a;\frac{1}{2};\frac{1-x}{2}\Bigr),
\] 
where ${_2F}_1$ is the classical Gauss hypergeometric function. In other words,  $T_a$ is the analytic continuation to $\C\setminus (-\infty,-1]$ of the following series, which converges  for $|1-x|<2$:
\[
  T_a(x)= \sum_{n\ge 0} \frac a{a+n} {a+n \choose 2n} 2^n (x-1)^n.
\]
When $a=m \in \N_0$,  the above sum ranges from $n=0$ to $a$, and $T_m$ is the classical Chebyshev polynomial. The function $T_a$ is D-finite for all values of~$a$, as the hypergeometric function itself. It satisfies the following differential equation:
\beq\label{de-Ta}
(1-x^2) T''_a(x)-xT_a'(x)+a^2T_a(x)=0.
\eeq
Other  useful expressions are
\beq\label{Ta-trig}
T_a(x)=   \cos \left(a \arccos x \right),
\eeq
(see~\cite[15.1.17]{abramowitz_stegun}) which is valid for $x$ in $\C\setminus ( (-\infty,-1] \cup [1, \infty))$,
and
\beq\label{Ta-alg}
T_a(x)=\frac{1}{2}
\Bigl(\bigl(x+\sqrt{x^2-1}\bigr)^a+\bigl(x-\sqrt{x^2-1}\bigr)^a\Bigr),
\eeq
(see~\cite[15.1.11]{abramowitz_stegun}) which is valid for $x$ in $\C\setminus  (-\infty,-1]$ (here we take $\sqrt{re^{it}}:= \sqrt r e^{it/2}$ for $t \in (-\pi, \pi]$ and $u^a: =\exp(a \log u)$ with the principal value of the logarithm on $\C\setminus \RR_-$). The latter expression shows that $T_a$ is algebraic when $a\in \Q$.   The Schwarz list~\cite{schwarz_1873} implies that $T_a$ is transcendental otherwise. This can be also proved by looking at the growth of $T_a(x)$ at infinity.

We will also use the fact  that both ${1-T_a}$ and ${1+T_a}$ are squares of D-finite functions. In fact, it follows from~\eqref{Ta-trig} and elementary trigonometry that 
\beq \label{sqrt_plus}
\sqrt{1+T_a}= \sqrt 2 \  T_{a/2},
\eeq
with the same domain of definition as $T_a$. Analogously, 
$\sqrt{1-T_a(x)}= \sqrt 2  \sin \left(\frac a 2 \arccos x \right)$, and this function  satisfies the same differential equation as $T_{a/2}$. Moreover,
\beq \label{sqrt_minus}
\frac 1 a   \sqrt{\frac{1-T_a(x)}{1-x}}=  {_2F}_1 \Bigl(\frac{-a+1}2,\frac{a+1}2;\frac{3}{2};\frac{1-x}{2}\Bigr)
\eeq
and this function is analytic on the same domain as $T_a$.   

Our final property deals with rational functions in $T_a$.

\begin{prop}\label{prop:rational_Ta}
  Let $S/R$ be an irreducible fraction with coefficients in $\C$. Then $(S/R)(T_a)$ is D-finite if and only if either $a \in \Q$ or $R$ is constant. 
\end{prop}

\begin{proof}
  If $a\in\mathbb Q$, we have seen that $T_a$ is algebraic, and then so is any fraction in $T_a$. If $R$ is a constant, then $(S/R)(T_a)$ is D-finite because D-finite functions form a ring.

  We now assume that $(S/R)(T_a)$ is D-finite and want to prove that either $a\in\mathbb Q$ or $R$ is a constant.

  Let us first suppose that $S$ is a constant. In this case,  both $R(T_a)$ and $1/R(T_a)$ are D-finite. By a result of Harris and Sibuya~\cite{HarrisSibuya}, the function $(R(T_a))'/R(T_a)$ is  algebraic.
  Assume that~$R$ is non-constant, and let us prove that $a\in \Q$. Since~$R$ is non-constant, there exist $\kappa\neq 0$,  $z_1,\ldots,z_n\in\mathbb C$ and positive integers $n,m_1,\ldots,m_n$ such that
  \begin{equation*}
    R(z) = \kappa\prod_{i=1}^n(z-z_i)^{m_i}.
  \end{equation*}
  With this notation, one has
  \begin{equation}
    \label{eq:to_square}
    \frac{(R(T_a(x)))'}{R(T_a(x))}=T_a'(x)\sum_{i=1}^n\frac{m_i}{T_a(x)-z_i}.
  \end{equation}
  In addition to the  second order   linear differential equation~\eqref{de-Ta}, the function $T_a(x)=\cos(a \arccos x)$ satisfies a first order non-linear differential equation:
  \begin{equation*}
    T_a'(x)^2={a^2}\frac{1-T_a(x)^2}{{1-x^2}}.
  \end{equation*}
  Since~\eqref{eq:to_square} is algebraic, we conclude that the function
  \begin{equation*}
    (1-T_a^2)\left(\sum_{i=1}^n\frac{m_i}{T_a-z_i}\right)^2
  \end{equation*}
  is also algebraic. 
  This function is a non-trivial fraction in $T_a$ (because of the multiple poles  in the denominator). This implies that $T_a$ is algebraic, so that $a \in \Q$.

  Let us now   consider the case of a non-constant polynomial $S$. The fraction $(S/R)(z)$ being irreducible in $\mathbb C(z)$, it comes from the classical B\'ezout theorem that there exist two polynomials~$U$ and $V$ such that $US+VR=1$. Dividing   by $R$, it follows that $U\frac{S}{R}+V=\frac{1}{R}$. This shows that if $(S/R)(T_a)$ is D-finite, then $1/R(T_a)$ should also be D-finite. We have just seen that this implies that $a\in \Q$ or $R$ is a constant.
\end{proof}

\subsection{The  invariant $\boldsymbol w$}
We now define a function $w$, which is analytic on $\C\setminus [y^+,\infty)$, by:
\beq
\label{eq:definition_w-bis}
w(y):=T_{\frac{\pi}{\beta}}\left(-\frac{2y-(y^+ +y^-)}{y^+ -y^-}\right).
\eeq
Using the normal variable $\yn$ of~\eqref{xy:normal} and the values $\yn^\pm$ in~\eqref{xyn-pm}, 
\[
  w(y)= T_{\frac{\pi}{\beta}}\left(-\yn +\cos(\beta-\theta)\right).
\]    
Note that when $y=y(s)$ is given by the parametrization~\eqref{eq:uniformization}, or equivalently $\yn=\yn(s)$ by~\eqref{param:normal}, the argument simplifies into
\[
  -\frac 1 2 \left(  {e^{-i\beta}} s+ {e^{i\beta}}/ s\right).
\]
In particular, we derive from~\eqref{Ta-alg} that for $s\in \C\setminus e^{i\beta} \RR_+$,
\beq\label{wys0}
w(y(s))=\frac 1 2 \left( \left( -\frac s{e^{i\beta}}\right)^{\pi/\beta} +\left( -\frac s {e^{i\beta}}\right)^{-\pi/\beta}\right),
\eeq  
if we define the $a$-th power on $\C\setminus \RR_-$ as before (with $a= \pi/\beta$). This can be rewritten as:
\beq\label{wys}
w(y(s))= -\frac{1}{2} \left((-s)^{\pi/\beta} + (-s)^{-\pi/\beta}\right) ,
\eeq
provided we now take the principal value of the logarithm on $\C\setminus e^{i\beta} \RR_-$.
In particular, $w(y(s))$ is real when $|s|=1$.

\smallskip
The function $w$ inherits the algebraic and differential properties of $T_{\pi/\beta}$.

\begin{prop}
  \label{prop:algebraic_nature_w}
  If $\frac{\pi}{\beta}\in\Z$, the function $w$ is a polynomial.
  If $\frac{\pi}{\beta}\in\Q\setminus \Z$, the function $w$ is  algebraic but irrational.
  If $\frac{\pi}{\beta}\notin\Q$, the function $w$ is D-finite  but not algebraic.

  The functions $\sqrt{1-w}$ and $\sqrt{1+w}$ are D-finite.

  Finally, a rational fraction in $w$, say $(S/R)(w)$, is D-finite if and only if either $\beta/\pi \in \Q$ or $S/R$ is a polynomial.
\end{prop}

\medskip

We will express $\phi_1$ in terms of $w$, using the fact that $w$ is, in a certain sense, a \emm canonical invariant,. The following lemma proves that it is an invariant (in the sense of Definition~\ref{def:inv}), and its canonical properties are described in  Proposition~\ref{prop:inv}.

\begin{lem} \label{lem:w}
  The function $w$  is an invariant in the sense of Definition~\ref{def:inv}. More precisely, it satisfies the following properties:
  \begin{enumerate}[label={\rm(\arabic*)},ref={\rm(\arabic*)}] 
  \item it is analytic in     an open domain containing $\bG$, namely $\C\setminus  [y^+, \infty)$,
  \item it goes to infinity at infinity, with 
    \[
      w(y) \underset{y\to\infty}{\sim}\kappa y^{\frac{\pi}{\beta}}
    \]
    for some constant $\kappa \neq 0$,
  \item it is bijective from $\G$ to $\C\setminus (-\infty, -1]$,
  \item it satisfies the boundary condition
    \[
      w(y)=w(\overline{y}), \quad \forall y\in\R,
    \]
  \item it is $2$-to-$1$ from $\mathcal{R}\setminus\{Y^\pm(x^-)\}$ to $(-\infty,-1)$,
  \item\label{it6lem:w} around the point $Y^\pm(x^-)=y(-1)$, we have
    \[
      w\left(Y^\pm(x^-)-y \right) =-1+  \frac{2 \pi^2}{\beta^2 (y^+-y^-)^2\sin^2 \beta}\,  {y^2}   + O(y^3)
    \]
    as $y\rightarrow 0$.
  \end{enumerate}
  We will denote by $w^{-1}$ the analytic function from $\C\setminus(-\infty, -1]$ to $\G$ that maps a complex number to its unique preimage lying in $\G$.
\end{lem}
\begin{proof}
  The first two points  follow from known properties of Gauss' hypergeometric function. The next two can be found in~\cite[Lem.~3.4]{franceschi_tuttes_2016}.

  The fifth point follows from~\eqref{wys}. Indeed, assume that $w(y_1)=w(y_2)$ with $y_1$ and $y_2$ in $\R$. By Lemma~\ref{lem:R-param}, there exist $t_1$ and $t_2$ in $\RR_-$ such that $y_i=y(t_i)$. By~\eqref{wys}, and the fact that $(-t_i)$ is a positive real number, we conclude that either $t_1=t_2$, or $t_1=1/t_2$. In the former case, $y_1=y_2$. In the latter one, we have $x:=x(t_1)=x(t_2)\in (-\infty, x^-]$, hence $y_1$ and $y_2$ are the two conjugate solutions of $\gamma(x,y)=0$.

  For the last point,  we first recall that $Y^\pm(x^-)=y(-1)$, so  that, by definition of the parametrization~\eqref{eq:uniformization},
  \[
    -\frac{2Y^\pm(x^-)-(y^+ +y^-)}{y^+ -y^-}=\cos \beta.
  \]
  In particular, $w(Y^\pm(x^-)) = T_{\pi/\beta}(\cos \beta)= \cos( \pi)=-1$ by~\eqref{Ta-trig}. More generally, by differentiating~\eqref{Ta-trig} twice, we obtain:
  \[
    T'_{\pi/\beta}(\cos \beta)= 0, \qquad    T''_{\pi/\beta}(\cos \beta)= \frac{\pi^2}{\beta^2\sin ^2 \beta} ,
  \]
  and the final property follows.
\end{proof}

Let us now explain in what sense the invariant $w$ is canonical.

\begin{prop}
  \label{prop:inv}
  Assume that $I$ is an invariant which has a finite number of poles in $\bG$ and grows at most  polynomially at infinity.
  Then there exist polynomials $R$ and $S$ in $\C[z]$ such that
  \[ 
    I=\frac{S \circ w}{R \circ w}.
  \] 
  Moreover,  an invariant with no pole in $\bG$ and a finite limit at infinity is  constant.
\end{prop}

\begin{proof}
  Let us consider the function $I\circ w^{-1}$, {where $w^{-1}$ is the analytic function  of  Lemma~\ref{lem:w}.} By composition, $I\circ w^{-1}$ is meromorphic on $\C\setminus (-\infty,-1]$.  Let us now take $z \in (-\infty, -1]$, and define $I\circ w^{-1}(z):=I(y)=I(\bar y)$, where $y$ and $\bar y$ are the two values of $\R$ such that $w(y)=z$. Hence $I\circ w^{-1}$ is now defined on $\C$. By Morera's theorem, $I\circ w^{-1}$ is analytic at $z \in (-\infty, -1]$, unless $y$ (and $\bar y$) is a pole of $I$. If $y_0\in \R$ is one of the  poles of $I$ then $z_0=w(y_0)$ is an isolated singularity of $I\circ w^{-1}$. Let  $\ell\in \N$ be such that $(w(y)-w(y_0))^\ell I(y)$ tends to $0$ as $y$ tends to $y_0$ or $\bar y_0$ in $\G$. Such an $\ell$ exists since $I$ is meromorphic and $w$ analytic in neighborhoods of $y_0$ and $\bar y_0$. If a sequence $(z_n)_n$ tends to $z_0$ in $\C\setminus(-\infty, -1]$, then for $n$ large enough,  $w^{-1}(z_n)$ is arbitrarily close to $y_0$ \emm or, arbitrarily close to $\bar y_0$ (since these are the only two pre-images of $z_0$ by $w$). Hence $(z_n-z_0)^\ell( I\circ w^{-1})(z_n)$ tends to $0$. This proves that $z_0$ is a pole of $I\circ w^{-1}$.

  The  function $I\circ w^{-1}$ is thus meromorphic on $\C$. Each of its poles is the image by $w$ of a pole of $I$ lying in $\bG$. Hence  $I\circ w^{-1}$ has finitely many poles, and can be written as $S/R$, where $R$ is a polynomial and $S$ is entire. Finally, since $w^{-1}$ has polynomial growth at infinity by Lemma~\ref{lem:w}, the same holds for  $S=R \cdot( I\circ w^{-1})$.  This implies,  by a standard  extension of Liouville's theorem, that $S$ is a polynomial.

  Finally, assume that $I$ has no poles in $\bG$ and a finite limit at infinity. The above paragraph shows that we can take $R$ to be a constant.
  By Lemma~\ref{lem:w}, $I(y)= S( w(y))$ grows as $y^{\mathrm{deg}(S)\frac{\pi}{\beta}}$ at infinity. Hence  $S$ has degree zero and is a constant.
\end{proof}

\section{Decoupling functions, and a second invariant}
\label{sec:existence-decoupling}

Let us now return to the inhomogeneous problem of Proposition~\ref{prop:BVP_Carleman_sketch}. A natural idea to transform it into an homogeneous one is to observe that the function $G$ in~\eqref{def:G} may be written as a ratio
\beq
\label{eq:G->F}
G(y) = \frac{F_0(y)}{F_0(\overline{y})},
\eeq
with $F_0(y)=\frac{\gamma_1}{\gamma_2}( X^-(y),y)$. Then the boundary condition~\eqref{eq:bound_cond_gen} rewrites as:
\[
  (F_0\cdot\phi_1)({y})=(F_0\cdot\phi_1)(\overline{y}), \quad \forall y \in\mathcal{R}.
\]
However, the function $F_0$ inherits, in general, the cut of $X^-$  on the half-line $(-\infty,y^-]$, which is contained in $\G$. Hence the function $F_0 \phi_1$ is \emm not,  an invariant in the sense of Definition~\ref{def:inv}, as it is \emm a priori, not meromorphic in a neighbourhood of $\bG$.

Solving the boundary value problem of Proposition~\ref{prop:BVP_Carleman_sketch} in full generality is  the main contribution of~\cite{franceschi_explicit_2017}, where an explicit expression for $\phi_1$ is obtained in terms of contour integrals.
However, these integrals are complicated, and do not give a handle to understand the exceptional parameters for which substantial simplifications may occur.
Our point in the present paper is different: we want  to characterize the cases for which the boundary condition~\eqref{eq:bound_cond_gen} may be transformed into  an homogeneous one, which is then easy to solve in terms of the   canonical  invariant of Section~\ref{sec:gluing}. This transformation relies on the notion of \emm decoupling functions,.

\subsection{Decoupling functions}
\label{sec:decoupling}

\begin{defn}
  \label{def:decoupling_function}
  Let $m$ be a positive integer. A quadrant model with parameters $\mu$, $\Sigma$ and $R$ is \emm $m$-decoupled, if there exist rational functions $F$ and $L$ such that 
  \[
    \left(\frac{\gamma_1}{\gamma_2}\right)^{\! m}(x,y) = \frac{F(y)}{L(x)}
  \] 
  whenever $\gamma(x,y)=0$. By this, we mean that  the following equivalent identities between algebraic functions hold:
  \beq\label{dec-Y}
  \left(\frac{\gamma_1}{\gamma_2}\right)^{\! m}(x,Y^+(x)) = \frac{F(Y^+(x))}{L(x)} , \qquad
  \left(\frac{\gamma_1}{\gamma_2}\right)^{\! m}(x,Y^-(x)) = \frac{F(Y^-(x))}{L(x)} .
  \eeq
  The functions $F(y)$ and $L(x)$ are then said to form a \emm decoupling pair, for the model, and more precisely an \emm $m$-decoupling pair,.
\end{defn}               

A few remarks are in order:
\begin{itemize}
\item First, the two identities of~\eqref{dec-Y} are equivalent because any rational relation between~$x$ and $Y^-(x)$ must hold as well with $x$ and $Y^+(x)$, by irreducibility of the quadratic polynomial $\gamma(x,y)$ (recall that $Y^\pm(x)$ are the two roots of this polynomial).
\item As will be seen in Theorem~\ref{thm:decoupling}, there may exist several decoupling pairs. This happens in particular when $\beta/\pi $ is rational.
\item In the enumeration of discrete walks confined to the
  quadrant~\cite{BeBMRa-16,BeBMRa-17}, decoupling pairs are defined as solving
  the equation $xy = F(y)+L(x)$ on a certain curve. In contrast, we have
  here a multiplicative  version of this notion.
\end{itemize}

We now relate the decoupling property to factorizations of $G(y)$ (or more generally $G^m(y)$) of the form~\eqref{eq:G->F}.  

\begin{lem}\label{lem:decoupling_function}
  A model is $m$-decoupled if and only if  the following equivalent assertions hold:
  \begin{itemize}
  \item There exists a rational function $F$ such that the following identity between algebraic function holds:
    \beq
    \label{eq:decoupling_function}
    \frac{\left(\frac{\gamma_1}{\gamma_2}\right)^{\! m}(x, Y^-(x))}{\left(\frac{\gamma_1}{\gamma_2}\right)^{\! m}(x, Y^+(x))}=\frac{F( Y^-(x))}{F( Y^+(x))}.
    \eeq
  \item There exists a rational function $F$ such that for all $y\in\R$, 
    \beq
    \label{eq:G->F_rational}
    G^{m}(y) = \frac{F(y)}{F(\overline{y})}.
    \eeq
  \end{itemize}

  Moreover, any rational function $F$ satisfying~\eqref{eq:decoupling_function} satisfies~\eqref{eq:G->F_rational}, and vice-versa.
  We call $F$ an \emm $m$-decoupling function,, and define 
  \[ 
    L(x) = \frac{F( Y^-(x))}{\left(\frac{\gamma_1}{\gamma_2}\right)^{\! m}(x, Y^-(x))}.
  \] 
  Then $L$ is a rational function in $x$ and $(F,L)$ is an $m$-decoupling pair in the sense of Definition~\ref{def:decoupling_function}.
\end{lem}

Note that Condition~\eqref{eq:decoupling_function} is left  unchanged upon exchanging $Y^+(x)$ and $Y^-(x)$.

\begin{proof}
  Let us first assume that the model is $m$-decoupled. We then obtain~\eqref{eq:decoupling_function} by taking the ratio of the two identities in~\eqref{dec-Y}.
  
  Now assume that  $F$ satisfies~\eqref{eq:decoupling_function}. Let $y\in \mathcal R$, and let $x=X^-(y)=X^-(\bar y)$ be the unique real number in $(-\infty, x^-]$ such that $\{y, \bar y\} =\{Y^+(x), Y^-(x)\}$ (Lemma~\ref{lem:R-param}). Writing~\eqref{eq:decoupling_function} for this pair $(x,y)$ gives~\eqref{eq:G->F_rational}, by definition~\eqref{def:G} of $G(y)$.

  Now assume that~\eqref{eq:G->F_rational} holds. As we have just observed, this means that~\eqref{eq:decoupling_function} holds for $x\in (-\infty, x^-]$. Since $Y^\pm(x)$ are the roots of a quadratic polynomial over $\RR(x)$, there exist rational functions $L(x)$ and $M(x)$ such that
  \[
    \frac{F( Y^\pm (x))}{\left(\frac{\gamma_1}{\gamma_2}\right)^{\! m}(x, Y^\pm(x))}
    = L(x) +M(x) Y^\pm(x) .
  \]
  Specializing this to $x\in (-\infty, x^-)$
  (using~\eqref{eq:decoupling_function}) shows that $M(x)=0$ on this half-line, and thus everywhere since $M$ is rational. Hence~\eqref{dec-Y} holds, and the model is decoupled with $(F,L)$ as a decoupling pair.
\end{proof}

The following simple observation underlines that decoupling functions yield invariants.

\begin{lem}
  \label{lem:Fphi1-invariant}
  If $F$ an $m$-decoupling function, then the product function $F\phi_1^m$ is an invariant in the sense of Definition~\ref{def:inv}.
\end{lem}

\begin{proof}
  This follows directly from~\eqref{eq:G->F_rational} and Proposition~\ref{prop:BVP_Carleman_sketch}.
\end{proof}

\subsection{The rational function $E(s)$}
In  Section~\ref{subsec:explicit-dec}, we will prove that decoupling functions exist if and only if one of the angle conditions~\eqref{eq:CNS0} or~\eqref{eq:CNS0double} holds (Theorem~\ref{thm:decoupling}). One important tool is   the rational parametrization~\eqref{eq:uniformization} of the kernel. In this subsection, we study the function $G(y(s))$.

Let us return to the function $G$ given by~\eqref{def:G}. By the definition of $s_1$ and $s_2$ given above Lemma~\ref{lem:angles}, there exist constants $c_1$ and $c_2$ such that
\[ 
  s\gamma_{1}(x(s),y(s))
  =c_1 (s-s_1)(s-s_0)    \quad\text{and}\quad    s \gamma_{2}(x(s),y(s))     
  =c_2(s-s_2)(s-s_0).
\] 
Let us introduce the following rational function:
\begin{align}
  E(s) 
  &=\frac{\gamma_1}{\gamma_2}(x(s),y(s))
    \frac{\gamma_2}{\gamma_1}(x(1/s),y(1/s)) \label{eq:definition_E_V0}
  \\
  &=\frac{s_2}{s_1}\frac{(s-s_1)(s-\frac{1}{s_2})}{(s-s_2)(s-\frac{1}{s_1})}. \label{eq:definition_E}
\end{align}

\begin{lem}\label{lem:GE}
  For $s\in (-\infty,0)$, we have
  \[ 
    G(y(s))=E(s),
  \] 
  where $E(s)$ is the above rational function.
\end{lem}

\begin{proof} We start from the    definition~\eqref{def:G} of $G(y)$, with  $y=y(s)$, and apply Lemma~\ref{lem:R-param},  as well as  $x(s)=x(1/s)$. We thus obtain the lemma, with the expression~\eqref{eq:definition_E_V0} of $E(s)$.
\end{proof}

For ease of notation, we denote by $q$ the complex number 
\[ 
  q=e^{2i\beta}.
\] 
Note that the condition $\beta/\pi \in \Q$ translates   into the fact that $q$ is a root of unity.

Now assume that the model is $m$-decoupled, and take $s \in (-\infty,0)$. By Lemma~\ref{lem:GE}, $y(s)$ and $y(1/s)= \overline{y(s)}$ lie in $\R$. Hence by~\eqref{eq:G->F_rational}, there exists a rational function $F$ such that
\[
  G^m(y(s))=\frac{F(y(s))}{F(y(1/s))}.
\]
But by Lemma~\ref{lem:GE}, this is also $E^m(s)$. 
Hence  the rational fractions $E^m(s)$ and $\frac{F(y(s))}{F(y(1/s))}$, which coincide on $(-\infty,0)$ must be equal, which gives
\begin{align}
  E^m(s)& =\frac{F(y(s))}{F(y(1/s))} =\frac{F(y(s))}{F(y(qs))}, \qquad \text{ since } y(1/s)=y(qs), \label{H-dec-strong}\\
        &= \frac{H(s)}{H(qs)}, \hskip 10mm \quad \text{with } \quad H(s)=F(y(s)).\nonumber
\end{align}
It is thus natural to ask when the rational function $E^m$ can be written in the form $ \frac{H(s)}{H(qs)}$. This is answered by the following elementary lemma, which shows how Conditions~\eqref{eq:CNS0} and~\eqref{eq:CNS0double} naturally arise.

\begin{lem}\label{lem:Ecocyclecaracterization-v0}
  Let $E(s)$ be the rational function given by~\eqref{eq:definition_E}:
  \[
    E(s)=\frac{s_2}{s_1}\frac{(s-s_1)(s-\frac{1}{s_2})}{(s-s_2)(s-\frac{1}{s_1})}.
  \]
  The following statements are equivalent:
  \begin{enumerate}[label={\rm\roman*)},ref={\rm\roman*)}] 
  \item \label{item:i-eco} there exist $m \in \N $    and $H \in \C(s)^*$ such that $E^m(s)=\frac{H(s)}{H(qs)}$, 
  \item \label{item:ii-eco} the \emm elliptic divisor, of $E$ relative to $q$ (defined in Definition~\ref{defn:elliptic_divisor}) is zero,
  \item \label{item:iii-eco} one of the conditions~\eqref{eq:CNS0} and~\eqref{eq:CNS0double} holds; that is,
    \beq\label{conds-combined}
    \frac{s_1}{s_2} \in q^\Z \quad\text{ or }\quad \left( s_1^2 \in q^\Z \text{ and }  s_2^2 \in q^\Z \right).
    \eeq
  \end{enumerate}
  Moreover, one can take $m=1$ when~\eqref{eq:CNS0} holds, and $m=2$ when~\eqref{eq:CNS0double} holds.
\end{lem} 

This  lemma is proved in Appendix~\ref{sec:standardform}.

\medskip
\noindent{\bf Remarks}
\\
{\bf 1.} The two conditions of~\eqref{conds-combined} are just a convenient reformulation of Conditions~\eqref{eq:CNS0} and~\eqref{eq:CNS0double}, respectively. They directly follow from the values of $s_1$ and $s_2$ given in Lemma~\ref{lem:angles}.
\\
{\bf 2.}  The reader should not worry about the terminology \emm elliptic divisor,, as  Condition~\ref{item:iii-eco}
is a straightforward translation of Condition~\ref{item:ii-eco}. 
\\
{\bf 3.}  In Theorem~\ref{thm:decoupling} below, we construct $H(s)$ explicitly (in the form $F(y(s))$),
assuming that~\eqref{eq:CNS0} or~\eqref{eq:CNS0double} holds; see for instance~\eqref{H-explicit}.

\subsection{Explicit decoupling functions}
\label{subsec:explicit-dec}
We can now establish the equivalence between  the simple and double angle    conditions and the existence of decoupling functions, and provide explicit decoupling functions.

For $r\in \Z$ and $\sigma \in \C$, let us define a rational function $   F_{r,\sigma}$ by:
\beq\label{Fr-def}
F_{r,\sigma}(y) =
\begin{cases}
  \displaystyle    P_{r,\sigma}(y) =   \prod_{j=0} ^{r-1} (y- y(\sigma q^{-j})) & \text{ if } r\ge 0, \\
  \displaystyle    \frac 1 {Q_{r,\sigma}(y)} =    \prod_{j=1}^{|r|} \frac 1{\left( y-y(\sigma q^j)\right) } & \text{ if } r<0.
\end{cases}
\eeq
By convention, the empty product, obtained in the first line when $r=0$, is $1$.
We further define  $P_{r,\sigma}(y) = 1$ when $r<0$ and $Q_{r,\sigma}(y)=1$ when $r\ge 0$, so that we can write in full generality
\[
  F_{r,\sigma}=\frac{P_{r,\sigma}}{Q_{r,\sigma}}.
\]
Moreover, we note that for any $r$, we have $P_{r,\sigma q^r}= Q_{-r,\sigma}$ and $Q_{r,\sigma q^r}= P_{-r,\sigma}$, so that
\beq\label{F-sym}
F_{r,\sigma q^r}= \frac 1 {F_{-r,\sigma}}.
\eeq

For $e$ and $\eps$ in  $\{0, 1\}$, let us define the polynomial   $f_{e,\eps}$ by
\beq\label{f_e_eps-def}
f_{e,\eps} (y)=
\begin{cases}
  1 & \text{ if  } e=0,\\
  y-y\left((-1)^\eps \sqrt q\right ) & \text { if } e=1, 
\end{cases}
\eeq
with $\sqrt q =e^{i\beta}$. We hope that no confusion will arise between the integer $\eps\in \{0,1\}$ and the reflection angle $\vareps$. Equivalently, returning to the definition~\eqref{eq:uniformization} of $y(s)$:
\[
  f_{e,\eps} (y)=
  \begin{cases}
    1 & \text{ if  } e=0,\\
    y-y^+ & \text { if } e=1 \text{ and } \eps=0,\\
    y-y^- & \text { if } e=1 \text{ and } \eps=1 .
  \end{cases}
\]

When Condition~\eqref{eq:CNS0} holds, or equivalently, $s_1/s_2 \in q^\Z$ (see~\eqref{conds-combined}),  we choose $r\in \Z$ such that
\beq\label{r-def}
s_1/s_2=q^r.
\eeq
Analogously, when  Condition~\eqref{eq:CNS0double} holds, we choose integers $r_1$ and $ r_2$, and numbers $e_1, e_2, \eps_1, \eps_2$ in $\{0,1\}$, such that
\beq\label{ri-def}
s_1=(-1)^{\eps_1} {\sqrt q}^{\,e_1} q^{r_1} \quad \text{ and } \quad
s_2=(-1)^{\eps_2} {\sqrt q}^{\,e_2} q^{r_2}.
\eeq

The following theorem relates the decoupling property to the conditions satisfied by $E(s)$ in Lemma~\ref{lem:Ecocyclecaracterization-v0}.

\begin{thm}
  \label{thm:decoupling}
  There exists an $m$-decoupling pair   in the sense of Definition~\ref{def:decoupling_function} if and only if any of the following equivalent conditions  holds: 
  \begin{enumerate}[label={\rm(\arabic{*})},ref={\rm(\arabic{*})}]
  \item\label{it:xi_inv}
    there exists a     rational function $F$ such that $E^m(s) =\frac{F(y(s))}{F(y(1/s))}$,
  \item\label{it:delta_inv} 
    there exists a     rational function $H$ such that $E^m(s)=\frac{H(s)}{H(qs)}$,
  \item\label{it:angle_condition} 
    one of the angle conditions~\eqref{eq:CNS0} or~\eqref{eq:CNS0double} holds.
  \end{enumerate} 
  If these conditions hold,  the function $F$ of Assertion~\ref{it:xi_inv} is an   $m$-decoupling function in the sense  of   Lemma~\ref{lem:decoupling_function}.

  Moreover, when~\eqref{eq:CNS0} holds and $r$ satisfies~\eqref{r-def}, we can take
  \beq\label{first-choice}
  m=1 \quad \text{ and }  \quad F=  F_{r, s_1}, 
  \eeq
  where $F_{r, \sigma}$ is defined by~\eqref{Fr-def}.
  When~\eqref{eq:CNS0double} holds and the $r_i$, $\eps_i$, and $e_i$ satisfy~\eqref{ri-def}, we can take
  \beq\label{second-choice}
  m=2 \quad \text{ and }  \quad F=\left( \frac{F_{r_1, s_1}}{F_{r_2, s_2}}\right)^2 \cdot \frac{f_{e_1, \eps_1}}{f_{e_2, \eps_2}}.
  \eeq
\end{thm}

Notice that the roots and poles of these decoupling functions  are 
real, since $q$, $s_1$  and $s_2$ have modulus $1$ and $s\mapsto y(s)$ sends the unit circle on the real line (see Remark~\ref{lem:realpointsmathcalS}). In particular, the only possible root or pole of these decoupling functions lying in $\R$ is $y(-1)$, the only point of $\R \cap \RR$.

Also, since all roots and poles of the function $F(y)$ are of the form $y(\sigma)$, for some complex number $\sigma$, the function $H(s)$ involved in Assertion~\ref{it:delta_inv}, which coincides with $F(y(s))$, has  explicit roots and poles. For instance, when $m=1$ and $r>0$, we have, up to a multiplicative factor,
\beq\label{H-explicit}
H(s)= \prod_{j=0}^{r-1} \left(q^j s-s_1\right)\left( s_1 -\frac {q^{j+1}} s\right).
\eeq

\begin{proof}[Proof of Theorem~\ref{thm:decoupling}]
  We have already explained in the previous subsection that  if the model is $m$-decoupled in the sense of Definition~\ref{def:decoupling_function}, then Assertion~\ref{it:xi_inv} holds (see~\eqref{H-dec-strong}). 
  Conversely, if this assertion   holds,  we can work out the same argument backwards to conclude that~\eqref{eq:G->F_rational} holds (because
  $\R= y( (-\infty,0))$ by Lemma~\ref{lem:R-param}), so that the model is decoupled by the second point of Lemma~\ref{lem:decoupling_function}.

  \medskip We now focus on the assertions~\ref{it:xi_inv},~\ref{it:delta_inv},~\ref{it:angle_condition}, and prove  that~\ref{it:xi_inv}$\Rightarrow$\ref{it:delta_inv}$\Rightarrow$\ref{it:angle_condition}$\Rightarrow$\ref{it:xi_inv}.
  
  Assume that~\ref{it:xi_inv} holds, and define  $H(s)=F(y(s))$. Then~\ref{it:delta_inv} follows from the fact that $y({1}/{s})=y(qs)$.

  Now assume that~\ref{it:delta_inv} holds.
  By Lemma~\ref{lem:Ecocyclecaracterization-v0},  one of the angle conditions~\eqref{eq:CNS0} or~\eqref{eq:CNS0double} holds. This gives~\ref{it:angle_condition}.
  
  It remains to check that if one of the angle conditions holds, the functions $F$ given by~\eqref{first-choice} and~\eqref{second-choice} are indeed $m$-decoupling functions, for $m=1$ and $m=2$ respectively. We first observe that $y(s)-y(\sigma q^{-j})$ is a Laurent polynomial in $s$, of degree $1$ and valuation $-1$, that vanishes for $s= \sigma q^{-j}$ and $s= q^{j+1}/\sigma$ (because $y(s)=y(q/s)$). Hence there exists a constant $\kappa$ (depending on $\sigma$ and $j$) such that
  \[
    y(s)-y(\sigma q^{-j}) = \kappa \left( sq^j -\sigma\right) \left( \frac{q^{j+1}} s -\sigma\right).
  \]
  It follows that, for any $r\ge 0$,
  \beq\label{F-ratio}
  \frac{F_{r,\sigma}(y(s))}{F_{r,\sigma}(y(sq))}=\prod_{j=0}^{r-1} \frac{(sq^j-\sigma) \left( \frac{q^{j+1}}s-\sigma\right)}
  {(sq^{j+1}-\sigma) \left( \frac{q^{j}}s-\sigma\right)}=
  \frac 1 {q^r} \cdot \frac{s-\sigma}{s-1/\sigma} \cdot \frac { s -q^r/\sigma}{ s-\sigma/q^r}.
  \eeq
  A similar calculation, or more directly the identity~\eqref{F-sym}, proves that this still holds for $r<0$.
  Given that
  \[
    E(s)=\frac{s_2}{s_1}\frac{(s-s_1)(s-\frac{1}{s_2})}{(s-s_2)(s-\frac{1}{s_1})},
  \]
  this already proves that~\eqref{first-choice} gives a 1-decoupling function when $s_1=q^r s_2$.

  Furthermore, if $\sigma=(-1)^\eps {\sqrt q }^{\,e} q^r$, with $r\in \Z$ and $e, \eps \in \{0,1\}$,  we derive from~\eqref{F-ratio} that 
  \beq\label{F-ratio-double}
  \frac{F_{r,\sigma}(y(s))}{F_{r,\sigma}(y(sq))}=
  \frac 1 {\sigma} \cdot \frac{s-\sigma}{s-1/\sigma} \cdot
  \frac {(-1)^\eps {\sqrt q}^{\,e} s -1}{ s-(-1)^\eps{\sqrt q}^{\,e}}.
  \eeq
  The rightmost ratio reduces to $(-1)^\eps$ if $e=0$, and its square is thus $1$. If $e=1$, its square is
  \[
    \left( \frac {(-1)^\eps {\sqrt q} s -1}{ s-(-1)^\eps{\sqrt q}} \right)^2=
    \frac{qs+ \frac 1 s -2 (-1)^\eps \sqrt q}{s+ \frac q s -2 (-1)^\eps \sqrt q}
    =\frac{y(qs)-y((-1)^\eps \sqrt q)}{y(s)-y((-1)^\eps \sqrt q)}=
    \frac{f_{e, \eps}(y(qs))}{f_{e, \eps}(y(s))}, 
  \]
  where $f_{e,\eps}$ is defined by~\eqref{f_e_eps-def}. Hence, whether $e=0$ or $e=1$, we obtain, by squaring~\eqref{F-ratio-double}:
  \[
    \left( \frac 1 {\sigma} \cdot \frac{s-\sigma}{s-1/\sigma}\right)^2=
    \left( \frac{F_{r,\sigma}(y(s))}{F_{r,\sigma}(y(sq))}\right)^2
    \frac{f_{e, \eps}(y(s))}{f_{e, \eps}(y(qs))}.
  \]
  Returning to the above expression of $E(s)$, this implies that the function $F$ given by~\eqref{second-choice} is indeed a $2$-decoupling function when the double angle condition~\eqref{ri-def} holds.
\end{proof}

We can now prove some parts of Theorems~\ref{thm:main_diffalg} and~\ref{thm:main_alg}.

\begin{cor}\label{cor:phi1m-expr}
  Assume that one of the angle conditions~\eqref{eq:CNS0} or~\eqref{eq:CNS0double} holds. Then there exist  $m\in \{1, 2\}$, a rational function $F$ and polynomials $S$ and $R$ such that
  \[
    F(y) \phi_1 ^m(y)= \frac S R \circ w(y),
  \]
  where $w$ is the canonical invariant of Section~\ref{sec:gluing}. In particular, $\phi_1$ is D-algebraic, and algebraic if $\beta/\pi \in \Q$.
\end{cor}
\begin{proof}
  Let $F$ be the decoupling function of Theorem~\ref{thm:decoupling}.  The function $F \phi_1^m$  is an invariant (Lemma~\ref{lem:Fphi1-invariant}).
  By Proposition~\ref{prop:BVP_Carleman_sketch},
  the function $\phi_1$ grows at most polynomially at infinity,  and  the same thus holds for $F\phi_1^m$.
  Proposition~\ref{prop:inv} thus applies to $F\phi_1^m$, and gives the expression of $\phi_1^m$. The algebraic/differential properties of $\phi_1$ come from those of $w$ (Proposition~\ref{prop:algebraic_nature_w}) and classical closure properties.
\end{proof}

In the next two sections, we will make the expression of $\phi_1^m$ completely explicit, by describing the roots of $S$ and $R$ in terms of the parameters of the model. We will derive from these expressions necessary and sufficient conditions for D-finiteness, algebraicity and rationality of~$\phi_1$. Since every pole or root of $S/R$ comes from a pole or root of $F$ or $\phi_1$ \emm lying in $\bG$,, we need to clarify how many  roots or poles of the function $F_{r,\sigma}$ defined by~\eqref{Fr-def} lie in $\bG$. In the following lemma, we focus on those that lie in $\G$. Indeed, since all roots and poles of  $F_{r,\sigma}$ are real, lying on the curve $\R$ simply means being equal to $y(-1)$. For $\sigma\in \C$ and $a, b \in \Z$, it will be convenient to define the numbers $m_+(\sigma; a,b)$ and $m_-(\sigma; a,b)$ by:
\beq\label{def:mult}
m_\pm(\sigma; a,b) =\sharp \left\{j : a\le j \le b \text{ and } \sigma q^{\pm j}=-1\right\}. 
\eeq
Note that $  m_\pm(\sigma; a,b)=0$ if $b<a$.

\begin{lem}\label{lem:count}
  Let $\sigma=e^{i\om}$ be a complex number of modulus $1$. For $r\ge 0$, the number of roots of the polynomial $F_{r,\sigma}(y)=P_{r,\sigma}(y)$ that lie in the open domain $\G$ is
  \[
    \left\lfloor \frac \om{2\pi} - \frac 1 2 \right \rfloor -
    \left\lfloor \frac \om{2\pi} - \frac 1 2 - r \frac \beta \pi\right \rfloor
    - m_-(\sigma;0, r-1).
  \]
  For $r\le 0$, the function  $F_{r,\sigma}(y)=1/Q_{r,\sigma}(y)$ is the reciprocal of a polynomial, and  the number of poles of $F_{r,\sigma}(y)$ that lie in the open domain $\G$ is
  \[
    \left\lfloor \frac \om{2\pi} - \frac 1 2 - r \frac \beta \pi\right \rfloor
    - \left\lfloor \frac \om{2\pi} - \frac 1 2 \right \rfloor  
    - m_+(\sigma; 1, |r|).
  \]
\end{lem}
\begin{proof}
  First, we note that, by definition of the parametrization~\eqref{eq:uniformization},
  \begin{align}
    y(s) \in \mathcal{R} & \iff  \arg(s) \in \{\pi, \pi+2\beta \}{\mod 2\pi}\label{eq:caracR}\\
    y(s) \in \G & \iff \arg(s) \in (\pi, \pi +2\beta ) {\mod 2\pi},
                  \label{eq:caracG}
  \end{align}
  where the second equation uses the fact that $y(s)$ is a negative real when $s\in e^{i(\pi+\beta)}\RR_+$. In other words,  the preimage by $y$ of ${\mathcal R}$ is $\RR_-\cup e^{2i\beta}\RR_-$ and  the preimage by $y$ of $\G$ is the green/shaded area   in Figure~\ref{fig:ellipse_uniformization} (in particular, $y(-e^{i\beta})=y^- \in \G$).

  Hence, when $r\ge 0$,  the question is to determine how many of the
  points $\sigma q^{-j}$, for $0\le j \le r-1$, have their argument in $(\pi,\pi + 2\beta)$ modulo $2\pi$. This argument is $\om-2j\beta$.

  This kind of counting problem is standard in the study of \emm Sturmian, or \emm mecanical, sequences~\cite[Chap.~2]{lothaire}.  Denoting by $\{x\}:= x-\lfloor x \rfloor$ the fractional part of $x$, we  want to determine
  \[
    \sharp \left\{ j: 0\le j< r : \left\{ \frac\om{2\pi} -\frac 1 2  -j  \frac\beta \pi \right\} \in (0,  \beta/ \pi) \right\}.
  \]
  Let us begin by counting those values of $j$ for which
  \beq\label{sturm}
  \left\{ \frac\om{2\pi} -\frac 1 2  -j  \frac\beta \pi \right\} \in [0,  \beta/ \pi).
  \eeq
  Observing that the difference
  \[
    \left  \lfloor \frac\om{2\pi} -\frac 1 2  -j  \frac\beta \pi \right\rfloor -
    \left\lfloor\frac\om{2\pi} -\frac 1 2  -(j+1)  \frac\beta \pi \right \rfloor
  \]
  takes values in $\{0,1\}$, and equals $1$ if and only if~\eqref{sturm} holds,  we conclude that
  \begin{align*}
    \sharp \left\{ j:  0\le j < r :  \left\{ \frac\om{2\pi} -\frac 1 2  -j  \frac\beta \pi \right\}  \in [0,  \beta/ \pi) \right\}
    &  =
      \sum_{j=0} ^{r-1} \left( \left\lfloor\frac\om{2\pi} -\frac 1 2  -j  \frac\beta \pi  \right\rfloor -   \left\lfloor \frac\om{2\pi} -\frac 1 2  -(j+1)  \frac\beta \pi\right\rfloor \right)\\
    & =   \left\lfloor \frac \om{2\pi} - \frac 1 2 \right \rfloor -
      \left\lfloor \frac \om{2\pi} - \frac 1 2 - r \frac \beta \pi\right \rfloor.
  \end{align*}
  We need to subtract the number of $j$ for which the fractional part shown in~\eqref{sturm} takes the value~$0$, which is equivalent to saying that $\sigma q^{-j}=-1$. This proves the lemma when $r\ge 0$.

  When $r<0$, the result follows by observing that the poles of $F_{r,\sigma}$ are the zeroes of $F_{-r,\sigma q^{-r}}$ (see~\eqref{F-sym}) and applying the above result, together with  $m_{-}(\sigma q^{-r},0,r-1)=m_{+}(\sigma,1,|r|)$.
\end{proof}

\section{Expression of the Laplace transform when $\boldsymbol{\alpha\in \Z + {\pi}\Z/{\beta}}$}
\label{sec:expression}

In this section, we assume that the simple angle  Condition~\eqref{eq:CNS0} holds. We choose an integer $r$ such that $s_1=q^r s_2$, and denote by $F$ the $1$-decoupling function of Theorem~\ref{thm:decoupling}. By Corollary~\ref{cor:phi1m-expr}, the function $F\phi_1$ can be written $(S/R) \circ w$  for some relatively prime  polynomials $S$ and $R$. In this section we determine the degrees and roots of these polynomials.

Returning to Lemma~\ref{lem:angles}, we see that the choice of $r$ defines an integer $k$ such that
\[
  \delta+\vareps= (1-r)\beta +(1+k)\pi.
\]
Equivalently, given the definition~\eqref{eq:def_alpha} of $\alpha$, 
\beq\label{eq:def_r_k}
\delta+\vareps-\pi - \beta  = \beta(\alpha-1) = k\pi- r\beta.
\eeq
If $q$ is not a root of unity, that is, if $\beta/\pi\not \in \Q$, then the choice of $r$ and $k$ is unique. 
Otherwise, write $\beta= n\pi/d$, with $n$ and $d$ relatively prime and $0<n<d$. Then if $(r,k)$ is a solution, all other solutions are of the form $(r+jd,k+jn)$, for $j\in \Z$. In particular, there always exist solutions such that $|r|<d$. What follows holds for any choice of $r$, but if we impose that $|r|<d$ when  $\beta= n\pi/d$, then all numbers $m_\pm$ that occur in the results of this section will be $0$ or $1$.

\subsection{Preliminaries}

We begin with a simple lemma that describes the behaviour of $F\phi_1$ at infinity.

\begin{lem}
  \label{lem:phi-infini}
  There exists a constant $\kappa \neq 0$ such that  
  \[
    (\phi_1F)(y) \underset{y\to\infty}{\sim} \kappa y^{k\frac{\pi}{\beta}}.
  \]
\end{lem}

\begin{proof}
  Recall from~\eqref{phi-asympt} that $\phi_1(y)$ grows like $y^{\alpha-1}$ at infinity.
  Since $\alpha-1=k\frac{\pi}{\beta}-r$ (see~\eqref{eq:def_r_k}) and $F(y)=F_{r, s_1}(y) \underset{y\to\infty}{\sim} y^{r}$   (see~\eqref{Fr-def}), the result follows.
\end{proof}

We now apply Lemma~\ref{lem:count} to the decoupling function $F=F_{r, s_1}$, to determine how many poles and roots of $F$ lie in $\G$. When $r\ge 0$ (resp.\ $r< 0$), we write $F=P$ (resp.\ $F=1/Q$) to lighten the notation $P_{r, s_1}$ (resp.\ $1/Q_{r,s_1}$). Recall the notation $m_\pm(\sigma;a,b)$ defined in~\eqref{def:mult}.

\begin{lem}
  \label{lem:poles_roots_F}
  Let us write as before $ \delta+\varepsilon=(1-r)\beta+(1+k)\pi$.
  Then either $r<0$ and $k<0$, or $r>0$ and $k\ge 0$.
  The $1$-decoupling function $F$ given in~\eqref{first-choice}  is rational with real roots and poles.
  \begin{enumerate}
  \item If $r>0$, so that  $F(y)=P(y)$, the number of roots of $F$ (counted with multiplicity)   lying in the   open region $\G$ (or equivalently in $(-\infty, y(-1))=(-\infty, Y^\pm(x^-))$) is:
    \[
      r_F:=k + \mathbbm 1_{2\beta -2\vareps -\theta\ge 0} 
      -  m_-(s_1; 0, r-1).
    \]

  \item If  $r<0$, so that $F(y)={1}/{Q(y)}$,  the number of poles  of $F$  lying in $\G$ is
    \[ 
      p_F:=-k - \mathbbm 1_{2\beta -2\vareps -\theta\ge 0} 
      -m_+(s_1;1, |r|).
    \] 
  \end{enumerate}
\end{lem}
If $q$ is not a root of unity, the numbers $m_\pm$ occurring in this lemma are $0$ or $1$. Otherwise, as discussed above, we can always choose $|r|<d$ if $\beta=n\pi/d$, and then this property still holds. 

\begin{proof}
  Recall from~\eqref{assumptions} that $0<\delta+\varepsilon-\beta<\pi$. This implies that $r$ cannot be $0$,  and that  $r<0$ implies $k<0$, while $r>0$ implies $k\ge 0$.

  Now recall that $F= F_{r, s_1}$, and let us apply Lemma~\ref{lem:count} with     $\sigma=s_1$, or equivalently
  \[
    \om = \pi + 2\beta-2\vareps-\theta= -\pi +2\delta -\theta + 2r\beta -2k\pi.
  \]
  From the first expression of $\om$, we derive 
  \beq\label{int-part1}
  \left\lfloor \frac \om{2\pi} - \frac 1 2 \right \rfloor =
  \left\lfloor \frac{2\beta -2\vareps -\theta}{2\pi}\right\rfloor
  =-\mathbbm 1_{2\beta -2\vareps -\theta< 0}.
  \eeq
  Indeed, the first  assumption in~\eqref{assumptions}, together with the fact that $\delta>0$ and $\theta<\pi$,  implies:
  \[
    -1 <  \frac{\delta -\theta -2 \pi}{2\pi}<\frac{2\delta -\theta -2 \pi}{2\pi}<\frac{2(\beta- \varepsilon)-\theta}{2\pi} < \frac{\theta}{2 \pi}< \frac{1}{2}.
  \]
  Moreover, the second expression of $\om$ given above leads to
  \beq\label{int-part2}
  \left\lfloor \frac \om{2\pi} - \frac 1 2 - r \frac \beta \pi\right \rfloor
  =\left\lfloor -k-1 + \frac {2\delta-\theta}{2\pi}  \right \rfloor 
  =-k-1,
  \eeq
  since $2\delta >\theta$ by~\eqref{assumptions}, and $\delta<\pi$.
  Then Lemma~\ref{lem:count}  gives the announced expressions for the number of roots of $Q$ and $P$ lying in $\G$.
\end{proof}

\subsection{Expression of $\boldsymbol{\phi_1}$}
\label{subsec:proofmainthm} 
We can now describe precisely the polynomials $S$ and $R$ such that $F \phi_1 = \frac S R \circ w$. Recall that $w$ is  the canonical invariant defined in~\eqref{eq:definition_w-bis}.

\begin{thm} 
  \label{thm:main}
  Let us assume that Condition~\eqref{eq:CNS0} holds, that is, $\alpha \in \Z+ \pi \Z/\beta$,  and let $k,r \in \Z$ be chosen so that $ \delta+\vareps = (1-r)\beta+(1+k)\pi$. Recall that $r\not = 0$.
  Let $F=F_{r, s_1}$ be the decoupling function of Theorem~\ref{thm:decoupling}, defined by~\eqref{Fr-def}. Depending on the sign of $r$, we have  $F=P$ (when $r>0$) or $F=1/Q$ (when $r<0$), where $P$ and $Q$ are (real-rooted) polynomials of degree $\vert r\vert$.

  The Laplace transform $\phi_1$ defined   by~\eqref{eq:Laplace_transform_boundary} can be meromorphically continued to $\C\setminus [y^+,\infty)$. Moreover,
  \begin{enumerate}[label={\rm(\arabic*)},ref={\rm(\arabic*)}] 
  \item\label{case1}
    if $r<0$, then $k< 0$ and 
    \[ 
      \phi_1(y)=\frac{Q(y)}{R(w(y))},
    \] 
    where $R$ is a polynomial of degree $|k|$ whose roots (taken with multiplicity) are
    \begin{itemize}
    \item 
      the $     w(y(s_1q^{j})) $ for   $j=1,\ldots , |r|$
      such that  $y(s_1q^{j}) \in  \G$, 
    \item    plus $w(y(s_1))$ if $2\beta-2\varepsilon-\theta> 0$,
    \item and finally $w(y(-1))=-1$, with multiplicity $m_+(s_1; 0, |r|-1)$.
    \end{itemize}

  \item\label{case2} if $r>0$, then $k\geqslant 0$ and
    \[
      \phi_1(y)=\frac{S(w(y))}{P(y)},
    \] 
    where $S$ is a polynomial of degree $k$ whose roots (taken with multiplicity) are
    \begin{itemize}
    \item 
      the $     w(y(s_1q^{-j}) )$,  for  $j=1,\ldots , r-1$
      such that  $y(s_1q^{-j})\in \G$,
    \item  plus $w(y(s_1))$ if $y(s_1)<y(-1)$ and  $2\beta-2\varepsilon-\theta\le  0$,
    \item and finally $w(y(-1))=-1$, with multiplicity $m_-(s_1;1, r-1)$.
    \end{itemize}
  \end{enumerate}
  In particular, $\phi_1$ is always D-algebraic. It is D-finite if and only if either $\beta/\pi\in \Q$ or $\alpha\in -\N_0 +\frac \pi \beta \Z$.
  It is   algebraic if and only if  $\beta/\pi \in \Q$  or  $\alpha\in-\mathbb{N}_0$. It is rational if and only if $\alpha\in-\mathbb{N}_0$. Finally, $1/\phi_1$ is D-finite if and only if $\beta/\pi \in \Q$  or  $\alpha\in -\N_0\cup (\mathbb{N} +\frac \pi \beta \Z)$.
\end{thm}

\noindent {\bf Remarks}\\
1. As discussed earlier, there is always a choice of  $r$ that gives the value $0$ or $1$ to the multiplicities $m_\pm(s_1;a,b)$ that occur in the theorem.\\
2. The above theorem characterizes  the polynomials $R$ and $S$ 
up to  multiplicative constants that can be adjusted thanks to the value $\phi_1(0)$ given in~\eqref{eq:valuephi1O}. Thus we can compute $\phi_1(y)$ explicitly, using the simple characterization of points $s$ such that $y(s)\in \G$ given by~\eqref{eq:caracG}. Several examples are worked out in Section~\ref{sec:examples}.

\begin{proof}[Proof of Theorem~\ref{thm:main}]
  By Corollary~\ref{cor:phi1m-expr}, the function $F\phi_1$ is an invariant.
  We  will construct  two   polynomials $R$ and $S$ such that $F\phi_1\times (R/S)( w)$, which is still an invariant,  has no pole in $\bG$ and has a finite limit at infinity.  Proposition~\ref{prop:inv} will then allow us to conclude that it   is a constant.

  \noindent
  {\bf An observation.} We begin with a useful observation, which relies on the properties of the map $w$ listed in Lemma~\ref{lem:w}. Consider a rational function $H(z)$, 
  and the function $H(w(y))$, which is well defined on $\bG$. Then $w$ induces a bijection between the 
  roots (resp.\ poles)    of $H\circ w$ lying in $\G$ and the 
  roots (resp.\ poles)    of $H$ lying away from the cut $(-\infty, -1]$.  This bijection preserves the multiplicity. Moreover, $z_0=-1$ is a 
  root (resp.\ pole)    of $H$ if and only if $y_0:=w^{-1}(-1)=y(-1)$ is a 
  root (resp.\ pole)    of $H\circ w$, and the multiplicity of $y_0$ in  $H\circ w$ is twice the multiplicity of $z_0$ in $H$. Finally, if~$H$ has no root (resp.\ pole)     in $(-\infty, -1)$, then $H\circ w$ has no root (resp.\ pole) 
  on $\R\setminus\{y(-1)\}$.

  \smallskip
  \noindent
  {\bf First case: $\boldsymbol{r<0}$}. Then  $k<0$ and $F=\frac{1}{Q}$. As explained above, we will list the poles of $F\phi_1$ lying in $\bG$ to construct    the polynomial $R$. Recall that all these poles are real. We refer to~\eqref{Fr-def} for the expression of $F$, and to Proposition~\ref{prop:BVP_Carleman_sketch} for the properties of $\phi_1$.
  First,   $F$ has $p_F$ poles in $\G$, where $p_F$ is given by Lemma~\ref{lem:poles_roots_F}. Moreover, $\phi_1$ has a pole in $\G$ if and only if $2\beta-2\vareps-\theta>0$. This pole is then simple, and located  at $y(s_1)$. Now, the multiplicity of $y(-1)=Y^\pm(x^-)$ as a pole of $F$ is 
  \begin{multline*}
    \sharp\left \{j: 1\le j\le |r|, s_1q^{j}\in \{-1,-q\}\right\}
    \\
    \begin{split}
      &  =\sharp\left \{j: 1\le j\le |r|, s_1q^{j}=-1\right\}
      +  \sharp\left \{j: 0\le j< |r|, s_1q^{j}=-1\right\}
      \\& = 2 \sharp\left \{j: 0\le j< |r|, s_1q^{j}=-1\right\} +\mathbbm 1_{s_1q^{{-r}}=-1} {-} \mathbbm 1_{s_1=-1}
      \\& = 2 m_+(s_1; 0, |r|-1) - \mathbbm 1_{s_1=-1},
    \end{split}
  \end{multline*}
  as $s_1q^{|r|}=s_2$ is never equal to $-1$ (see Lemma~\ref{lem:angles}). This multiplicity is not always even, but we should remember that $\phi_1$ has a (simple) pole at $y(-1)$ if and only if $2\beta-2\vareps-\theta=0$, that is,
  if $s_1=-1$.  Consequently,  the multiplicity of $y(-1)$ as a pole of $F\phi_1$ is $2 m_+(s_1; 0, |r|-1)$.
  
  These considerations lead us to introduce the polynomial $R$ defined (up to a multiplicative constant) in the theorem. Its degree is
  \[
    p_F + \mathbbm 1_{2\beta-2\vareps-\theta>0}+ m_+(s_1; 0, |r|-1)
    =-k.
  \]
  We have used again the fact that $s_1q^{|r|}\not = -1$ and that $s_1=-1 \Leftrightarrow 2\beta-2\vareps-\theta=0$.  Now consider the invariant $I:=(F \phi_1)\times  R(w)$. By the above observation, it has no pole in $\bG$. Now recall that $w(y)$ behaves like $y^{\pi/\beta}$ at infinity, and $F\phi_1(y)$ as $y^{k\pi/\beta}$ (Lemma~\ref{lem:phi-infini}).  Given that $R$ has degree $-k$, we conclude that $I$  is bounded    at infinity.  By Proposition~\ref{prop:inv}, it is constant, and we can take $S=1$ upon adjusting the multiplicative constant in $R$.

  \medskip \noindent {\bf Second case: $\boldsymbol{r>0}$.} Then $k\geqslant 0$ and $F=P$. This time we will construct a candidate for $S$ by examining the roots (rather than the poles) of $F\phi_1$ lying in $\G\cup\{y(-1)\}$.
  The decoupling function $F$ has $r_F$ roots in $\G$, where $r_F$ is given by Lemma~\ref{lem:poles_roots_F}. Only one of these roots, namely $y(s_1)$, is likely to be cancelled by a pole of $\phi_1$ in the product $F\phi_1$. This happens if and only if     $2\beta-2\vareps-\theta>0$. Now, the multiplicity of $y(-1)$ as a root of $F$ is
  \begin{multline*}
    \sharp\left \{j: 0\le j<r, s_1q^{-j}\in \{-1,-q\}\right\}
    \\
    \begin{split}
      &  =\sharp\left \{j: 0\le j<r, s_1q^{-j}=-1\right\}  +  \sharp\left \{j: 1\le j\le r, s_1q^{-j}=-1\right\}
      \\& = 2 \sharp\left \{j: 1\le j< r, s_1q^{-j}=-1\right\} + \mathbbm 1_{s_1=-1}
      + \mathbbm 1_{s_1q^{-r}=-1} 
      \\& = 2 m_-(s_1; 1, r-1)      + \mathbbm 1_{s_1=-1},
    \end{split}
  \end{multline*}
  by the same arguments as in the case $r<0$. This multiplicity is not always even, but we should remember that $\phi_1$ has a (simple) pole at $y(-1)$
  if and only if  $s_1=-1$.  Consequently,  the multiplicity of $y(-1)$ as a root of $F\phi_1$ is $  2 m_-(s_1; 1, r-1)$.
  These considerations lead us to introduce the polynomial $S$ described in the theorem. Its degree is
  \[
    r_F - \mathbbm 1_{2\beta-2\vareps-\theta>0}+
    m_-(s_1; 1, r-1)
    =k.
  \]
  Now consider the invariant $F \phi_1/ S(w)$. By construction,  it has no pole in $\bG$. We can argue as in the previous case to prove that it is bounded at infinity.
  Hence, it is constant by Proposition~\ref{prop:inv}.

  \medskip
  Now that we have given expressions for $\phi_1$, its meromorphicity on $\C\setminus [y^+,\infty)$ follows from the fact that the  canonical invariant $w$ is analytic on this domain. Moreover, $\phi_1$ is D-algebraic because $w$ is D-finite.

  \medskip
  
  Let us now discuss the other differential/algebraic properties of $\phi_1$, starting from rational cases. If  $k=0$, that is, $\alpha \in {-}\N_0$, then $\phi_1$ is  the reciprocal of a polynomial, hence a rational function. Conversely, if $\phi_1$ is rational, then~\eqref{phi-asympt} implies that $\alpha$ is an integer, and thus belongs to $-\N_0$ since we have assumed $\alpha<1$. This concludes the characterization of rational cases, and we  now assume that $k\not = 0$.  It then follows from the expressions of $\phi_1$  that if $\phi_1$ is algebraic then so is~$w$, which forces  $\beta/\pi \in \Q$ by Proposition~\ref{prop:algebraic_nature_w}. Conversely, if $\beta/\pi \in \Q$ then $w$ is algebraic and so is $\phi_1$. Finally, the characterization of D-finite cases stems from  Proposition~\ref{prop:algebraic_nature_w}. Indeed, $\phi_1$ is D-finite if and only if $(S/R)(w)$ is D-finite, and $R$ is non-trivial as soon as $r<0$. Hence $\phi_1$ is D-finite if and only if either $\beta/\pi \in \Q$ or $r>0$. The latter condition translates into $\alpha\in -\N_0 +\frac \pi \beta \Z$.

  A similar argument proves that $1/\phi_1$ is D-finite if and only if $\beta/\pi \in \Q$, or $r<0$, or $k=0$, which translates into the conditions stated in the theorem.
\end{proof}

\subsection{Examples}\label{sec:examples}

We now give five applications of Theorem~\ref{thm:main}.
We start with the three already known cases mentioned in Section~\ref{sec:main} and illustrated in Figure~\ref{fig:three_examples}. Then we detail an algebraic case, and finally a D-finite one.
Recall that we choose integers  $r$ and $k$ such that
\[
  \delta+\varepsilon =(1-r)\beta+(1+k)\pi.
\]

\subsubsection{The skew symmetric case}
The  model is said to be skew symmetric if $\delta+\varepsilon=\pi$, that is, $\alpha=0$. It can be shown thanks to~\eqref{eq:beta} and~\eqref{eq:expression_delta_epsilon} that this is equivalent to 
\[
  2\sigma_{12}=\frac{r_{21}}{r_{11}}\sigma_{11}+\frac{r_{12}}{r_{22}}\sigma_{22}.
\]
One can  take in  this case $r=1$ and $k=0$.
Then Theorem~\ref{thm:decoupling} gives the decoupling function $P(y)=y-y(s_1)$. Theorem~\ref{thm:main}\ref{case2} implies that for some constant $\kappa $,
\[
  \phi_1(y)=\frac{\kappa }{P(y)}=\frac{y(s_1)\phi_1(0)}{y(s_1)-y},
\]
where $\phi_1(0)$ is given by~\eqref{eq:valuephi1O} and $y(s_1)$ by Proposition~\ref{prop:BVP_Carleman_sketch}. 

If we invert the Laplace transform, we find that the density of the invariant measure $\nu_1$ is exponential. This result is very well known and can be found for instance in~\cite{HaRe-81,harrison_multidimensional_1987,DaMi-13}. 
Note that~\cite{harrison_multidimensional_1987} actually contains a higher dimensional version of this result.

\subsubsection{The Dieker and Moriarty condition~\cite{DiMo-09}}
\label{sec:DM}
It reads $\alpha=\frac{\delta+\varepsilon-\pi}{\beta}\in -\mathbb{N}_0$ and generalizes the previous case. We can take $r=1-\alpha>0$ and $k=0$. The decoupling function  $F(y)=F_{r,s_1}$ is given by~\eqref{Fr-def}, and it is a polynomial $P(y)$. Theorem~\ref{thm:main}\ref{case2} implies that for some constant $\kappa $,
\[
  \phi_1(y)=\frac{\kappa }{P(y)}.
\]
When  all roots of $P$ are simple, we obtain by inverting the Laplace transform  that the density of~$\nu_1$ is a sum of exponentials. In fact, we can show using~\eqref{Fr-def} and the expression of  $s_1$ given in Lemma~\ref{lem:angles} that $P$ has a multiple root if and only if $2\varepsilon +\theta+j\beta =0 \mod \pi$ for some $j\in \llbracket 0, 2r-4\rrbracket$.  Equivalently, since $\alpha=-r+1$, this is equivalent to saying that $\theta -2\delta-j\beta=0 \mod \pi$ for some $j\in \llbracket 2, -2\alpha \rrbracket$. Note that Dieker and Moriarty prove that   the density of $\nu_1$ is a sum of exponentials under the (slightly stronger) assumption that $\theta -2\delta-j\beta\not =0 \mod \pi$ for all $j\in \llbracket 0, -2\alpha \rrbracket$ (this is equivalent to their condition $``\theta \in \Theta_l "$ occurring in~\cite[Thm.~1]{DiMo-09}).

A double pole occurs for instance when $\delta=\pi-\varepsilon -\beta$ (so that $\alpha=-1$ and $r=2$) and $ \theta=\pi-2\varepsilon$. Then $s_1=q$, $y(s_1)=y(q)=y(1)=y(s_1q^{-1})$,  and
\[
  \phi_1(y)= \frac{y(1)^2\phi_1(0)}{(y(1)-y)^2}.
\]
Hence the density of the invariant measure $\nu_1$ is, up to a multiplicative constant, $p_1(z)=z e^{-y(1)z}$ and $\nu_1$ is a Gamma/Erlang distribution.
An example satisfying the angle conditions~\eqref{assumptions} of the paper is $(\beta, \delta, \varepsilon, \theta)= (5/16,5/16,3/8,1/4)\pi$.

\subsubsection{The orthogonal case}
In this case $R$ is a diagonal matrix, or equivalently $\delta=\varepsilon=\beta$. Thus we can take $r=-1$ and $k=-1$.  The decoupling function $F_{r,s_1}=1/Q$ of Theorem~\ref{thm:decoupling} reads $F=1/(y-y(s_1q))=1/(y-y(s_2))$. Since $\gamma_2(x,y)= r_{22}y$ in the orthogonal case, we have  $y(s_2)=0$ by definition of $s_2$.
Hence  $Q(y)$ is simply $y$.
Theorem~\ref{thm:main}\ref{case1}  gives $R(w(y))=\kappa (w(y)-w(y(s_2)))=\kappa (w(y)-w(0))$ for some constant $\kappa $. Using the identity $\phi_1(0)=-\mu_1/r_{11}$ derived from~\eqref{eq:valuephi1O}, we obtain
\[
  \phi_1(y)=\frac{Q(y)}{R(w(y))}=-\frac{\mu_1 }{{r_{11}}}\frac {w'(0) y}{w(y)-w(0)} ,
\]
which is the main result of~\cite{franceschi_tuttes_2016}. Note that the term $r_{11}$ does not appear in~\cite{franceschi_tuttes_2016}, because the reflection matrix is taken to be the identity therein.

\subsubsection{An algebraic case}  As stated in Theorem~\ref{thm:main}, under the angle condition $\alpha\in \Z+ \pi\Z/\beta$, the function~$\phi_1$ is algebraic if and only if either $\alpha$ is an integer (this is the rational case, already discussed above) or $\beta/\pi$ is rational. So let us assume that $\beta=n\pi/d$, with $n$ and $d$ coprime and $0<n<d$. In this case $T_{\pi/\beta}$ and $w$ are algebraic of degree (at most) $n$.  Observe that if $n=1$, the angle condition simply reads $\alpha\in \Z$, so that we are again in a rational case. So let us assume that $n=2$ and $d=3$, so that $\beta=2\pi/3$. Then the angle condition requires  $\alpha$ to be a half-integer, say $\alpha=1/2$, that is, $\delta+\vareps=4\pi/3$. Then we can take $k=r=-1$, and $\phi_1(y)$ has the following form:
\[
  \phi_1(y)=\kappa\, \frac{y-y(s_1q)}{w(y)-w(y(s_1q^e) )},
\]
where $e=1$ if $2\beta-2\vareps-\theta <0$ and $e=0$ otherwise. It is a quadratic function of $y$, since $T_{\pi/\beta}= T_{3/2}$ is itself quadratic:
\[
  T_{3/2}(x)= (2x-1)\sqrt{\frac{1+x}2}.
\]
One particularly interesting case is $\theta=2\beta-2\vareps$. Then $s_1=-1$ (see Lemma~\ref{lem:angles}), $y(-1)=y(-q)$ and $w(y(-1))=-1$. It is convenient to introduce $u=\sqrt{\frac{y^+-y}{y^+-y^-}}$, which yields
  \[
    w(y)= T_{3/2}(2u^2-1)=u(4u^2-3) \quad \text{and} \quad w(y)+1=(1+u)(1-2u)^2.
  \]
  When $y=y(-1)$, we have $w(y)=-1$, which gives $2u^2-1=-1/2$. Hence the numerator of $\phi_1(y)$, namely $y-y(-q)=y-y(-1)$, is $1-4u^2$ (up to a multiplicative constant). Finally,
\[
  \phi_1(y) = \kappa' \frac{1-4u^2}{(1+u)(1-2u)^2}= \kappa' \frac{1+2u}{(1+u)(1-2u)}.
\]

In Section~\ref{sec:examples-double}, under the double angle condition, we will detail another case where $\phi_1$ is quadratic while $\beta/\pi \not \in \Q$, with an even simpler expression of $\phi_1$ (see~\eqref{phi1-very-simple}). In that case, we will give explicit expressions for  the density of $\nu_1$ (and in fact of the whole stationary distribution~$\nu$).

\subsubsection{A D-finite example: recurrence for the moments}
\label{sec:DF-example}
Suppose now that
\beq\label{example4}
\delta+\varepsilon+\beta=2\pi,
\eeq
that is, $\alpha=\pi/\beta-1$. Then we can take  $r=2$ et $k=1$.
Applying Theorems~\ref{thm:decoupling} and~\ref{thm:main}, we obtain
\[
  \phi_1(y)=\kappa\,\frac{w(y)-w_0}{(y-y(s_1))(y-y(s_1 /q))},
\]
where the constant $\kappa$ can be derived from the normalisation~\eqref{eq:valuephi1O}, and 
\[
  w_0=
  \begin{cases}
    w(y(s_1/q)) & \text{if } y(s_1/q) \in \G,\\
    w(y(s_1)) & \text{if } y(s_1)<y(-1) \text{ and } 2\beta-2\varepsilon -\theta \le 0,\\
    w(y(-1))=-1 & \text{if } s_1 =-q.
  \end{cases}
\]
These three cases are those of Theorem~\ref{thm:main}. Using~\eqref{eq:caracG} and the expression of $s_1$ in Lemma \ref{lem:angles}, as well as~\eqref{example4} and the basic conditions~\eqref{assumptions}, they can be rewritten respectively as
\[
  2\varepsilon +\theta < 2\pi, \qquad 2\pi < 2\varepsilon + \theta, \quad \text{and} \quad 2\varepsilon +\theta = 2\pi.
\]
%
%
These three cases actually occur, for instance with the three following values of $(\beta, \delta, \varepsilon, \theta)$:
\[
  (2/3, 5/6, /2, 1/4)\pi, \qquad (2/3, 4/9, 8/9, 1/4)\pi, \qquad (2/3, 7/12, 3/4, 1/2)\pi.
\]

Starting from the expression~\eqref{eq:definition_w-bis} of $w$ in terms of $T_a$ (with $a=\pi/\beta$), the differential equation~\eqref{de-Ta} satisfied by $T_a$ leads to a (non-homogeneous) second order linear differential equation with polynomial coefficients in $y$ satisfied by $\phi_1(y)$. Upon expanding it in $y$, it gives a linear recurrence relation between the moments
\[
  M_n= \int _{\RR_+}t^n \nu_1({\rm d}t) = n! [y^n] \phi_1(y).
\]
This recurrence is found to be of fourth order. Its coefficients   are polynomials in $n$, of degree $4$. We refer to our {\sc Maple} worksheet
for details. We have used the {\sc Gfun} package~\cite{gfun} to derive the recurrence relation from the differential equation.

To give an explicit example, let us focus on the simplest case, that is, $2\varepsilon +\theta = 2\pi$, where $s_1=-q$. Note that~$\delta$ and $\vareps$ are now completely determined in terms of the two remaining angles, $\beta$ and $\theta$. Then 
\[
  \phi_1(y)=\kappa\,\frac{w(y)+1}{(y-y(-1))^2},
\]
since  $y(-1)=y(-q)$ as $y(s)=y(q/s)$. Equivalently, denoting $z:={2y}/(y^+-y^-)$ and using~\eqref{eq:uniformization} and~\eqref{eq:definition_w-bis}, we find
\beq\label{phi1-simple}
\phi_1(y)=\kappa'\, \frac{T_{a}(c_2-z)+1}{(c_2-c_1-z)^2},
\eeq
with
\[
  a=\pi/\beta, \qquad   c_2=\frac{y^++y^-}{y^+-y^-}=\cos(\beta-\theta), \qquad c_1=\cos \beta.
\]
The second expression of $c_2$ follows from~\eqref{xyn-pm}.
Setting $z=0$ gives
\[
  \kappa'=\phi_1(0) \frac{(c_2-c_1)^2}{T_a(c_2)+1}= \phi_1(0) \frac{(c_2-c_1)^2}{1-\cos(\pi \theta/\beta)},
\]
where $\phi_1(0)$ is given by~\eqref{eq:valuephi1O}. Then, writing
\[
  \phi_1(y)=\sum_{n\ge 0} \frac{M_n} {n!} y^n = \sum_{n\ge 0} \frac{\widetilde {M}_n} {n!} z^n,
\]
that is,
\[
  M_n = \frac {2^n}{(y^+-y^-)^n} \widetilde M_n,
\]
we obtain the first two coefficients $\widetilde M_n$ from~\eqref{phi1-simple}:
\[
  \widetilde M_0=\phi_1(0), \qquad 
  \widetilde M_1=\phi_1(0) \left( \frac 2{c_2-c_1} - \frac{a \sin(\pi\theta/\beta)}{\sin(\beta-\theta)(1-\cos(\pi\theta/\beta))}\right),
\]
and then  the sequence $\widetilde M_n$ satisfies the following recurrence relation, valid for $n\ge 0$:
\begin{multline*}
  \left(1- c_2^2 \right)  \left( c_1-c_2 \right) ^{2}\widetilde M_{n+2}= a^2 \kappa' \mathbbm 1_{n=0}+
  \left( c_2-c_1 \right)  \left(2\left( c_1c_2-2\,c_2^2+1 \right) n+c_1c_2-
    5\,c_2^2+4\right)\widetilde M_{n+1} 
  \\ + \left( \left( {c_1}^{2}-6\,c_1c_2+6\,c_2^2-1 \right) {n}^{2}
    - 3\left( 2\,c_1c_2-3\,c_2^2+1 \right) n
    - \left( c_{{1}}-c_2 \right) ^{2}{a}^{2}-2\,c_1c_2+4\,c_2^2-
    2
  \right) \widetilde M_n
  \\
  -n \left(2\left( 2\,c_2 -c_1\right) {n}^{2}+3\,c_2n+ 2\left( 
      c_1        -c_2 \right) {a}^{2}+c_2    
  \right) \widetilde M _{n-1}
  + n\left( n-1 \right)  \left( n^2-a^2 \right) \widetilde M_{ n-2}.
\end{multline*}

\section{Expression of the Laplace transform when $\boldsymbol{\alpha_1, \alpha_2\in \Z + \pi\Z/\beta}$}
\label{sec:expression-double}

We now assume that the double angle  Condition~\eqref{eq:CNS0double} holds. Then by Theorem~\ref{thm:decoupling}, there exists a $2$-decoupling function $F$, and by Corollary~\ref{cor:phi1m-expr}, the function $F\phi_1^2$ 
can be written $(S/R) \circ w$ for some polynomials $S$ and $R$. In this section we determine the rational function $S/R$. The arguments are the same as in the previous section, but all expressions are a bit heavier, as can be foreseen from the expression~\eqref{second-choice} of  $F$. Recall from~\eqref{ri-def}  that we have chosen integers $r_1$ and~$r_2$, and $e_1$, $e_2$, $\eps_1$, $\eps_2$ in $\{0,1\}$, such that $s_i=(-1)^{\eps_i} {\sqrt q} ^{\,e_i}q^{r_i}$ for $i=1,2$, with $\sqrt q =e^{i\beta}$. Returning to Lemma~\ref{lem:angles}, this defines two integers $k_1$ and $k_2$ such that the arguments of $s_1$ and~$s_2$ are respectively:
\begin{alignat}{3}
  \label{eq:def_r_k-double1}
  \om_1&:= \pi + 2\beta -2\varepsilon -\theta \ &= (2r_1+e_1)\beta -(2k_1+\eps_1) \pi, \\
  \label{eq:def_r_k-double2}
  \om_2&:= - \pi +2\delta-\theta& = (2r_2+e_2)\beta - (2k_2+\eps_2) \pi.
\end{alignat}
Equivalently, given the definition~\eqref{eq:def_alpha_i} of $\alpha_1$ and $\alpha_2$,
\beq \label{alpha12-expr}
\begin{array}{rllll}
  1-\alpha_1&=&  {\om_1}/\beta &= &2r_1+e_1 -(2k_1+\eps_1) \frac \pi \beta,       \\
  \alpha_2 &=&  {\om_2}/\beta& =& 2r_2+e_2 -(2k_2+\eps_2) \frac \pi \beta.
\end{array}
\eeq
Note that we now have
\[
  \delta+\vareps =\left(1-r_1  + r_2 - \frac{e_1-e_2} 2\right) \beta +
  \left(1+ k_1 -k_2  + \frac{\eps_1-\eps_2} 2\right) \pi,
\]
which should be compared to the condition $\delta+\vareps =(1-r)\beta+(1+k)\pi$ of the previous section.

As in the previous section, the numbers $r_i$ (and $e_i$, and $k_i$, and $\eps_i$) are uniquely defined when~$q$ is not a root of unity. Otherwise, if $\beta =\pi n/d$ with $0<n<d$, we may always choose each $r_i$ such that ${2}|r_i|<d$.
Such a choice of $r_i$ will sometimes simplify certain expressions, but what follows holds for any choice.

In the previous section, we had either ($r<0$ and $k<0$), or ($r>0$ and $k\ge 0$). The counterpart of these properties reads as follows.

\begin{lem}\label{lem:ineq-double}
  If $r_1\le 0$ then $2k_1+\eps_1\le 0$, and in particular $k_1\le 0$. 
  If $r_2\le 0$ then     $k_2\le 0$.

  If $2r_1+e_1 >0$ (and in particular if $r_1>0$), then $k_1 \geq 0$. If $r_2\ge 0$ then $k_2\ge 0$. 
\end{lem}
\begin{proof}
  We will give lower and upper bounds on the arguments $\om_i$ defined by~\eqref{eq:def_r_k-double1} and~\eqref{eq:def_r_k-double2}, using  the angle assumptions~\eqref{assumptions} and  the fact that the  angles $\beta$, $\theta$, $\delta$ and $\vareps$   lie in $(0, \pi)$.
  
  First, since $\theta<\beta$, 
  \[
    \om_1=\pi+2\beta-2\vareps-\theta > \pi+\beta -2\vareps .
  \]
  Moreover, since $\delta>\theta>0$,
  \[
    \om_2 =-\pi+2\delta-\theta >-\pi+\delta>-\pi.
  \]
  Let us denote $\bar r_i=2r_i+e_i$ and $\bar k_i=2k_i+\eps_i$, for $i=1,2$.
  Then $\om_1= \bar r _1\beta-\bar k_1\pi >\pi+\beta -2\vareps $ rewrites as  $(\bar r _1-1)\beta +2\vareps>(\bar k_1+1)\pi$. If $r_1\le 0$, that is, $\bar r_1 \le 1$, this implies that $\bar k_1\le 0$, because $\vareps<\pi$.   Analogously, $\om_2= \bar r _2\beta-\bar k_2\pi >-\pi$ rewrites as  $\bar r _2\beta>(\bar k_2-1)\pi $. If $r_2\le 0$, that is, $\bar r_2 \le 1$, this implies that $\bar k_2\le 1$, that is, $k_2\le 0$.

  Now, given that $\theta>\beta-\vareps$, we have
  \[
    \om_1 < \pi+\beta -\vareps <\pi+\beta,
  \]
  which implies $(\bar r_1-1)\beta <(\bar k_1+1)\pi$. Hence if $\bar r_1>0$, then $\bar k_1 \ge0$, that is $k_1\ge 0$.
  Finally, we have $\om_2<-\pi+2\delta < \pi$, which gives $\bar r _2\beta<(\bar k_2+1)\pi$. Hence if $\bar r_2\ge 0$, that is, $r_2\ge 0$, then $\bar k_2\ge 0$, that is, $k_2\ge 0$.
\end{proof}

\subsection{Preliminaries}
We begin with a simple lemma that describes the behaviour of $F\phi_1^2$ at infinity.

\begin{lem}
  \label{lem:phi-infini-double}
  There exists a constant $\kappa \neq 0$ such that
  \[
    (F\phi_1^2)(y) \underset{y\to\infty}{\sim} \kappa\, y^{(2k_1-2k_2+\eps_1-\eps_2)\frac{\pi}{\beta}}.
  \] 
\end{lem}

\begin{proof}
  The arguments are the same as in the proof of Lemma~\ref{lem:phi-infini}. Ignoring multiplicative factors, the behaviour at infinity of $\phi_1(y)$ is $y^{\alpha-1}$, where now  $\alpha=(\delta+\varepsilon-\pi)/\beta$ satisfies
  \[
    \alpha-1=\frac 1 2 \left( 2r_2-2r_1+e_2-e_1 + (2k_1-2k_2+\eps_1-\eps_2) \frac \pi \beta \right).
  \]
  The behaviour at infinity of the decoupling function $F(y)$ of~\eqref{second-choice} is $y^{2r_1-2r_2+e_1-e_2}$, and the result follows.
\end{proof}

Our next lemma will be used to prove that certain  polynomials have no common roots, under an additional assumption.

\begin{lem}\label{lem:interferences}
  Assume that the simple angle condition~\eqref{eq:CNS0} does \emm not, hold, that is, $s_1/s_2 \not \in q^\Z$.
  Then there exist no integers $i$ and~$j$ such that $y(s_1q^i)=y(s_2q^j)$.
\end{lem}
\begin{proof}
  Recall that $y(s)=y(s')$ if and only if $s'=s$ or $s'=q/s$. Hence if $y(s_1q^i)=y(s_2q^j)$, then either $s_1/s_2 \in q^\Z$, or $s_1 s_2 \in q^{\mathbb{Z}}$.  Since $s_2^2 \in q^\Z$, in both cases we would have $s_1/s_2\in q^\Z$, which we have excluded.
\end{proof}

Let us denote, for $i=1,2$: 
\beq\label{F-shortcut}
F_i= F_{r_i, s_i}, \quad P_i= P_{r_i, s_i}, \quad Q_i= Q_{r_i, s_i}.
\eeq
We now apply Lemma~\ref{lem:count} to determine how many poles and roots of $F_i$ lie in $\G$, for $i=1,2$. Recall the definition of the numbers $m_\pm(\sigma;a,b)$ in~\eqref{def:mult}.

\begin{lem}
  \label{lem:poles_roots_F-double}
  Define the integers $k_1$ and $k_2$ by~\eqref{eq:def_r_k-double1} and~\eqref{eq:def_r_k-double2}.
  \begin{itemize}
  \item When $r_1\ge 0$, the number of roots of $F_1=P_1$ (counted with multiplicity)\ lying in the   open region $\G$ is
    \[
      r_{F_1}:=k_1+\mathbbm 1_{2\beta-2\vareps-\theta \ge 0}- m_-(s_1;0,r_1-1).
    \]
  \item When $r_1\le 0$, the number of poles of $F_1=1/Q_1$  lying in $\G$ is
    \[
      p_{F_1}:=-k_1-\mathbbm 1_{2\beta-2\vareps-\theta \ge 0} -m_+(s_1;1, |r_1|).
    \]
  \item When $r_2\ge 0$, the number of poles of $1/F_2=1/P_2$  lying in $\G$ is
    \[
      p_{F_2}:=k_2-m_-(s_2;1, r_2-1).
    \]
  \item When $r_2\le 0$, the number of roots of $1/F_2=Q_2$   lying in $\G$ is
    \[
      r_{F_2}:=-k_2 - m_+(s_2;1, |r_2|).
    \]
  \end{itemize}
\end{lem}

If $q$ is not a root of unity, then all the numbers $m_\pm$ occurring in this lemma   equal $0$ or $1$.  Otherwise, we may choose the $r_i$'s so as to minimize $|r_i|$, and then this property still holds.

\begin{proof}
  We  apply Lemma~\ref{lem:count}, with $\om_1$ and $\om_2$ given by~\eqref{eq:def_r_k-double1} and~\eqref{eq:def_r_k-double2}. We use the following four identities. The first two  have already been justified in the proof of Lemma~\ref{lem:poles_roots_F}, and the other two relie on the fact that $\frac{e_i\beta}{2\pi} + \frac{1 -\eps_i}{2} \in (0, 1)$ for $i=1,2$ (because $0<\beta<\pi$ and $\{e_i,\eps_i\}\subset\{0,1\}$):
  \begin{align*}
    \left\lfloor \frac{\om_1}{2\pi} -\frac 1 2 \right\rfloor &= \left \lfloor \frac{2\beta -2\vareps-\theta}{2\pi}\right\rfloor
                                                               = - \mathbbm 1_{2\beta-2\vareps-\theta <0} &\text{as in~\eqref{int-part1}},
    \\
    \left\lfloor \frac{\om_2}{2\pi} -\frac 1 2 \right\rfloor
                                                             & = \left \lfloor \frac{2\delta -\theta}{2\pi}-1\right\rfloor = - 1  &\text{as in~\eqref{int-part2}},\\
    \left\lfloor \frac{\om_1}{2\pi} -\frac 1 2 -r_1\frac \beta \pi\right\rfloor
                                                             &  =\left \lfloor -k_1 -1+  \frac{e_1\beta}{2\pi} + \frac{1-\eps_1} 2\right\rfloor =
                                                               -k_1 - 1,\\
    \left\lfloor \frac{\om_2}{2\pi} -\frac 1 2 -r_2\frac \beta \pi\right\rfloor
                                                             &   =\left \lfloor-k_2 -1 +  \frac{e_2\beta}{2\pi} + \frac{1-\eps_2} 2\right\rfloor =
                                                               -k_2 - 1.
  \end{align*}
  This gives the announced formulas. We have used the fact that $s_2\not = -1$ (see Lemma~\ref{lem:angles}) to replace $m_-(s_2;0, r_2-1)$ by $m_-(s_2;1, r_2-1)$ in the expression of $p_{F_2}$.
\end{proof}

\subsection{Expression of $\boldsymbol{\phi_1}$}
\label{subsec:proofmainthm-double} 
We can now describe precisely the fraction $S/R$
such that $F \phi_1^2 = \frac S R \circ w$. In order to avoid 
handling four different cases depending on the signs of $r_1$ and $r_2$, we will use a compact form. In the case where $\alpha \in \Z + \pi/\beta \Z$, we can indeed make the expression of $\phi_1$ given in Theorem~\ref{thm:main} more compact by writing
\[
  \phi_1 (y) =  \frac {Q(y)}{P(y)} \frac{S(w(y))}{R(w(y))},
\]
where $P=P_{r, s_1}$, $Q=Q_{r, s_1}$, and
\[
  \text{  either }  \left( P=1   \text{ and } S=1 \right)
  \text{ or } \left(Q=1  \text{ and } R=1\right).
\]
When $\alpha_1, \alpha_2 \in \Z + \pi/\beta \Z$, we will express $\phi_1$ in an analogous compact form.
First, for $i=1, 2$ we use again the notation~\eqref{F-shortcut}, as well as $f_i = f_{e_i, \eps_i}$.
We further denote
\beq\label{apm-def}
a^-= {e_1\eps_1-e_2\eps_2}\qquad \text{and } \quad  a^+= {e_1(1-\eps_1)-e_2(1-\eps_2)} ,
\eeq
so that
\beq\label{f-ratio}
\frac{  f_1(y)}{f_2(y)} = (y-y^-)^{a^-}(y-y^+)^{a^+}.
\eeq
Also, let
\beq\label{b-def}
b= \eps_1(1-e_1)-\eps_2(1-e_2).
\eeq
Observe that $a^+$, $a^-$ and $b$ take their values in $\{-1, 0,1\}$.

We finally introduce four polynomials denoted $R_i$ and $S_i$, for $i=1,2$, which  we define up to a constant factor by giving the list of their roots.   The values of their degrees  easily follow from Lemma~\ref{lem:poles_roots_F-double}, as will be established in the proof of Theorem~\ref{thm:main-double} below.
\begin{itemize}
\item If   $r_1 \le 0$,  we take  $R_1$ to be a polynomial of degree $-k_1$ whose roots (taken with multiplicity) are
  \begin{itemize}
  \item 
    the $  w(y(s_1q^{j}))$,  for $j=1,\ldots , |r_1|$
    such that  $y(s_1q^{j}) \in  {\G} $,
  \item plus $w(y(s_1)) $ if  $2\beta-2\vareps-\theta >0$,
  \item   plus $w(y(-1))=-1$ with multiplicity $m_+(s_1;0, |r_1|)$.
  \end{itemize}     
  If $r_1 > 0$, we take $R_1$ to be constant.
\item If  $r_1 > 0$,  we take $S_1$ to be a polynomial of degree $k_1$ whose roots are
  \begin{itemize}
  \item 
    the $     w(y(s_1q^{-j}))$,  for  $j=1,\ldots , r_1-1$
    such that  $y(s_1q^{-j}) \in  {\G}$,
  \item plus $w(y(s_1))$ if $y(s_1)<y(-1)$ and $2\beta-2\vareps-\theta \le 0$,
  \item   plus $w(y(-1))=-1$ with multiplicity $m_-(s_1;1, r_1-1)$.
  \end{itemize}
  If $r_1 \le  0$, we take  $S_1$ to be constant. 

\item If   $r_2 \le 0$,  we take $ R_2$ to be a polynomial of degree $-k_2$ whose roots are
  \begin{itemize}
  \item 
    the $     w(y(s_2q^{j}))$, { for }  $j=1,\ldots , |r_2|$, { such that } $y(s_2q^{j}) \in  {\G} $,
  \item   plus $w(y(-1))=-1$ with multiplicity $m_+(s_2;1, |r_2|)$.
  \end{itemize}
  If   $r_2 > 0$, we take  $R_2$ to be constant. 
  

\item  If  $r_2 > 0$,  we take $S_2$ to be a polynomial of degree $k_2 $ whose roots  are
  \begin{itemize}
  \item 
    the
    $     w(y(s_2q^{-j}))$  for  $ j=0,\ldots , r_2-1$,  such that $ y(s_2q^{-j}) \in  {\G}$,
  \item   plus $w(y(-1))=-1$ with multiplicity $m_-(s_2;1,r_2-1)$.
  \end{itemize}
  If    $r_2 \le  0$, we take  $S_2$ to be constant.
\end{itemize}

\begin{thm}  
  \label{thm:main-double}
  Let us assume that Condition~\eqref{eq:CNS0double} holds,
  that is, $\alpha_1, \alpha_2 \in \Z+ \pi\Z/\beta$.
  Let the integers $r_i$, $k_i$, $e_i$, $\eps_i$, for $i=1,2$,  satisfy~\eqref{eq:def_r_k-double1} and~\eqref{eq:def_r_k-double2}.
  The $2$-decoupling function of Theorem~\ref{thm:decoupling} reads:
  \[
    F=\left( \frac{P_{1}}{Q_{1}}\cdot \frac{Q_{2}}{P_{2}}\right)^2 \cdot \frac{f_{1}}{f_{2}},
  \]
  and, if $\phi_1$ denotes  the Laplace transform defined   by~\eqref{eq:Laplace_transform_boundary}, the function $F \phi_1^2$  is an invariant.

  The function $\phi_1$ can be meromorphically continued to $\C\setminus [y^+,\infty)$. Moreover, the multiplicative constants in the polynomials $R_i$ and $S_i$ defined above     can be chosen so that
  \beq\label{phi-double-expr}
  \phi_1(y)=  \left(\frac{Q_{1}}{P_{1}}\cdot \frac{P_{2}}{Q_{2}}\right)\! (y)\cdot
  \left(\frac{S_{1}}{R_{1}}\cdot \frac{R_{2}}{S_{2}}\right)\! \left( w(y)\right)
  \cdot \sqrt{\frac{1-w(y)}{y-y^-}}^{\, a^-} \! \!
  \cdot      \frac {{\sqrt{1+w(y)}}^{\, b}} {\sqrt{y^+-y}^{\ a^+}},
  \eeq
  where $a^+$, $a^-$ and $b$ are defined by~\eqref{apm-def} and~\eqref{b-def} and take their values in $\{-1, 0, 1\}$.

  The function $\phi_1$ is always D-algebraic. It is D-finite if
  \beq\label{DF-double}
  \beta/\pi \in \Q \quad \text{or} \quad   \{ \alpha_1, \alpha_2\} \subset  \Z \cup \left(-\N+ \frac \pi\beta \Z\right),
  \eeq
  algebraic if
  \beq\label{alg-double}
  \beta/\pi \in \Q \quad \text{or } \quad  \{\alpha_1, \alpha_2\} \subset \Z,
  \eeq
  and rational if $\alpha\in -\N_0$. Moreover, when $\{\alpha_1, \alpha_2\} \subset \Z$, then $\phi_1^2$ is actually rational.

  If the simple angle condition~\eqref{eq:CNS0} does \emm not, hold, that is, $s_1/s_2 \not \in q^\Z$, Condition~\eqref{DF-double} (resp.~\eqref{alg-double}) is also necessary for $\phi_1$ to be D-finite (resp.\ algebraic), and moreover $\phi_1$ is never rational.
\end{thm}

Several explicit examples are worked out in Section~\ref{sec:examples-double}.

\begin{proof}[Proof of Theorem~\ref{thm:main-double}]
  We consider each element of the decoupling function $F$ separately, starting with $P_2$ and $Q_2$.
  We remind the reader of the preliminary observation made at the beginning of the proof of Theorem~\ref{thm:main}.
  
  $\bullet$  When $r_2> 0$,  the fraction $1/F_2=1/P_2$ has $p_{F_2}$ poles in $\G$, where $p_{F_2}$ is given by Lemma~\ref{lem:poles_roots_F-double}. Moreover, the multiplicity of $y(-1)$ as a pole of $1/P_2$ is
  \begin{multline*}
    \sharp \left\{  j : 0\le j <r_2, s_2 q^{-j} \in \{-1, -q\} \right\}\\
    \begin{split}
      &    =    \sharp \left\{  j : 0\le j <r_2, s_2 q^{-j} =-1\right\}
      +\sharp \left\{  j : 1\le j \le r_2, s_2 q^{-j} =-1\right\}\\
      &  = 2 \sharp \left\{  j : 1\le j <r_2, s_2 q^{-j} =-1\right\} + \mathbbm 1_{s_2 q^{-r_2}=-1} \\
      &   =2 m_{-}(s_2,1,r_2-1) + \mathbbm 1_{s_2 q^{-r_2}=-1},
    \end{split}  \end{multline*}    
  since $s_2$ never equals $-1$ by Lemma~\ref{lem:angles}. (Recall that $m_-(s_2;a,b)=0$ if  $a>b$.) This leads us to introduce the polynomial $S_2$ defined above. 
  \[
    p_{F_2}+ m_-(s_2;1, r_2-1)    =  k_2.
  \]
  By construction, $S_2(w(y))/P_2(y)$ has at most one pole in $\bG$, namely a simple pole at $y(-1)$ if $s_2 q^{-r_2}=(-1)^{\eps_2} \sqrt q^{\, e_2}=-1$, or equivalently if $\eps_2(1-e_2)=1$ (recall that $q\not = 1$).
  Consequently,
  \beq\label{nopoleP2}
  \left(\frac {S_2(w(y))}{P_2(y)} \right)^2 (w(y)+1)^{\eps_2(1-e_2)}
  \eeq
  has no pole in $\bG$.

  $\bullet$   When $r_2\le  0$,  the polynomial  $1/F_2=Q_2$ has $r_{F_2}$ roots in $\G$, where $r_{F_2}$ is given by Lemma~\ref{lem:poles_roots_F-double}. Moreover, the multiplicity of $y(-1)$ as a root of $Q_2$ is
  \begin{multline*}
    \sharp \left\{  j : 1\le j \le |r_2|, s_2 q^{j} \in \{-1, -q\} \right\}\\
    \begin{split}
      &   =    \sharp \left\{  j : 1\le j \le |r_2|, s_2 q^{j} =-1\right\}
      +\sharp \left\{  j : 0\le j < |r_2|, s_2 q^{j} =-1\right\}\\
      &   = 2 \sharp \left\{  j : 1\le j \le |r_2|, s_2 q^{j} =-1\right\} - \mathbbm 1_{s_2 q^{-r_2}=-1}\\
      &   =2 m_{+}(s_2,1,|r_2|)-\mathbbm 1_{s_2 q^{-r_2}=-1},
    \end{split}  \end{multline*}
  since $s_2$ never equals $-1$ by Lemma~\ref{lem:angles}.
  This leads us to introduce the polynomial $R_2$ defined above. Its degree is
  \[
    r_{F_2}+ m_+(s_2;1, |r_2|)    =  -k_2. 
  \]
  By construction, $Q_2(y) /R_2(w(y))$ has at most one pole in $\bG$, namely a simple pole at $y(-1)$ if $s_2 q^{-r_2}=-1$, or equivalently if $\eps_2(1-e_2)=1$. Consequently,
  \beq\label{nopoleQ2} 
  \left(\frac{Q_2(y)} {R_2(w(y))}\right)^2 (w(y)+1)^{\eps_2(1-e_2)}
  \eeq
  has no pole in $\bG$.

  $\bullet$  When $r_1>  0$,  the polynomial $F_1=P_1$ has $r_{F_1}$ roots in $\G$, where $r_{F_1}$ is given by Lemma~\ref{lem:poles_roots_F-double}. If  $y(s_1)<y(-1)$, one of them is $y(s_1)$, which cancels with the pole of $\phi_1$ at this point when $2\beta-2\vareps-\theta >0$. Moreover, the multiplicity of $y(-1)$ as a root of $P_1$ is 
  \[ 
    \sharp \left\{  j : 0\le j <r_1, s_1 q^{-j} \in \{-1, -q\} \right\}
    = 2 \sharp \left\{  j : 1\le j < r_1, s_1 q^{-j} =-1\right\} +\mathbbm 1_{s_1=-1}+ \mathbbm 1_{s_1 q^{-r_1}=-1}.
  \] 
  Recall that $\phi_1$ has a (simple) pole at $y(-1)$ if $s_1=-1$. This leads us to introduce the polynomial~$S_1$ defined above. Its degree is
  \[
    r_{F_1} - \mathbbm 1_{2\beta-2\vareps-\theta>0} +m_-(s_1;1,r_1-1)
    =k_1.
  \]
  By construction, $P_1(y)\phi_1(y)/S_1(w(y))$ has no pole in $\bG$, but has  a simple root  at $y(-1)$ if $s_1 q^{-r_1}=-1$, or equivalently if $\eps_1(1-e_1)=1$.   Consequently,
  \beq\label{nopoleP1} 
  \left(\frac{P_1(y)\phi_1(y)} {S_1(w(y))} \right)^2 \frac 1 {(w(y)+1)^{\eps_1(1-e_1)}}
  \eeq
  has no pole in $\bG$.

  $\bullet$  When $r_1\le  0$,  the fraction  $F_1=1/Q_1$ has $p_{F_1}$ poles in $\G$, where $p_{F_1}$ is given by Lemma~\ref{lem:poles_roots_F-double}. Recall that $\phi_1$ also has a pole in $\G$, located at $y(s_1)$,  if $2\beta-2\vareps-
  \theta >0$. Moreover, the multiplicity of $y(-1)$ as a pole  of $1/Q_1$ is
  \[ 
    \sharp \left\{  j : 1\le j \le |r_1|, s_1 q^{j} \in \{-1, -q\} \right\}
    = 2 \sharp \left\{  j : 0\le j \le |r_1|, s_1 q^{j} =-1\right\}-\mathbbm 1_{s_1=-1} - \mathbbm 1_{s_1 q^{-r_1}=-1}.
  \] 
  Recall that $\phi_1$ has a pole at $y(-1)$ if $s_1=-1$.  This leads
  us to introduce the polynomial $R_1$ defined above. Its degree is 
  \[
    p_{F_1}+\mathbbm 1_{2\beta-2\vareps-\theta >0}+ m_+(s_1;0,|r_1|)
    =-k_1.
  \]
  By construction, $\phi_1(y) R_1(w(y))/Q_1(y) $ has no pole in $\bG$, but has a  simple zero  at $y(-1)$ if $s_1 q^{-r_1}=-1$, or equivalently if $\eps_1(1-e_1)=1$. Consequently,
  \beq\label{nopoleQ1}  
  \left(\frac {\phi_1 (y) R_1(w(y))}{Q_1(y)}\right)^2 \frac 1 {(w(y)+1)^{\eps_1(1-e_1)}}
  \eeq
  has no pole in $\bG$.

  We have now constructed polynomials $R_i$ and $S_i$ from $\phi_1$, the $P_i$'s and the $Q_i$'s. We still need to investigate the term $f_1/f_2$, given by~\eqref{f-ratio}. Recall that $y^-$ lies in $\G$, but not $y^+$. Indeed, $y(-1)\le y^+$ by Lemma~\ref{lem:Ypositive}, and we cannot have $y(-1)=y^+=y(\sqrt q)$ because the only values~$s$ such that $y(s)=y(-1)$ are $-1$ and $-q$, which are both distinct from $\sqrt q$.
  Moreover, it follows from~\eqref{wys0} that $w(y^-)=w(y(-e^{i\beta}))=1$. This leads us to include a factor $(1-z)^{a^-}$ in $S/R$. By construction,
  \beq\label{nopolef}
  \frac{f_1}{f_2}(y)\cdot \frac 1 {(1-w(y))^{a^-}}
  \eeq
  has no pole in $\bG$.

  So let us now define the rational function
  \[
    \frac S R (z) := \left( \frac {S_1}{R_1} (z) \frac {R_2}{S_2}(z) \right)^2
    (1-z)^{a^-} (1+z)^b
  \]
  where $b= \eps_1(1-e_1)-\eps_2(1-e_2)$.
  It follows from  the fact that the functions~\eqref{nopoleP2},~\eqref{nopoleQ2},~\eqref{nopoleP1},~\eqref{nopoleQ1} and~\eqref{nopolef} have no pole in $\bG$ that
  $F\phi_1^2 \times (R/S)(w)$, which is an invariant, has no pole in $\bG$ either.
  But the behaviour at infinity of  $(S/R)(z)$   is in $z^d$, where 
  \[
    d=2k_1 -2k_2 +a^-+b= 2k_1-2k_2+\eps_1-\eps_2.
  \]
  Since $w(y)$ behaves at infinity in $y^{\pi/\beta}$,  Lemma~\ref{lem:phi-infini-double} implies that  $F\phi_1^2 \times (R/S)(w)$    has  a finite limit at infinity. It is thus constant by Proposition~\ref{prop:inv}, and we have obtained an explicit expression for $\phi_1^2$.
  Recall finally that both $\sqrt{1+w(y)}$ and $\sqrt{(1-w(y))/(y-y^-)}$ are defined analytically on $\C\setminus [y^+,\infty)$ (see~\eqref{sqrt_plus},~\eqref{sqrt_minus} and the definition~\eqref{eq:definition_w-bis} of $w$ in terms of $T_{\pi/\beta}$). The announced expression of $\phi_1$ follows, and  $\phi_1$ is  meromorphic on $\C\setminus [y^+,\infty)$.

  \medskip

  Let us now discuss the differential/algebraic nature of $\phi_1$.
  First, $\phi_{1}$ is D-algebraic, as compositions of D-algebraic functions are D-algebraic.

  \medskip
  
  \noindent {\bf D-finiteness.} Let us first assume that $\phi_1$ is D-finite.
  Then its square is D-finite as well, or equivalently, 
  \beq\label{phi1-w}
  \left(\frac {S_1}{R_1} \frac {R_2} {S_2}\right)^2\!( w) \cdot
  \left({1-w}\right)^{a^-} \left({1+w}\right)^b
  \eeq
  is D-finite.  By Proposition~\ref{prop:algebraic_nature_w},  either $\beta/\pi\in \Q$ (in which case $\phi_1$ is in fact algebraic), or  the rational function 
  \beq\label{polynomial}
  \left(\frac {S_1}{R_1} \frac {R_2} {S_2}\right)^2\!( z) \cdot
  \left({1-z}\right)^{a^-} \left({1+z}\right)^b
  \eeq
  is a polynomial.

  Let us now assume that $s_1/s_2\not \in q^\Z$, and prove that the latter condition implies the second part of~\eqref{DF-double}. Recall that one of $R_1$ and $S_1$ (resp.\ $R_2$ and $S_2$) is always a constant, and observe that all roots of $R_1$ and $S_1$ (resp.\ $R_2$ and $S_2$) are of the form $w(y(s_1 q^j))$ (resp.\ $w(y(s_2 q^j))$) for some integer $j$ such that $y(s_1q^j)\in \G \cup\{y(-1)\}$. By  Lemma~\ref{lem:interferences}, $y(s_1q^i) \not = y(s_2q^j)$ for all integers $i,j$, and since  $w$ is injective on   $\G \cup\{y(-1)\}$ (Lemma~\ref{lem:w}),  
  we conclude that $S_1 R_2$ and $R_1 S_2$ have no common root. Since $a^-$ and $b$ are at most~$1$, while $R_1 S_2$ is squared in~\eqref{polynomial}, we conclude that if~\eqref{polynomial} is a polynomial, then
  \begin{enumerate}[label={\rm(\roman*)},ref={\rm(\roman*)}] 
  \item\label{it1:cond}$R_1$ and $S_2$ are  constants,
  \item\label{it2:cond}if $a^-=-1$    then $1$ is    a root of $S_1 R_2$,
  \item\label{it3:cond}if $b=-1$    then $-1$ is     a root of $S_1 R_2$.
  \end{enumerate}
  By definition of $R_1$ and $S_2$,  the  conditions~\ref{it1:cond} read respectively:
  \beq\label{cond-DF-1}
  r_1>0 \text{ or }\left( r_1 \le 0 \text{ and } k_1=0 \right),
  \eeq
  and
  \beq\label{cond-DF-2}
  r_2\le 0      \text{ or }\left( r_2 >0   \text{ and } k_2=0 \right).
  \eeq
  By Lemma~\ref{lem:ineq-double}, if $r_1\le 0$ and $k_1=0$ then $\eps_1=0$. 
  Thus Condition~\eqref{cond-DF-1} can be rewritten as $ r_1>0 \text{ or }\left( r_1 \le 0 \text{ and } k_1=\eps_1=0 \right)$, or in simpler terms $ (r_1>0 \text{ or } k_1=\eps_1=0 )$,  which, according to~\eqref{alpha12-expr}, translates into $\alpha_1\in -\N+ \pi/\beta \Z$ or $\alpha_1\in \Z$. 

  We will now combine~\eqref{cond-DF-2} with the conditions~\ref{it2:cond} and~\ref{it3:cond} to prove that the same holds for~$\alpha_2$, which means that $(r_2<0 \text{ or } k_2=\eps_2=0)$.

  If $a^-=-1$, so that $\eps_2=e_2=1$ or equivalently $s_2=-q^{r_2}\sqrt q $, we want one of the roots of $S_1 R_2$ to be $1=w(y^-)=w(y(-\sqrt q))=w(y(s_2q^{-r_2}))$. By Lemma~\ref{lem:interferences}, and the injectivity of $w$ on $\G$, this value cannot be a  root of $S_1$.  The description of $R_2$ shows that it admits $1$ as a root   if and only if $r_2<0$.

  Similarly, if $b=-1$, so that $\eps_2=1$ and $e_2=0$ or equivalently $s_2=-q^{r_2}$, we want $-1=w(y(-1))=w(y(s_2 q^{-r_2}))$ to be a root of $S_1 R_2$. Again, Lemma~\ref{lem:interferences} and the injectivity of $w$ on $\G\cup\{y(-1)\}$ prevent $-1$ to be a root of $S_1$. Moreover, $-1$ will be a root of $R_2$ if and only if $r_2<0$.

  Hence, it follows from Conditions~\ref{it2:cond} and~\ref{it3:cond} that if $\eps_2=1$, then  $r_2<0$. Thus we can now complete~\eqref{cond-DF-2} into
  \[
    r_2< 0    \text{ or }\left(r_2=0 \text{ and }  \eps_2=0\right)     \text{ or }\left( r_2 >0   \text{ and } k_2=\eps_2=0 \right).
  \]
  By Lemma~\ref{lem:ineq-double}, if $r_2=0$ then $k_2=0$.
  Hence we can summarize the above conditions into $(r_2<0 \text{ or } k_2=\eps_2=0)$, that is, $\alpha_2\in -\N+ \pi/\beta \Z$ or $\alpha_2\in \Z$, as claimed in the theorem.

  \medskip
  Conversely,  assume that  Condition~\eqref{DF-double} holds. If $\beta/\pi \in \Q$, then $\phi_1$ is algebraic. Otherwise,
  \[
    \left( r_1 >0 \text{ or } k_1=\eps_1=0\right) \text{ and }
    \left( r_2 < 0 \text{ or } k_2=\eps_2=0\right).
  \]
  Then we can check that the above conditions~\ref{it1:cond},~\ref{it2:cond} and~\ref{it3:cond}  hold, so that~\eqref{polynomial} is a polynomial. Moreover, all its roots have an even multiplicity, with the possible exceptions of $-1$ and $1$. But $\sqrt{1\pm w}$ is D-finite (Proposition~\ref{prop:algebraic_nature_w}).
  Hence
  \[
    \left(\frac {S_1}{R_1} \frac {R_2} {S_2}\right) \!(w) \sqrt{1-w}^{\,a^-} \sqrt{1+w}^{\,b}
  \]
  is D-finite, and $\phi_1$ is D-finite as well. 

  \medskip   
  \noindent{\bf Algebraicity.} Let us first assume that $\phi_1$, or equivalently $\phi_1^2$, or Expression~\eqref{phi1-w}, is algebraic. Then either $w$ is algebraic, which means that $\beta/\pi\in \Q$ by  Proposition~\ref{prop:algebraic_nature_w}, or the fraction~\eqref{polynomial}    is in fact a constant. 

  Let us now assume that $s_1/s_2\not \in q^\Z$, and prove that the latter condition implies the second part of~\eqref{alg-double}.  As already argued,  $S_1R_2$ and $R_1S_2$  have no common root. This  forces the $R_i$'s and $S_i$'s to be constants. That is, $k_1=k_2=0$. Moreover,   $a^-$ and  $b$ must be zero. By definition of $a^-$ and $b$, this implies in particular that    $\eps_1=\eps_2$. If  $\eps_1=\eps_2=1$, this forces moreover $e_1=e_2$, but then $s_1/s_2 =q^{r_1-r_2} \in q^\Z$, a contradiction.  Hence $\eps_1=\eps_2=k_1=k_2=0$, so that $\alpha_1$ and $\alpha_2$ are both in $\Z$.

  Conversely, assume that Condition~\eqref{alg-double} holds. If $\beta/\pi\in \Q$, then $w$ is algebraic and so is $\phi_1$. If $\alpha_1$ and $\alpha_2$ are integers, that is,  $\eps_1=\eps_2=k_1=k_2=0$, then the $R_i$ and $S_i$ are constants and  $a^-=b=0$. Hence~\eqref{polynomial} is a constant, $\phi_1^2(y)$ is a rational function in $y$, and  $\phi_1$ is algebraic (of degree $2$ at most).

  \medskip
  \noindent{\bf Rationality.} If $\phi_1$ is rational, then~\eqref{phi-asympt} implies that $\alpha$ is an integer.
  But then the simple angle condition~\eqref{eq:CNS0} holds. Conversely, we have already seen that if $\alpha$ is an integer then $\phi_1$ is rational.
\end{proof}

\subsection{Examples}\label{sec:examples-double}

We now give three applications of Theorem~\ref{thm:main-double}, focussing on cases where $\phi_1$ is algebraic, even if $\beta$ is not a rational multiple of $\pi$.  
Indeed, we have seen that $\phi_1$ is algebraic if $\alpha_1$ and $\alpha_2$ are both integers, that is, if we can choose $k_i=\eps_i=0$ for $i=1,2$. We focus here on this case, where, according to~\eqref{alpha12-expr}:
\beq\label{alpha-alg-double}
1-\alpha_1=   2r_1+e_1, \qquad   \alpha_2=2r_2+e_2. 
\eeq
We assume that $\alpha_1$ and $\alpha_2$ are equal modulo $2$, that is, $e_1\not=e_2$, otherwise $\alpha= (\alpha_1+\alpha_2+1)/2$ is an integer and we are in the rational case studied in Section~\ref{sec:expression}. Hence $\alpha$ is here a half-integer. The condition $\alpha<1$ translates into $2r_2+e_2<2r_1+e_1$. As seen in the proof of Theorem~\ref{thm:main-double}, the expression of $\phi_1$ does not involve $w(y)$ and reduces to
\beq\label{phi1-alg}
\phi_1(y)=  \left(\frac{Q_{1}}{P_{1}}\cdot \frac{P_{2}}{Q_{2}}\right)\! (y)
\cdot   \sqrt{y^+-y}^{\, e_2-e_1},
\eeq
because $a^-=b=0$ and $a^+=e_1-e_2$. Recall that one of the polynomials $P_1$ and $Q_1$ (resp.~$P_2$ and $Q_2$) is always a constant, and that the degree of $P_i$ (resp.\ $Q_i$) is $\max(0, r_i)$ (resp.~$\max(0, -r_i)$). Observe that if $r_i<0$, then $Q_i$ contains a factor $(y-y(s_iq^{-r_i}))$, which equals $(y-y(\sqrt q))=(y-y^+)$ if $e_i=1$ (see~\eqref{eq:uniformization}). 

\smallskip
Let us be more explicit in four cases where the $|\alpha_i|$ and $|r_i|$ are small, so that the polynomials $P_i$ and $Q_i$ have small degrees. We first study in detail the case $\alpha_1=\alpha_2=0$, for which we obtain
an explicit expression of the density of the two-dimensional stationary distribution of the SRBM in the quadrant. The other three  cases are less detailed. By definition of $\alpha_1$ and $\alpha_2$, fixing these two values means prescribing two relations between the four angles $\beta$, $\theta$, $\delta$ and~$\vareps$; see~\eqref{eq:def_alpha_i}.

\subsubsection{The case $\alpha_1=\alpha_2=0$}
In this case $\alpha=1/2$, and the four angles are related by
\beq\label{angles-00}
\theta= 2\delta-\pi, \qquad \beta-\theta=2\vareps-\pi.
\eeq
Hence we can express all quantities either in terms of $\delta$ and $\vareps$, or in terms of $\theta$ and $\beta$. Moreover, the assumptions $\{\theta, \beta-\theta\} \subset (0, \pi)$ imply that
\beq\label{conds-new}
\{ \delta, \vareps\} \subset (\pi/2, \pi), \quad \text{with}\quad  \pi< \delta+\vareps <3\pi/2.
\eeq
In terms of the original parameters $\Sigma$, $R$, and $\mu$, we have
\[
  \mu_2\,   \frac{{r_{12}}}{{r_{22}}}={\mu_1}+{\sqrt{\frac{\Delta}{{\sigma_{22}}}}}
  \qquad \text{and} \qquad
  \mu_1\,  \frac{{r_{21}}}{{r_{11}}}={\mu_2}+{\sqrt{\frac{\Delta}{{\sigma_{11}}}}}.
\]
Indeed, let us explain (for instance) why the first identity is equivalent to $\theta=2\delta-\pi$. Using~\eqref{normal-sigma},~\eqref{normal-mu-sigma}, and~\eqref{normal-R-sigma}, this identity reads
\[
  -\sin \theta \sin (\beta-\delta) + \sin \delta \sin (\beta-\theta) = \sin \beta \sin \delta,
\]
or equivalently,
\[
  \sin \beta\left(\sin (\delta-\theta) -\sin \delta \right) =0,
\]
or finally $\sin (\delta-\theta) =\sin \delta=\sin (\pi-\delta)$. Given that $0<\theta <\delta <\pi$, this implies that either $\delta-\theta= \pi -\delta$ (which is the desired identity), or that $\delta-\theta=
\delta$, which is impossible.

In sight of~\eqref{alpha-alg-double}, we have $r_1=r_2=0$, $e_1=1$ and $e_2=0$. That is, $s_1=\sqrt q$ and $s_2=1$.  Then $P_1=Q_1=P_2=Q_2=1$ (see~\eqref{Fr-def}), and   there exists a constant $\kappa$ such that
\beq\label{phi1-very-simple}
\phi_1(y)=  \frac {\kappa} {\sqrt{y^+-y}} = \frac{\phi_1(0)}{\sqrt{1-y/y^+}}.
\eeq
It follows that the measure $\nu_1/\phi_1(0)$ is a Gamma distribution of parameters $1/2$ and $1/y^+$, with density
\[
  \frac{p_1(z_2)}{\phi_1(0)}=   \sqrt{  \frac {y^+} { \pi}}\cdot \frac{e^{-z_2 y^+}}{\sqrt {z_2}}
\]
for $z_2\ge 0$.
Since an $x/y$-symmetry in the quadrant model exchanges $\alpha_1$ and $\alpha_2$, we also have
\beq\label{phi2-very-simple}
\phi_2(x)=  \frac {\kappa'} {\sqrt{x^+-x}} = \frac{\phi_2(0)}{\sqrt{1-x/x^+}}.
\eeq
Using the functional equation~\eqref{eq:functional_equation},  one can obtain
the following algebraic expression for the bivariate Laplace transform:
\beq\label{phi-alg}
\phi(x,y)=\frac{\kappa_0\,(\tx +\ty )}{\tx \ty 
  (\tx ^2\sin^2\delta+\ty ^2\sin^2\vareps-2\tx \ty \sin\delta\sin \vareps \cos(\delta+\vareps)-\sin^2(\delta+\vareps))},
\eeq
where $\tx=\sqrt{1-x/x^+}$ and $\ty =\sqrt{1-y/y^+}$, and
\[
  \kappa_0=-2\sin \delta \sin \vareps \cos(\delta+\vareps) >0.
\]
This Laplace transform can be inverted explicitly to obtain the
density $p_0(z_1,z_2)$ of the stationary distribution.  The form of
$\phi(x,y)$, plus the use of normal variables, leads us to consider
\beq\label{tilde-p0}
\widetilde p_0(z_1,z_2):= p_0\left( \frac{\det \Sigma}{\sqrt{\Delta
      \sigma_{22}}} \, z_1, \frac{\det \Sigma}{\sqrt{\Delta
      \sigma_{11}}} \, z_2\right)
=p_0\left( 2\sin^2\!\delta \, \frac{z_1}{x^+}, 2\sin^2\!\vareps
  \,  \frac{z_2}{y^+}\right),
\eeq
where we recall that $\Delta$ is defined by~\eqref{Delta-def}. We have used~\eqref{xyn-pm} and~\eqref{angles-00} to obtain the second expression above.

\begin{prop}
  \label{prop:specialcase}
  When $\alpha_1=\alpha_2=0$, the density   of the stationary distribution of the SRBM in the quadrant satisfies
  \begin{equation}
    \label{eq:pi}
    \widetilde p_0(z_1,z_2)=   \kappa \,
    \frac{\cos(\frac{\theta-a}2)}{\sqrt{|z|}}
    \exp\left(-2|z| \cos^2\left(\frac{\theta-a}2\right)\right),
  \end{equation}
  where $z= z_1 + z_2e^{i\beta}$ and $a=\arg z\in (0,\beta)$. The integration constant is
  \[
    \kappa = \frac{2\sqrt 2\Delta\sin \delta \sin \vareps}{\sqrt{\pi} (\det \Sigma)^{3/2} \sin(\beta/2)}.
  \]

  Equivalently, the stationary distribution of the SRBM in the corresponding $\beta$-wedge has density $q_0(u,v)$, where, for $\rho>0$ and $a\in (0, \beta)$,
  \beq\label{q0}
  q_0(\rho \cos a, \rho \sin a) = \kappa '  \,
  \frac{\cos(\frac{\theta-a}2)}{\sqrt{\rho}}
  \exp\left(-2
    |\widetilde{\mu}|
    \, \rho \cos^2\left(\frac{\theta-a}2\right)\right),
  \eeq
  with $|\tilde \mu|$  given by~\eqref{Delta-def} and
  \[
    \kappa'= \kappa \,\Delta^{-1/4} (\det \Sigma)^{3/4}
    =\frac{|2\widetilde{\mu}|^{3/2}\sin \delta \sin \vareps}{\sqrt{\pi}\sin\left({\beta}/{2}\right)}.
  \]
\end{prop}
In more explicit terms, the quantities occurring in~\eqref{eq:pi} are
\[  |z|=\sqrt{z_1^2+z_2^2 +2z_1 z_2 \cos \beta},  \quad\text{  and  }\quad   \cos\left(\frac{\theta-a}2\right)=\sqrt{\frac{z_1\cos\theta+z_2\cos(\beta-\theta) + |z|}{2|z|}}.
\]

The expression of $q_0$ follows from the expression of $p_0$ (and $ \widetilde p_0$) via an elementary calculation starting from~\eqref{phi-Phi}, or equivalently,
\beq
\label{eq:change_variable_q0_p0}
q_0(\rho \cos a, \rho \sin a)=\frac 1 {\det T} \ p_0\!\left((\rho \cos a, \rho \sin a)\, ^t(T^{-1})\right).
\eeq
In particular, in this calculation we evaluate $ \widetilde p_0$ at a point $(z_1, z_2)$ such that $z:=z_1+e^{i\beta }z_2= \sqrt{\frac \Delta{\det \Sigma}}\, \rho\, e^{ia}$.

In Appendix~\ref{app:alg}, we prove the above proposition by \emm checking, that the  Laplace transform of   density $p_0$ is indeed $\phi$. In a final appendix (Appendix~\ref{app:alg-inverse}), only available on the arXiv version of this paper~\cite{BmEFHR-arxiv}, we perform the (much longer) inverse computation: that is, we   explain  how to \emm derive, $p_0$ from $\phi$ by
inverse Laplace transformation. The same approach might lead to
explicit densities for the other algebraic cases that we briefly
describe below in Sections~\ref{sec:alg2} to~\ref{sec:alg4}.

\medskip 
\noindent{\bf Remarks}\\
1. The expression~\eqref{eq:pi} of the density has already been established by Harrison~\cite[Sec.~9]{harrison_78_diffusion} in the limit case  where the covariance matrix is the identity {(so that $\beta=\pi/2$), $\mu_2=0$ (so that $\theta=0$), and the reflection angles are $\delta=\pi/2$ and $\varepsilon= 3\pi/4$.
  Observe that $\alpha_1=\alpha_2=0$ indeed. Since we have assumed that $\mu_2<0$,  this limit case is not  covered by our paper, but Harrison's result shows that Proposition~\ref{prop:specialcase} still holds (there seems however to be a typo in~\cite[Prop.~8]{harrison_78_diffusion}, as the constant denoted $K$ therein should probably be divided by $2$).

\medskip

\noindent 2. When $(z_1,z_2)=(r\cos \om, r\sin\om)$, the above proposition gives
\[
  p_0(z_1,z_2)= \frac{c(\om)}{\sqrt r} \exp(-rb(\om)),
\]
for values $b(\om)$ and $c(\om)$ that depend on $\om$ only. This exact expression matches the \emm asymptotic, expression of $p_0(r \cos\om, r\sin \om)$ established, for $\om$ fixed and $r\rightarrow\infty$, in~\cite[Thm.~4(1)]{franceschi_asymptotic_2016}. For this reason, we may call the case $\alpha_1=\alpha_2=0$ a \emm pure asymptotic case,. In fact, we have originally guessed the value of $p_0$ in this case by starting from its asymptotic expression.

Let us now clarify the link between our expression of  the function $b(\om)$ and its characterization given in~\cite[Thm.~4(1)]{franceschi_asymptotic_2016}. It follows from the latter theorem that
\beq\label{rb-FK}
rb(\om)= \max_{u\in \RR}\left(  (z_1, z_2) \cdot \left(x(e^{iu}), y(e^{iu})\right)\right),
\eeq
where we recall that $x(\cdot)$ and $y(\cdot)$ are the coordinates of the uniformization of the kernel curve, given by~\eqref{eq:uniformization}. 
On the other hand, if we return to our expression of $p_0(z_1,z_2)$, written for $(z_1,z_2)= (\rho \cos a,\rho \sin a)\, ^t(T^{-1})$ as in~\eqref{eq:change_variable_q0_p0}, we find
\[
  r b(\om)=   2   |\widetilde{\mu}| \rho
  \, \cos^2\left(\frac{\theta-a}2\right)
  =\rho \sqrt{\frac{\Delta}{\det \Sigma}} (1+\cos(\theta-a)).
\]
But, using the definition~\eqref{eq:uniformization} of the parametrization of the kernel curve (written using the normal form~\eqref{param:normal}), and the expression~\eqref{sin-beta} of $\sin \beta$, we can rewrite the above expression as
\begin{align*}
  r b(\om)
  &= x(e^{i a}) \rho  \sqrt{\sigma_{11}}\sin (\beta-a)+ y(e^{i a}) \rho  \sqrt{\sigma_{22}}\sin a \\
  &=  (z_1, z_2) \cdot \left(x(e^{ia}), y(e^{ia})\right),
\end{align*}
because $(z_1,z_2)= (\rho \cos a,\rho \sin a)\, ^t(T^{-1})$ reads $(z_1,z_2)= \rho ( \sqrt{\sigma_{11}}\sin (\beta-a), \sqrt{\sigma_{22}}\sin a)$.  This means that the maximum in~\eqref{rb-FK} is reached at $u=a$. In~\cite[Sec.~4.2]{franceschi_asymptotic_2016}, the maximizing point  $(x(e^{i u}),y(e^{i u}))$ appears as a \textit{saddle point} in the  asymptotic estimates of two integrals.

\medskip

\noindent3. When $\mu\rightarrow (0,0)$, that is, $\Delta\rightarrow 0$, it follows from~\eqref{q0} that the density in the $\beta$-wedge satisfies
\[
  q_0(\rho \cos a, \rho \sin a) \sim \kappa '  \,
  \frac{\cos(\frac{\theta-a}2)}{\sqrt{\rho}}.
\]
This is in adequation with a result of Williams~\cite[Sec.~6]{williams_recurrence_1985}, which gives the density in the case $\mu=0$, $0\le \alpha <2$ in the form $\rho^{-\alpha} \cos(\theta_1-\alpha a )$, up to a multiplicative constant, with $\theta_1=\delta-\pi/2$. In our case, $\theta_1= \theta/2$ and $\alpha=1/2$. But the right-hand side of the above estimate is also, for $\mu\not = 0$, the behaviour of  $q_0(\rho \cos a, \rho \sin a)$ as $\rho\rightarrow 0$. Heuristically, this means that on small scales, the drift of the Brownian part of the process is negligible.

\subsubsection{The case $\alpha_1=-2$ and $\alpha_2=0$}
\label{sec:alg2}
In this case $\alpha=-1/2$, $r_1=e_1=1$ and $r_2=e_2=0$. That is, $s_1=q\sqrt q$ and $s_2=1$. Then $P_1(y)= (y-y(s_1))$ and $Q_1=P_2=Q_2=1$.  There exists a constant $\kappa$ such that 
\beq\label{cas-m2-0}
\phi_1(y)=  \frac \kappa {y-y(q\sqrt q)} \cdot  \frac 1{\sqrt{y^+-y}}.
\eeq
The value $\phi_2(x)$ can be obtained using  Theorem~\ref{thm:main-double}, applied to the symmetric case $\alpha_1=0$ and $\alpha_2=-2$. One finds:    
\beq\label{cas-0-m2}
\phi_2(x)=  \frac {\kappa'} {x-x(\sqrt q)} \cdot  \frac 1{\sqrt{x^+-x}}.
\eeq
(Note that $y(s)$ becomes $x(s/\sqrt q )$ under a diagonal reflection; see~\eqref{eq:uniformization}.)                       

The densities of $\nu_1$ and $\nu_2$ can be expressed in terms of the error function $\erf$, since, for $b>0$ and $y<a$,
\[
  \int_{\RR_+} \erf(\sqrt{bz}) e^{-az} e^{zy} dz= \frac{\sqrt b}{(a-y) \sqrt{a+b-y}}.
\]
For instance, it follows from~\eqref{cas-m2-0} that the density of $\nu_1$ is of the form
\beq\label{density-erf}
p_1(z_2)=\kappa \erf(\sqrt{bz_2}) e^{-az_2}
\eeq
with $a=y(q\sqrt q)$ and $b=y^+-y(q\sqrt q)= (y^+-y^-) \sin^2\beta$. 

\subsubsection{The case $\alpha_1=-1$ and $\alpha_2=1$}
\label{sec:alg3}
In this case, $\alpha=1/2$ again, $r_1=1$, $r_2=0$, $e_1=0$ and $e_2=1$. That is, $s_1=q$ and $s_2=\sqrt q$. Again $P_1(y)= (y-y(s_1))$ and $Q_1=P_2=Q_2=1$.  There exists a constant $\kappa$ such that

\[
  \phi_1(y)= \kappa \,\frac {\sqrt{y^+-y}} {y-y(q)}.
\]
Theorem~\ref{thm:main-double}, applied to the symmetric case $\alpha_1=1$ and $\alpha_2=-1$, gives
\[
  \phi_2(x)=  \frac {\kappa'} {\sqrt{x^+-x}}.
\]
As in the first case,  the measure $\nu_2/\phi_2(0)$ is a Gamma distribution of parameters $1/2$ and $1/x^+$.

\subsubsection{The case $\alpha_1=\alpha_2=-1$}
\label{sec:alg4}
We finish with an $x/y$-symmetric case where two of the polynomials $P_1$, $P_2$, $Q_1$, $Q_2$ are not trivial -- instead of one in the above examples. However, a simplification occurs between $Q_2(y)$ and the term $\sqrt{y^+-y}$, and the Laplace transforms of~$\nu_1$ and $\nu_2$ end up being very close to those of Section~\ref{sec:alg2}.

With $\alpha_1=\alpha_2=-1$, we have $\alpha=-1/2$, $r_1=1$, $r_2=-1$, $e_1=0$ and $e_2=1$. That is, $s_1=q$ and $s_2=1/\sqrt q$. Now $P_1(y)= (y-y(s_1))$, $Q_2(y)=(y-y(s_2 q))=(y-y^+)$ and $Q_1=P_2=1$.  Given that $e_2-e_1=1$, the general formula~\eqref{phi1-alg} specializes to
\[
  \phi_1(y)=\frac{\kappa}{y-y(q)} \cdot \frac 1 {\sqrt{y^+-y}}.
\]
By symmetry,
\[
  \phi_2(x)=\frac{\kappa '}{x-x(\sqrt{q})}\cdot \frac 1 {\sqrt{x^+-x}}.
\]
Observe the analogy with~\eqref{cas-m2-0} and~\eqref{cas-0-m2}. In particular, the densities of $\nu_1$ and $\nu_2$ are of the form~\eqref{density-erf} again.

\section{Differential transcendence}
\label{sec:DT}

The aim of this section is to complete the proof of Theorem~\ref{thm:main_diffalg}, which deals with the case $\beta/\pi \not \in \Q$. Comparing its statement with the conclusions of Theorems~\ref{thm:main} and~\ref{thm:main-double}, we see that the only  point that remains to be proven is the following:
\begin{quote}
  if neither of the angle conditions~\eqref{eq:CNS0} or~\eqref{eq:CNS0double} holds, the Laplace transform~$\phi_1$ is not D-algebraic.
\end{quote}
To prove this, a key tool is difference Galois theory.
This theory builds a dictionary between the algebraic   relations satisfied by the solutions of a linear difference equation and the algebraic dependencies among the coefficients of the difference equation.
We refer to~\cite{vdps} for a complete introduction.
This theory also  has applications to the study of the \emm
differential, properties of the solutions, which is what we use here.

Our strategy  will  be first to  transform  the  boundary value condition~\eqref{eq:bound_cond_gen} into  a  finite difference equation (Section~\ref{subsec:funcequ}; see~\eqref{eq:funceqomegaplane}) and then to apply a Galoisian criterion for differential transcendence (Section~\ref{subsec:difftrans}). 
This criterion is presented in Section~\ref{subsec:difftranscriteria}.

\subsection{Galoisian criteria for differential transcendence  }
\label{subsec:difftranscriteria}
The classical difference Galois theory studies algebraic relations
between solutions  $g_0,\ldots,g_n$ of linear difference equations of the form 
\beq\label{eq:shift}
\sigma(g_i)=g_i +b_i,
\eeq
for $0\le i \le n$, where the coefficients $b_i$ lie in a  field $K$ endowed with an automorphism $\sigma$.
In particular, a theorem due to  Ostrowski (in the context of
differential equations rather than difference equations~\cite{Ostrowski}) gives a necessary and sufficient condition for the algebraic independence  of $g_0,\ldots,g_n$ over $K$ in terms of algebraic relations satisfied by the coefficients~$b_i$.

This setting allows one to study as well the \emm differential,
algebraicity of a function $g$ satisfying $\sigma(g)=g+b$, provided
that the derivation (denoted $\partial$) commutes with
$\sigma$. Indeed, the functions $g_i=\partial ^i g$ then satisfy a
system of the form~\eqref{eq:shift}, with $b_i=\partial ^i b$, and are
algebraically related if and only if $g$ satisfies a differential
equation of order at most $n$. We then say that $g$ is \emm
$\partial$-algebraic,.

We will use as a black box the following theorem, proved in~\cite{DreyfHardtderiv}. It  relaxes some assumptions of~\cite[Prop.~ 2.1 and~2.3]{hardouinCompo} (in particular, it does not require that the solution   $g$ of the linear difference equation belongs to a difference field). 

\begin{thm}[Thm.~C.8 in~\cite{DreyfHardtderiv}, case $\Delta=0$]
  \label{thm:abstractdiffalgcriterianonhomog}
  Let $K$ be a field endowed with a field automorphism~$\sigma$ and a
  derivation $\partial$ commuting with $\sigma$. We assume that the
  field $K^\sigma=\{f \in K: \sigma(f)=f \}$, called the field of \emm constants,, is relatively algebraically  closed in $K$; that is, there is no proper algebraic extension of  $K^\sigma$ in~$K$.

  Let $L$ be a ring extension of $K$ endowed with an automorphism $\sigma_L$ extending $\sigma$ and a derivation $\partial_L$ extending $\partial$. 
  Let $g \in L$ satisfy $\sigma_{L}(g)= g +b$ for some $b \in K$. If $g$
  is $\partial_L$-algebraic over $K$ then there exist $N \in \N_0$,
  constants $c_0,\dots,c_N \in K^\sigma$, not all zero, and
  finally $h \in K$ such that
  \[
    c_0 b +c_1 \partial b +\dots +c_N \partial^N(b)= \sigma(h)-h.
  \] 
\end{thm}

The difference equation~\eqref{eq:funceqomegaplane} that we will derive from the boundary condition~\eqref{eq:bound_cond_gen} is not additive as above, but multiplicative, of the form $\sigma (f)= af$.
But a simple logarithmic transformation yields an additive equation for  $g:= \frac{\partial f}f$:
\[
  \sigma(g)= \sigma \left(\frac{\partial f}f \right)=\frac{\partial (af)}{af} = g+b,
\]
with $b= \partial a/a$. Moreover, $g$ is $\partial$-algebraic if and only if $f$ is $\partial$-algebraic.

\begin{thm}\label{thm:abstractdiffalgcriteria}
  Under the same assumptions on $K$ and $L$ as in
  Theorem~\ref{thm:abstractdiffalgcriterianonhomog}, let $f$ be invertible in $L$ and 
  satisfy $\sigma_{L}(f)= af$ for some $a \in K$. If $f$ is
  $\partial_L$-algebraic over $K$ then there exist $N \in \N_0$,
  constants $c_0,\dots,c_N \in K^\sigma$, not all zero, and finally $h \in K$ such that
  \[
    c_0 \frac{\partial a }{a} +c_1 \partial \left(\frac{\partial a }{a}\right) +\dots +c_N \partial^N\left( \frac{\partial a}{a}\right)= \sigma(h)-h.
  \] 
\end{thm}

\subsection{A finite difference equation}
\label{subsec:funcequ}

Recall that  $q$ is defined to be $e^{2i\beta}$, and is thus a root of
unity if and only if $\beta/\pi \in \Q$.
The results of this  subsection hold whether this is the case or not, and we will use them in Section~\ref{subsec:rootofunityDF}.

By Proposition~\ref{prop:BVP_Carleman_sketch}, the function $\phi_1$ is meromorphic in a domain containing $\overline{\mathcal G}_\R$.
As observed in~(\ref{eq:caracR})--(\ref{eq:caracG}),  the map $s\mapsto y(s)$ defined in~\eqref{eq:uniformization} sends the closed wedge  $\arg (s) \in [\pi, \pi+2\beta]$ to $\bG$.
Hence we can define $\widetilde{\phi}_{1}(s):=\phi_1(y(s))$, at least in a  neighbourhood of this wedge.

Let us specialize the  boundary value condition~\eqref{eq:bound_cond_gen} to $y=y(s)$ with $s\in (-\infty, 0)$. By Lemmas~\ref{lem:R-param} and~\ref{lem:GE}, this gives
\[
  \phi_1(y(1/s))= E(s) \phi_1(y(s)),
\]
where $E(s)$ is the simple rational function~\eqref{eq:definition_E}. Equivalently, since $y(1/s)=y(qs)$ and $\arg (qs)=\pi+2\beta$,
\beq
\label{eq:BVP_qdiff}
\widetilde{\phi}_{1}(qs)=E(s)\widetilde{\phi}_{1}(s).
\eeq
By analytic continuation, this holds in some neighbourhood of $(-\infty, 0)$.  However, we cannot apply directly Galois theory techniques to this $q$-difference  equation, because we would need $\widetilde{\phi}_{1}$ to be defined in a domain closed under the rotation  $s\mapsto qs$. We prove in Appendix~\ref{app:continuation} (Corollary~\ref{cor:meromorphy_phi_1_tilde}) that $\widetilde{\phi}_{1}$ may be continued as a meromorphic function on the slit plane $\C\setminus e^{i\beta}\, \RR_+$, but again, this is not enough. 
We will remedy this by a second parametrization, this time  of the $s$-plane, in terms of a new variable~$\omega$.

Let us write $s=e^{i \omega}$ with $\omega \in \C$. This
transformation of $\omega$ into $s$ sends  the strip $[\pi,
\pi+2\beta] + i\RR$ to the wedge $\arg(s) \in [\pi, \pi+2\beta]$.
Hence we can define, at least in a neighbourhood of this strip, 
\[
  \psi_1(\omega)= \widetilde{\phi}_{1}(e^{i \omega})=\phi_{1}(y(e^{i \omega})).
\]
Note that for $\omega \in \pi +i\RR$, we have $\psi_1(\omega +2 \beta)= \widetilde \phi_1(qe^{i \omega})$,  so that 
the $q$-difference equation~\eqref{eq:BVP_qdiff} becomes a  finite difference equation $\psi_1(\omega+ 2\beta) =E(e^{i\omega}) \psi_1(\omega)$. Moreover, we show in Appendix~\ref{app:continuation} how to extend $\psi_1$
meromorphically to the whole complex plane, starting from the basic functional equation~\eqref{eq:functional_equation}  (see Theorem~\ref{thmpro} and~\eqref{prop-psi1-formula}). This will put us in the correct setting to apply difference Galois theory.

\begin{prop}\label{prop:psi1}
  The function $\psi_1$ can be continued as a meromorphic function on   $\C$ that satisfies 
  \beq
  \label{eq:funceqomegaplane}
  \psi_1(\omega +2 \beta)=M(\omega) \psi_1(\omega),
  \eeq
  where $M(\omega)=E(e^{i\omega})$, and $E(s)$ is the rational function given by~\eqref{eq:definition_E}. 
\end{prop}

We now relate the differential properties of $\phi_1, \widetilde \phi_1$
and $\psi_1$. The proof of the following lemma  is 
analogous to the proof  of~\cite[Prop.~2.3]{DHRS0} and is therefore omitted.

\begin{lem}
  \label{lem:difftranstheta2omega}
  The following statements are equivalent:
  \begin{itemize}
  \item $\phi_1$ is $\frac{d}{d y}$-algebraic over $\C(y)$,
  \item $\widetilde{\phi}_{1}$ is $\partial$-algebraic over
    $\C(s)$ with $\partial=is \frac{d}{ds}$,
  \item $\psi_1$ is $\frac{d}{d \omega}$-algebraic over $\C(e^{i\omega})$. 
  \end{itemize}
\end{lem}

\noindent {\bf Remarks}\\
1. Of course, since $y$ is $\frac{d}{d y}$-algebraic over $\C$, we could replace  in the first statement the field $\C(y)$ by $\C$, or even $\RR$. A similar remark applies to the other two statements.
However, we choose to keep this formulation to emphasize  the fact that~\eqref{eq:BVP_qdiff} and~\eqref{eq:funceqomegaplane} have coefficients	in the base fields $\C(s)$ and $\C(e^{i\om})$, respectively. \\
2.  The choice of the derivation $\partial$ might seem peculiar 
since $\frac d{ds}$ would do as well. But it will be crucial to have $\partial$ 
commute with the multiplication  of $s$ by $q$. That is, if we define  the operator~$\zeta^*$, acting on $\C(s)$, by $\zeta^* f (s)= f(qs)$,  we have 
\beq\label{zeta-partial}
\partial \zeta ^* = \zeta^* \partial .
\eeq
(The notation $\zeta^*$ has been chosen so as to match the map $\zeta$ defined by~\eqref{eq:elements_group}.)
Note also that for any  analytic   function $f(s)$, we have
\beq\label{d-delta}
\frac{d}{d \omega} (f(e^{i\omega}))= (\partial f)(e^{i \omega}).
\eeq

\subsection{Differential transcendence of   the Laplace transform}
\label{subsec:difftrans}
We now assume again that $\beta/\pi$ is irrational, that is, that $q$ is not a root of unity.  As discussed at the beginning of the section, we now  prove the ``only if'' part in the first statement of Theorem~\ref{thm:main_diffalg}.

\begin{prop}\label{prop:DT}
  Assume that $q$ is not a root of unity, and   that $\phi_1(y)$ is  $\frac{d}{dy}$-algebraic. Then one of Conditions~\eqref{eq:CNS0} or~\eqref{eq:CNS0double} holds.
\end{prop}

\begin{proof}
  Assume that $\phi_1(y)$ is  $\frac{d}{dy}$-algebraic.  Then  Lemma~\ref{lem:difftranstheta2omega} implies that $\psi_1$ is  $\frac{d}{d \omega}$-algebraic. 
  Recall from Proposition~\ref{prop:psi1} that $\psi_1$ satisfies the difference equation
  $\psi_1(\omega+2\beta)= M(\omega) \psi_1(\omega)$ with $ M(\omega)= E(e^{i\omega})$. We now apply Theorem~\ref{thm:abstractdiffalgcriteria} to this equation, with $f=\psi_1$, $a=M(\om)$ and the following algebraic setting:
  \begin{itemize}
  \item [--] the field $K$ is $\C(e^{i\omega})$, endowed with the
    automorphism $\sigma(h)(\omega)= h(\omega +2\beta)$, and the
    derivation  $\frac{d}{d \omega}$,
  \item [--] the ring $L$ is $\C(e^{i\omega})[1/\psi_1, \psi_1,\frac{d}{d\omega}(\psi_1),\ldots, \frac{d^j}{d\omega^j}(\psi_1),\ldots]$ (derivatives of $\psi_1$ of any order exist since $\psi_1$ is meromorphic), with the same automorphism and derivation (we prove stability of $L$ by $\sigma$ below).
  \end{itemize}
  Let us check that the assumptions of Theorem~\ref{thm:abstractdiffalgcriteria} (or Theorem~\ref{thm:abstractdiffalgcriterianonhomog}) hold. We begin with the assumptions on~$K$. First, $\sigma$ clearly commutes with the derivation $d/d\om$. Then, since $q$ is not a root of unity, it  is well-known, and easy to see,  that $K^\sigma=\C$. This field is algebraically closed, hence relatively algebraically closed in $K$.
  Regarding $L$, we first note that this ring is  closed by the derivation $\frac{d}{d \omega}$. Moreover, since $\frac{d}{d\omega}$ and $\sigma$  commute, the following formula holds for all $j \in \N_0$:
  \[ \sigma\left(\frac{d^j}{d\omega^j} \psi_1\right)=\frac{d^j}{d\omega^j} (\sigma(\psi_1))= \frac{d^j}{d\omega^j}(M(\omega)\psi_1)
  \]
  by Proposition~\ref{prop:psi1}.
  It follows that $L$ is fixed by $\sigma$. Finally, $\psi_1$ is invertible in $L$ by definition of~$L$.

  The functional equation  of  Proposition~\ref{prop:psi1} now reads $\sigma (f)= a f$ with $f=\psi_1$ and $a=M(\omega) =E(e^{i\omega})\in K$.
  Since we have assumed that $\psi_1$ is D-algebraic,  Theorem~\ref{thm:abstractdiffalgcriteria} implies that there exist $N\in \N_0$,
  constants $ c_0,\dots,c_N \in \C$, not all zero, and finally $h \in \C(e^{i\omega})$, such that 
  \[
    c_0 \frac{M'}{M}(\omega)
    +c_1    \frac{d}{d\omega}\left(\frac{M'}{M} \right)(\omega)+ \cdots
    +    c_N\frac{d^N}{d\omega^N}\left(
      \frac{M'}{M}\right)(\omega)=h(e^{i\omega  +2i\beta})-h(e^{i\omega}),
  \]
  where $M'= \frac{dM}{d\omega}$. Upon replacing $e^{i\omega}$ by $s$, and recalling  that  $\frac{d}{d\omega}(f(e^{i\omega})) =(\partial f)(e^{i\omega})$ (see~\eqref{d-delta}),
  we conclude that
  \[ 
    c_0 \frac{\partial E}{E}+c_1 \partial\left(\frac{\partial E}{E}\right)+ \dots + c_N\partial^N\left( \frac{\partial E}{E}\right)=\zeta^*(h)-h,
  \] 
  where we recall that  $\zeta^*(h)(s)=h(qs)$.

  We now apply to this equation Lemma~3.8 of~\cite{hardouinCompo} (with
  $n=1$), and conclude that the \emm elliptic divisor, of $E$ must be zero\footnote{This lemma is only proved in~\cite{hardouinCompo} when $|q|\not = 1$, but the proof works \emm verbatim, as long as $q$ is not a root of unity, which we have assumed here.}.
  By Lemma~\ref{lem:Ecocyclecaracterization-v0}, this means that one of  Conditions~\eqref{eq:CNS0} or~\eqref{eq:CNS0double}  holds. This completes the proof of Proposition~\ref{prop:DT}, and of Theorem~\ref{thm:main_diffalg}.
\end{proof}

\section{When $\boldsymbol{\beta \in \pi \Q}$}
\label{sec:rat}

We now assume that $q=e^{2i\beta}$  is a root of unity, that is, that $\beta/\pi\in \Q$, and prove Theorem~\ref{thm:main_alg} in two  steps. In Section~\ref{sec:DF-dlog} we prove that the log-derivative of $\phi_1$ is D-finite. The criteria for D-finiteness, algebraicity or rationality of $\phi_1$ are established in Section~\ref{subsec:rootofunityDF}.

\subsection{D-finiteness of the log-derivative}
\label{sec:DF-dlog}
The first part of  Theorem~\ref{thm:main_alg} will follow from an explicit integral expression of $\phi_1$ given in~\cite[Thm.~1]{franceschi_explicit_2017}. This theorem states  that  there exists a constant $c$ 
in  $\C^*$ 
such that 
\[ 
  \phi_1(y)=c \left(\frac{w(0) -w(p)}{w(y) - w(p)} \right)^{k} \exp \left( \frac{1}{2 i \pi} \int_{\mathcal{R}^-} \log G(t)\left[ \frac{w'(t)}{w(t) -w(y)} - \frac{w'(t)}{w(t) -w(0)}\right] dt \right),
\] 
where $w$ is the canonical invariant~\eqref{eq:definition_w-bis}, $p$ is the pole of $\phi_1$ lying in $\overline{\mathcal G}_\R$, if any (see Proposition~\ref{prop:BVP_Carleman_sketch}), $k$ is $1$  if this pole exists, and $0$ otherwise, $\mathcal{R}^{-} =\mathcal{R} \cap \{y\in\C : \Im y \leq 0\}$ is the bottom part of the branch of hyperbola $\R$ defined by~\eqref{eq:def_R}, and $G$ is the algebraic function~\eqref{def:G}.  In particular, $\log G$ is D-finite.
Since $\beta/\pi \in \Q$, the function $w$ is algebraic (Proposition~\ref{prop:algebraic_nature_w}).
Then  the above expression rewrites as
\[
  \phi_1(y)= g(y) \exp (h(y)),
\]
where $g$ is algebraic   and $h$ is  D-finite: indeed,  the integral of a D-finite function along a  curve remains D-finite~\cite[Prop.~3.5]{Ze-90}.
Hence
\[
  \frac   {\phi_1'}{\phi_1} = \frac {g'}{g}  + h'
\]
is clearly D-finite. \qed

\subsection{D-finite/algebraic/rational  cases}
\label{subsec:rootofunityDF}

Assume now  that  $\phi_1$ is D-finite, and let us prove that it is in fact algebraic, and that one of the angle conditions~\eqref{eq:CNS0} or~\eqref{eq:CNS0double} holds.   We have just shown that $\phi_1'/\phi_1$ is also D-finite. This implies that $\phi_1'/\phi_1$ is in fact \emm algebraic, over $\C(y)$:  this follows from  the final statement of~\cite{singer_1986}, which says that if $g(y)$ and $\exp(\int g(y))$ are D-finite, then $g(y)$ is algebraic (we apply this to $g=\phi_1'/\phi_1$).

Since $\phi_1'/\phi_1$ is  algebraic over $\C(y)$, the function $f :=\frac{\partial \widetilde{\phi}_{1}}{\widetilde{\phi}_{1}}$ is algebraic over $\C(s)$, where we define   $\partial =is \frac{d}{ds}$ and $\widetilde{\phi}_{1}(s)=\phi_1(y(s))$ in a neighbourhood of the wedge $\arg(s) \in [\pi, \pi+2\beta]$, as in Section~\ref{subsec:funcequ}. Recall that for $s$ in a neighbourhood of $(-\infty, 0)$, the following equation holds (see~\eqref{eq:BVP_qdiff}):
\[ 
  \zeta^*(\widetilde{\phi}_{1})(s):= \widetilde{\phi}_{1}(qs)=E(s)\widetilde{\phi}_{1}(s),
\]
where $E(s)$ is the rational function~\eqref{eq:definition_E}.  Since $\partial$ commutes with $\zeta^*$  (see~\eqref{zeta-partial}),   the above defined function $f$ satisfies
\[ 
  \zeta^*(f)= \zeta^*\left( \frac{\partial \widetilde \phi_1}{\widetilde \phi_1}\right)
  = \frac{\partial \zeta^* \widetilde \phi_1}{  \zeta^* \widetilde \phi_1}
  =\frac{\partial (E \widetilde \phi_1)}{E\widetilde \phi_1}
  =f+\frac{\partial E}{E}.
\]
Write the minimal polynomial equation satisfied by $f$ over $\C(s)$ as
\[
  c_0  +c_1 f+ \dots +c_{d-1} f^{d-1} + f^d =0   ,
\]
and apply to this identity the operator $\zeta^*$. This gives, for $s$ in a neighbourhood of $(-\infty, 0)$,
\[
  \zeta^*(c_0)+  \zeta^*(c_1) \left(f+ \frac {\partial E} E\right) + \cdots +
  \zeta^*(c_{d-1}) \left(f+ \frac {\partial E} E\right)^{d-1}
  + \left(f+ \frac {\partial E} E\right)^d=0.
\]
 Comparing with  the minimal  equation of $f$ gives, for $j =0, \ldots, d-1$:
\[
  c_j= \sum_{i= j}^d \zeta^*(c_i) {i\choose j} \left(
    \frac {\partial E} E\right)^{i-j} 
\]
with $c_d=1$. In particular,
\[
  c_{d-1}=\zeta^*(c_{d-1}) +d \,  \frac {\partial E} E.
\]
Hence we can write $\frac{\partial E}{E} = \zeta^*(h)-h$ where $h=-c_{d-1}/d \in \C(s)$. Let $n>0$ be minimal such that $q^n=1$.  Applying ${\zeta^*}^j$ with $j=0,\dots,n-1$ to the latter equation, we obtain 
\[ 
  {\zeta^*}^j \left(\frac{\partial E}{E}\right)= {\zeta^*}^{j+1} (h)-{\zeta^*}^j(h).
\]
Summing these $n$ identities and noting that ${\zeta^*}^n$ is the identity on $\C(s)$, we find
\[
  0 = \sum_{j=0}^{n-1} {\zeta^*}^j\left(\frac{\partial E}{E}\right)=\frac{\partial (\prod_{j=0}^{n-1} {\zeta^*}^j(E))}{\prod_{j=0}^{n-1} {\zeta^*}^j(E)}.
\]
The latter identity uses again the fact that $\partial$ and $\zeta^*$ commute (see~\eqref{zeta-partial}). By definition of $\zeta^*$ and of the derivation $\partial$,  this implies that  $\prod_{j=0}^{n-1} E(sq^j)=c$ for some  $c \in \C^*$. By setting $s$ to $0$, we find that $c=E(0)^n= (s_1/s_2)^n$. If instead we let $s$ tend to infinity, we find that $c=(s_2/s_1)^n$. Hence $(s_2/s_1)^n=\pm 1$.
Now consider the cyclic field extension $\C(s^n) \subset \C(s)$. Its Galois group is generated by $\zeta^*: h(s) \mapsto h(qs)$. Thus, the norm of $ (s_2/s_1) E$ with respect to is
\[
     \prod_{j=0}^{n-1}  \left( \frac{s_2}{s_1}{E(q^js)} \right)= \left( \frac{s_2}{s_1}\right)^nc =1.
\]
By Hilbert's Theorem~90~\cite[Thm.~6.1]{Langalgebra}, 
there exists $H \in \C(s)$ such that $ (s_2/s_1) E(s) =\frac{H(s)}{{H(qs)}}$. Since $(s_2/s_1)^{2n}=1$, it follows that $E(s)^{2n}$ can be written as $\frac{\widetilde H(s)}{\widetilde H(qs)}$. We conclude,
using  Lemma~\ref{lem:Ecocyclecaracterization-v0}, that one of the angle conditions~\eqref{eq:CNS0} or~\eqref{eq:CNS0double} holds. But then Theorem~\ref{thm:main} or Theorem~\ref{thm:main-double} applies, and tells us that $\phi_1$ is in fact algebraic (since $\beta/\pi \in \Q$).

\medskip

We have thus proved that $\phi_1$ is D-finite if and only if it is algebraic, and that in this case one of Conditions~\eqref{eq:CNS0} or~\eqref{eq:CNS0double} holds. To conclude the proof of Theorem~\ref{thm:main_alg}, we now investigate the rational case.

If $\alpha \in -\N_0$, then Condition~\eqref{eq:CNS0} holds, and $\phi_1$ is rational by Theorem~\ref{thm:main}.  Conversely, if $\phi_1$ is rational, then $\alpha \in -\N_0$ by~\eqref{phi-asympt}.

The proof of Theorem~\ref{thm:main_alg} is now complete. \qed

\bigskip

\noindent{\bf Acknowledgements.} The authors are grateful to S\'ebastien Labb\'e for his advice on the counting problem of Lemma~\ref{lem:poles_roots_F}, related to Sturmian sequences. They also thank Frédéric Chyzak for interesting discussions and advances in an attempt to compute the density of Proposition~\ref{prop:specialcase} via computer algebra.

\appendix

\section{Basic conditions: from the quadrant to the $\beta$-wedge}

In this appendix, we prove the equivalence between various conditions on the parameters of the model, considered in a quadrant or in the $\beta$-wedge. Surprisingly, we could not find these proofs in the literature.

Let us recall the angle conditions met in Sections~\ref{sec:introduction} and~\ref{sec:prelim} of the paper. The first four are relevant to the $\beta$-wedge:

$\eqref{eq:conditionRsemimart1} :
\delta+\varepsilon -\pi<\beta, \quad {\hbox{or equivalently}}  \quad
\alpha<1,$

$\eqref{theta-cond1}: 0<\theta<\beta,$

$\eqref{eq:stationary_distribution_CNS1}: \beta-\varepsilon<\theta<\delta,$

$\eqref{eq:conditionRsemimart1}+\eqref{theta-cond1}+\eqref{eq:stationary_distribution_CNS1} =\eqref{assumptions}:
\delta - \pi < \beta-\varepsilon <\theta <\delta, \quad  0
<\theta<\beta,$

\noindent while the next three are relevant to the quadrant:

$
\eqref{eq:conditionRsemimart}:
\det R >0 \quad \text{or} \quad (  r_{12}>0 \text{ and }  r_{21}>0),
$

$\eqref{eq:drift_negative0} 
: \mu_1<0, \ \mu_2<0,$ 

$
\eqref{eq:stationary_distribution_CNS}:
\det R >0,
\quad r_{22} \mu_1 - r_{12}  \mu_2 < 0, \quad r_{11} \mu_2 - r_{21}  \mu_1 < 0.
$

\medskip
\noindent Throughout the paper, we have $\beta\in(0,\pi)$, and   the following natural reflection conditions hold:
\begin{itemize}
\item in the $\beta$-wedge, the reflection angles  $\delta$ and $\varepsilon$ lie in $(0, \pi)$,
\item    in the quadrant, the reflection matrix $R$ satisfies $ r_{11} > 0$ and $ r_{22} > 0$.
\end{itemize}

\begin{lem} We have the following equivalences:
  \begin{enumerate}[label={\rm\roman*)},ref={\rm\roman*)}] 
  \item \label{item:i} The semimartingale condition~\eqref{eq:conditionRsemimart1} for the $\beta$-wedge    
    is equivalent to the semimartingale condition~\eqref{eq:conditionRsemimart} for the quadrant.
  \item \label{item:ii}  Condition~\eqref{theta-cond1} is equivalent to the drift condition~\eqref{eq:drift_negative0}.
  \item\label{item:iii} Condition $\eqref{assumptions}=\eqref{eq:conditionRsemimart1}+\eqref{theta-cond1}+\eqref{eq:stationary_distribution_CNS1}$  is equivalent to $\eqref{eq:drift_negative0}+\eqref{eq:stationary_distribution_CNS}$.
  \end{enumerate}
  \label{lem:equivcondition}
\end{lem}

\begin{proof}
  Recall from~\eqref{eq:expression_delta_epsilon} and~\eqref{eq:tan_theta} that:
  \beq
  \label{eq:expression_delta_epsilon_theta}
  \tan\delta=
  \frac{\sin\beta}{\frac{r_{12}}{r_{22}}\sqrt{\frac{\sigma_{22}}{\sigma_{11}}}+\cos\beta} ,
  \quad
  \tan\varepsilon=
  \frac{\sin\beta}{\frac{r_{21}}{r_{11}}\sqrt{\frac{\sigma_{11}}{\sigma_{22}}}+\cos\beta}  ,
  \quad
  \tan  \theta=\frac{\sin \beta}{\frac {\mu_1}{\mu_2} \sqrt{\frac{\sigma_{22}}{\sigma_{11}}}  +\cos \beta}.
  \eeq
  This implies that
  \beq
  \tan (\beta-\varepsilon)= \frac{\sin \beta}{\frac{r_{11}}{r_{21}}\sqrt{\frac{\sigma_{22}}{\sigma_{11}}}+\cos \beta},
  \qquad
  \tan (\beta-\theta)=\frac{\sin\beta}{\frac{\mu_2}{\mu_1}\sqrt{\frac{\sigma_{11}}{\sigma_{22}}}+\cos\beta} .
  \label{eq:tanbetaepsilon}
  \eeq
  We will repeatedly use the fact that  the cotangent function is $\pi$-periodic, and decreasing in $(0, \pi)$. For instance, since by~\eqref{eq:expression_delta_epsilon_theta},
  \[
    \cot \varepsilon - \cot \beta= \frac 1{\tan \varepsilon} - \frac 1{\tan \beta} 
    = \frac{r_{21}}{r_{11}}\sqrt{\frac{\sigma_{11}}{\sigma_{22}}}\frac{1}{\sin \beta},
  \]
  while $\vareps, \beta\in (0, \pi)$,   we have 
  \beq
  \varepsilon\leqslant \beta \Leftrightarrow {r_{21}}\ge 0.
  \label{eq:ebr21}
  \eeq

  We now begin with the proof of~\ref{item:i}. We have the following sequence of equivalences, starting from Condition~\eqref{eq:conditionRsemimart1}:
  \begin{align*}
    \delta+\varepsilon-\pi <\beta & \Leftrightarrow
                                    \left( 0<\delta<\beta-\varepsilon+\pi<\pi \right) \text{ or } (\beta-\varepsilon\geqslant 0)
    & & \text{as } \delta\in(0,\pi)
    \\
                                  & \Leftrightarrow
                                    \left(\frac{1}{\tan (\beta-\varepsilon)}<\frac{1}{\tan \delta} 
                                    \text{ and } {r_{21}}<0 
                                    \right) \text{ or } {r_{21}}\geqslant 0
    & & \text{by }~\eqref{eq:ebr21}
    \\
                                  & \Leftrightarrow
                                    \left(\frac{r_{11}}{r_{21}} <\frac{r_{12}}{r_{22}} \ \text{ and } r_{21}<0 
                                    \right)  \text{ or }r_{21}\geqslant 0 &&  \text{by \eqref{eq:expression_delta_epsilon_theta} and~\eqref{eq:tanbetaepsilon}}
    \\
                                  & \Leftrightarrow
                                    \left(r_{11}r_{22}-r_{12}r_{21}>0  \text{ and } r_{21}<0  \right) \text{ or } r_{21}\geqslant 0 .
  \end{align*}
  A case analysis reveals that this is equivalent to Condition~\eqref{eq:conditionRsemimart}, namely $ \det R >0
  \text{ or } (r_{21}> 0 \text{ and } r_{12}> 0)$.

  We go on with the proof of~\ref{item:ii}.
  First, it follows from~\eqref{theta_precise} and~\eqref{eq:mu} that $\theta>0$ is equivalent to $\mu_2<0$. Hence we will now assume that $\theta>0$ and $\mu_2<0$, and prove, under these assumptions, that $\theta<\beta$ is equivalent to $\mu_1<0$.  Given that $\theta, \beta\in(0,\pi)$, we have the following sequence of equivalences:
  \begin{align*}
    \theta<\beta 
    & \Leftrightarrow  \frac 1{\tan \beta} < \frac 1 {\tan \theta}&&\\
    &     \Leftrightarrow  \frac{\mu_1 }{\mu_2} >0 && \text{by~\eqref{eq:expression_delta_epsilon_theta}}\\
    &         \Leftrightarrow \mu_1 <0  && \text{since} \quad \mu_2<0.
  \end{align*}

  Let us now prove~\ref{item:iii}. We have already seen that Conditions~\eqref{theta-cond1} and~\eqref{eq:drift_negative0} are equivalent. We thus assume that they hold, that is, that $0<\theta<\beta$ and $\mu_1, \mu_2 <0$, and prove the equivalence of~$\eqref{eq:conditionRsemimart1}+\eqref{eq:stationary_distribution_CNS1}$ and~\eqref{eq:stationary_distribution_CNS} under this assumption. We begin  with the first part of~\eqref{eq:stationary_distribution_CNS1}.  We have the following sequence of equivalences:
  \begin{align*}
    \beta-\varepsilon<\theta \Leftrightarrow \beta-\theta < \varepsilon 
    & \Leftrightarrow
      \frac{1}{\tan\varepsilon} <\frac{1}{\tan (\beta-\theta)} && \text{as } \vareps, \beta-\theta \in(0,\pi)
    \\ 
    &\Leftrightarrow \frac{r_{21}}{r_{11}}< \frac{\mu_2}{\mu_1} &&  \text{by }   \eqref{eq:expression_delta_epsilon_theta} \text{ and }                          \eqref{eq:tanbetaepsilon}
    \\
    & \Leftrightarrow r_{11}\mu_2<r_{21}\mu_1& &  \text{as } \mu_1<0 \text{ and } r_{11}>0.
  \end{align*}
  We recognize the third part of~\eqref{eq:stationary_distribution_CNS}.
  In a similar fashion, we  show that the second part of~\eqref{eq:stationary_distribution_CNS1} is equivalent to the second part of~\eqref{eq:stationary_distribution_CNS}:
  \begin{align*}
    \theta<\delta
    & \Leftrightarrow
      \frac{1}{\tan\delta} <\frac{1}{\tan \theta} && \text{as } \theta, \delta \in (0,\pi)
    \\ 
    &\Leftrightarrow \frac{r_{12}}{r_{22}}< \frac{\mu_1}{\mu_2}&&  \text{by } \eqref{eq:expression_delta_epsilon_theta}
    \\
    & \Leftrightarrow r_{22}\mu_1<r_{12}\mu_2& &  \text{as } \mu_2<0 \text{ and } r_{22}>0.
  \end{align*}
  Assume now that~\eqref{eq:stationary_distribution_CNS} holds. The above two calculations show that~\eqref{eq:stationary_distribution_CNS1} holds as well. Moreover, the first part of~\eqref{eq:stationary_distribution_CNS} implies~\eqref{eq:conditionRsemimart}, which implies~\eqref{eq:conditionRsemimart1} by item~\ref{item:i}.

  Conversely, assume that Conditions~\eqref{eq:conditionRsemimart1} and~\eqref{eq:stationary_distribution_CNS1} hold. The above two calculations show that the second and third parts of~\eqref{eq:stationary_distribution_CNS} hold as well:
  \beq\label{CNS-part} 
  r_{22}\mu_1<r_{12}\mu_2 \quad \text{ and }  \quad r_{11}\mu_2<r_{21}\mu_1.
  \eeq
  Moreover, by item~\ref{item:i} we have $\det R>0$ or ($r_{12}>0$ and $r_{21}>0$). If $\det R>0$ we are done, since this is the first part of~\eqref{eq:stationary_distribution_CNS}.  If  ($r_{12}>0$ and $r_{21}>0$), it follows from~\eqref{CNS-part} and the fact that $\mu_1<0$ and $\mu_2<0$ that
  \[
    r_{22}/r_{12}>\mu_2 / \mu_1 \quad \text{  and } \quad r_{11}/r_{21}>\mu_1 /\mu_2.
  \]
  This implies that $r_{22} r_{11}/(r_{21}r_{12}) >1$, so that $\det R>0$ again. This concludes the proof.
\end{proof}

\section{Proof of Lemma~\ref{lem:Ecocyclecaracterization-v0}}
\label{sec:standardform}

Several  of  the results that we prove here were already proved  in~\cite{hardouinCompo}, under the assumption $|q| \neq 1$. In this article however, $q=e^{2i\beta}$, so we always have $|q|=1$. Nonetheless, the proofs of~\cite{hardouinCompo} apply \emm verbatim, as long as  $q$ is \emm not, a root of unity. For that reason we often separately consider the case where
$q$ is a root of unity, that is, when $\beta/\pi \in \Q$.

We choose an arbitrary system $\cS \subset \C^*$ of representatives of $\C^*/q^\Z$. In other words, $\cS$ is a subset of $\C^*$ such that for  any $z \in \C^*$, there exists a unique $z' \in \cS$
such that  $z=q^\ell z'$ for some $\ell \in \Z$. We denote by
$[z]$ the equivalence class of $z$ for the relation
defined by   $z\sim z'$ if $z/z'\in q^{\Z}$.

\begin{defn}
  Let  $a \in \C(s)^*$. We say that $a$ is \emph{standard} if for any $z\in \C^*$, at most one element of $[z]$ is a zero or a pole of $a$, possibly of multiple order.
\end{defn}

\begin{lem}
  \label{lem:standardform}
  Let $a \in \C(s)^*$. There exist   $f \in \C(s)^*$ and $\overline{a} \in \C(s)$ standard such that   $a =\overline{a}\, \frac{f(qs)}{f(s)}$.
\end{lem}

\begin{proof}
  We refer to~\cite[Lem.~3.3]{hardouinCompo} for a proof that holds when $q$ is not a root of unity. If $q$ is a root of unity of order $n$, let us write
  \beq
  \label{eq:exp_a}
  a=\kappa s^\ell \prod_{z \in \cS} \prod_{k=0}^{n-1}(q^{k}s-z)^{m_{k,z}},
  \eeq
  where $\kappa \in \C$, $\ell\in\Z$, and  only finitely many of the integers  $m_{k,z}$ are non-zero. Define
  \[
    \overline{a}=\kappa s^\ell \prod_{z \in    \cS}(s-z)^{\sum_{k=0}^{n-1} m_{k,z}}.
  \]
  Then $\overline{a}$ is clearly standard and one easily checks that $a
  =\overline{a}\, \frac{f(qs)}{f(s)}$, with
  \[
    f(s)= \prod_{z \in \cS}
    \prod_{k=0}^{n-1}\prod_{j=0}^{k-1}(q^{j}s-z)^{m_{k,z}}.\qedhere
  \]
\end{proof}

\begin{defn}
  \label{defn:elliptic_divisor}
  Let $a \in \C(s)^*$. If $q$ is a root of unity of order $n$, let us write $a$ as in~\eqref{eq:exp_a}.
  The \emph{elliptic divisor} of $a$ is defined as the formal sum
  \[
    \mathrm{div}_q(a)=\sum_{z \in \cS} \left(\sum_{k=0}^{n-1} m_{k,z}\right) [z].
  \]
  If $q$ is not a root of unity, let us write 
  \[
    a =\kappa s^\ell \prod_{z \in \cS} \prod_{k \in \Z} (q^ks -z)^{m_{k,z}},
  \]
  where $\ell \in \Z$, $\kappa \in \C$ and  finitely many of the $m_{k,z} \in \Z$ are non-zero. We define the \emph{elliptic divisor} of $a$ as the formal sum
  \[
    \mathrm{div}_q(a)=\sum_{z \in \cS} \left(\sum_{k \in \Z} m_{k,z}\right) [z].
  \]
\end{defn}

Let us observe that:
\begin{itemize} 
\item for $a$ and $b $ in $\C(s)^*$, we have $\mathrm{div}_q(ab)=\mathrm{div}_q(a) +\mathrm{div}_q(b)$,
\item for $f \in \C(s)^*$, we have $\mathrm{div}_q\bigl(\frac{f(qs)}{f(s)}\bigr)=0$.
\end{itemize}

\begin{lem}\label{lem:caracelliptidivzero}
  Let $a \in \C(s)^*$. The following statements are equivalent:
  \begin{itemize}
  \item there exist $\kappa \in \C$, $\ell \in \Z$ and $f \in \C(s)$ such that $a=\kappa s^\ell \frac{f(qs)}{f(s)}$,
  \item the elliptic divisor $\mathrm{div}_q(a)=\sum_{z \in \cS} n_z [z]$ is zero, that is, $n_z =0$ for all $z \in \cS$. 
  \end{itemize}
\end{lem}

\begin{proof}
  It is clear from the above two observations  that the first condition
  implies the second, as $\mathrm{div}_q(s^\ell)=0$. We now assume that
  $\mathrm{div}_q(a)=0$, and prove that the first condition holds. As before, we refer to~\cite[Lem.~3.5]{hardouinCompo} when $q$ is not
  a root of unity, and assume that $q$ is a root of unity of order~$n$. Let us write $a =\overline a f(sq)/f(s)$  as in Lemma~\ref{lem:standardform}, with $\overline a$ standard.
  Using again the observations above, and the assumption $\mathrm{div}_q(a)=0$, we find that $\mathrm{div}_q(\overline{a})=0$.
  Let us write $\overline{a}= \kappa s^\ell \prod_{z \in
    \cS}(q^{k_z}s -z)^{m_z}$ where $k_z \in \Z$ and
  only finitely many of  the $m_z \in \Z$ are non-zero. Then
  \[
    0=   \mathrm{div}_q(\overline{a})= \sum_{z \in   \cS} {m_z}[z],
  \]
  which implies that all exponents  ${m_z}$ are zero, so that
  $\overline{a}=\kappa s^\ell$. This concludes the proof.
\end{proof}

We are now ready to prove Lemma~\ref{lem:Ecocyclecaracterization-v0}.

\begin{proof}[Proof of Lemma~\ref{lem:Ecocyclecaracterization-v0}]
  Assume that there exists $m \in \N$ and $H \in \C(s)$ such that
  $E^m(s)=\frac{H(s)}{H(qs)}$.  By Lemma~\ref{lem:caracelliptidivzero},
  one has $0=\mathrm{div}_q(E^m)=m\,\mathrm{div}_q(E)$, so that $\mathrm{div}_q(E)=0$. Hence~\ref{item:i-eco}$\Rightarrow $\ref{item:ii-eco}.

  Now let us prove the equivalence between~\ref{item:ii-eco} and~\ref{item:iii-eco}.
  Since $\mathrm{div}_q(E)= [s_1]+[\frac{1}{s_2}]- [s_2]-[\frac{1}{s_1}]$,
  Condition~\ref{item:ii-eco} means that  either $[s_1]=[s_2]$ (in which case $[\frac{1}{s_1}]=[\frac{1}{s_2}]$) or that  $[s_i]=[\frac{1}{s_i}]$ for
  $i=1,2$. The first case can be restated by saying that $s_1/s_2 \in q^\Z$, and  the second by saying that $s_1^2$ and $s_2^2$ belong to $q^\Z$.

  Finally, if~\ref{item:ii-eco} holds,  then Lemma~\ref{lem:caracelliptidivzero} implies that $E(s)=\kappa s^\ell f(qs)/f(s)$ for some $\kappa\in \C$, $\ell \in \Z$ and $f(s) \in \Q(s)$. If $f(s)$ grows like $s^e$ for some $e\in \Z$, the function $\widetilde f(s):= s^{-e} f(s)$ also satisfies $E(s)=\widetilde \kappa s^\ell \widetilde f(qs)/\widetilde f(s)$ for another constant $\widetilde \kappa$. Hence we can assume without loss of generality that $f(s)$ tends to a non-zero finite limit as $s$ tends to infinity.   By letting $s$ tend to infinity in $E(s)=\kappa s^\ell f(qs)/f(s)$, where $E(s)$ is given by~\eqref{eq:definition_E}, we see that $\ell=0$ and  $  {s_2}/{s_1}= \kappa$.
  By setting $s=0$ instead, we see that $ {s_1}/{s_2}= \kappa$. Hence $\kappa=1/\kappa=\pm 1$, and $E^2(s)=f^2(qs)/f^2(s)$.
\end{proof}

\section{The case $\alpha_1=\alpha_2=0$}
\label{app:alg}
In this section we prove Proposition~\ref{prop:specialcase}, which gives the expression of the density of the stationary distribution of the SRBM under the assumption $\alpha_1=\alpha_2=0$. We proceed as follows: we denote by $\hp_0(z_1,z_2)$ the density described in Proposition~\ref{prop:specialcase}, by $\hphi$ its Laplace transform:
\[
 \hphi(x,y)= \iint_{{\RR}_+^2} e^{xz_1+yz_2}\, \hp_0(z_1,z_2)\mathrm{d} z_1\mathrm{d} z_2,
\]
and we prove that $\hphi(x,y)$ is indeed given by~\eqref{eq:functional_equation}:
\beq\label{phi-hat-form}
  -\gamma (x,y) \hphi (x,y) =\gamma_1 (x,y) \varphi_1 (y) + \gamma_2 (x,y) \varphi_2 (x),
\eeq
where the densities $\phi_1(y)$ and $\phi_2(x)$ are those of~\eqref{phi1-very-simple} and~\eqref{phi2-very-simple}. We  work with the normal variables $\xn$ and $\yn$ defined by~\eqref{xy:normal}. We first perform the change of variables $(z_1, z_2) \mapsto\left( \frac{\det \Sigma}{\sqrt{\Delta
        \sigma_{22}}} \, z_1, \frac{\det \Sigma}{\sqrt{\Delta
        \sigma_{11}}} \, z_2\right)$, which yields
\[
  \hphi(x,y)= \frac {\det ^2\Sigma}{\Delta \sqrt{\sigma_{11}\sigma_{22}}}\iint_{{\RR}_+^2} e^{\xn z_1+\yn z_2}\widetilde p_0(z_1,z_2) \mathrm{d} z_1\mathrm{d} z_2,
\]
with $\widetilde p_0(z_1,z_2)$ given by~\eqref{eq:pi}. Then we introduce the variables $\rho:=|z|$ and $a$ involved in~\eqref{eq:pi}. That is,
$z_1=\rho \sin (\beta-a)/\sin \beta$ and $z_2= \rho \sin a/\sin \beta$.  The Jacobian is found to be $\rho/\sin \beta$, and $\sin \beta$ is given by~\eqref{sin-beta}. Hence:
\[
  \hphi(x,y)= \kappa_1 \int _0^ \beta \int_{{\RR}_+}  \sqrt\rho\, \cos\left( \frac{\theta-a}2\right) \exp\left(-\rho\left(2\cos^2\left(\frac{\theta-a}2\right) - \xn \frac{\sin{(\beta-a)}}{\sin \beta}- \yn \frac{\sin{a}}{\sin \beta}  \right) \right) \mathrm{d}\rho \,\mathrm{d}a,
\]
with $\kappa_1= \frac{2\sqrt 2\sin \delta \sin \vareps}{\sqrt \pi \sin \beta/2}$. The integration in $\rho$ is easily performed:
\[
  \hphi(x,y)= \frac{\kappa_1 \sqrt \pi}{2} \int _0^ \beta\cos\left( \frac{\theta-a}2\right)
  \frac {\mathrm{d}a} {\left(2\cos^2\left(\frac{\theta-a}2\right) - \xn \frac{\sin{(\beta-a)}}{\sin \beta}- \yn \frac{\sin{a}}{\sin \beta}  \right)  ^{3/2}}.
\]
The integration in $a$ looks more impressive, but can be performed as well. We write $a=\theta-2s$, with $s$ ranging from $-(\beta-\theta)/2$ to $\theta/2$, and then introduce $t=\tan s$. The integral becomes
\begin{align*}
  \hphi(x,y)&= \kappa_1\sqrt \pi\int_{-\frac{\beta-\theta}2}^{ \frac \theta2}
              \frac{\cos s\,  \mathrm{d} s} {\left(2\cos ^2 s - \xn \frac{\sin{(\beta-\theta+2s)}}{\sin \beta}- \yn \frac{\sin(\theta-2s)}{\sin \beta}  \right)  ^{3/2}}
\\
            &= \kappa_1\sqrt \pi \int_{-\tan \frac{\beta-\theta}2}^{\tan \frac \theta2}
              \frac{ (\sin \beta) ^{3/2}\, \mathrm{d} t}
              {\left( 2\sin \beta- (1-t^2)A -{2tB}\right)^{3/2}}
\end{align*}
with $A= ( \xn \sin (\beta-\theta) +\yn \sin \theta) $ and $B=  ( \xn\cos(\beta-\theta) -\yn \cos \theta)$. Now the integral in $t$ can be done explicitly:
\beq\label{phi-hat}
  \hphi(x,y)  =  \frac{\kappa_1\sqrt \pi (\sin \beta) ^{3/2}}{{A^2+B^2-2A\sin \beta}} \left[ \frac{B-t A }
                {\left( 2\sin \beta- (1-t^2)A -{2tB}\right)^{1/2}}\right]_{-\tan \frac{\beta-\theta}2}^{\tan \frac \theta2}.
\eeq
The rest of the calculation is tedious but straightforward.  One finds
\begin{align}
  A^2+B^2-2A\sin \beta&= \xn^2+\yn^2-2\xn \yn \cos \beta - 2\xn \sin \beta \sin(\beta-\theta) - 2\yn \sin \beta\sin \theta \nonumber \\
                      &= \frac{2\sin^2\beta \det \Sigma}\Delta \gamma(x,y),  \hskip 20mm \text{by}~\eqref{gamma-normal}.
                        \label{den}
\end{align}
 The function between square brackets in~\eqref{phi-hat}, denoted $I(t)$,  can be written back in trigonometric terms using the variable $s$ such that $t=\tan s$:
\[
  I(t)= \frac{\xn\, \cos(\beta-\theta+s) -\yn\, \cos(\theta-s)}{\left(2\cos ^2 s\sin \beta - \xn\, {\sin{(\beta-\theta+2s)}}- \yn\, {\sin(\theta-2s)} \right)  ^{1/2}}.
\]
At $t= \tan \frac \theta 2$, this takes the value
\begin{align}
  I\left(\tan \frac \theta 2\right)=
  \frac{\xn\, \cos(\beta-\theta/2) -\yn\, \cos(\theta/2)}
  {\left(2 \sin \beta \cos ^2 (\theta/2)- \xn\, {\sin{\beta}}\right)  ^{1/2}} \nonumber
  &= -\frac{\xn\, \sin(\beta-\delta) + \yn\, \sin \delta}
    {\sqrt{2\sin\beta} \sin \delta \sqrt{1-\xn/\xn^+}},
  \\
  &=\frac {\det \Sigma \sqrt{\sin \beta} \sin (\delta+\vareps)}
    {\sqrt 2 \Delta \sin \delta \sin \vareps} \, \gamma_2(x,y) \phi_2(x).
    \label{int-top}
\end{align}
In the first line, we have used the relations~\eqref{angles-00} between the angles $\beta$, $\theta$, $\delta$ and $\vareps$, and the expression~\eqref{xyn-pm} of $\xn^+$. In the second line, we have also used the expressions of $\gamma_2$, $\phi_2(x)$ and $\phi_2(0)$; see~\eqref{gamma_i:normal},~\eqref{phi2-very-simple}, and~\eqref{masses}. In an analogous fashion, we determine
\beq\label{int-bottom}
  I\left(-\tan \frac {\beta -\theta} 2\right)=-
   \frac {\det \Sigma \sqrt{\sin \beta} \sin (\delta+\vareps)}
  {\sqrt 2 \Delta \sin \delta \sin \vareps}\, \gamma_1(x,y) \phi_1(y).  
  \eeq
  We now get back to~\eqref{phi-hat}, and inject~\eqref{den},~\eqref{int-top},~\eqref{int-bottom}, and the value of $\kappa_1$. Using finally $\sin (\beta/2)=-\sin(\delta+\vareps)$, this gives  the desired expression of $\hphi(x,y)$,  namely~\eqref{phi-hat-form}.

\section{Lifting  of $\phi_1$ and $\phi_2$ to the universal covering of  $\mathcal S\setminus\{0, \infty\}$}
\label{app:continuation}
In this section, we explain how the function
\[
  \psi_1(\omega):= \phi_1(y(e^{i\omega}))
\]
that has been used in Section~\ref{sec:DT}, originally defined  analytically in
a neighbourhood of the line $\Re\om= \pi+\beta$ where $\Re y(e^{i\om})\le 0$,
can be extended to a meromorphic function on $\C$, together with its counterpart $\psi_2$ defined by $ \psi_2(\omega):= \phi_2(x(e^{i\omega}))$. The key idea is to use the relation
\[
  \gamma_1(x(s),y(s)) \phi_1(y(s))+  \gamma_2(x(s),y(s)) \phi_2(x(s)) =0,
\]
derived from the basic functional equation~\eqref{eq:functional_equation}, to construct $\psi_1$ and $\psi_2$ on larger and larger domains. Since $\phi(x,y)$ is, \emm a priori,, defined  when $\Re x\le 0$ and $\Re y\le 0$,  the above identity holds at least when $\Re x(s)\le 0$ and $\Re y(s)\le 0$. 

  Observe that we have already used in this paper a continuation of $\phi_1(y)$ beyond the half-plane $\{y:\Re y\le 0\}$, constructed in~\cite[Lem.~3]{franceschi_explicit_2017} to include the domain $\G$; see Proposition~\ref{prop:BVP_Carleman_sketch}. Moreover,  a continuation of the function $\widetilde \phi_1(s):= \phi_1(y(s))$ beyond the set $\{s: \Re y(s)\le 0\}$ is also constructed in~\cite[Sec.~3]{franceschi_asymptotic_2016}. However, we need to go one step higher and work with the variable $\om\in \C$ to apply Galois theoretic tools, as explained in Section~\ref{subsec:funcequ}. We work from scratch and do not use the earlier continuations.

Before we embark on the details of our construction, let us mention that in the discrete setting, where one considers reflected random lattice walks in the positive quadrant and their stationary distributions, a similar meromorphic continuation  of the stationary probability generating function is constructed  in~\cite[Chap.~3]{FIM17}. The details of the construction are however quite different, mostly because the counterpart of the curve $\{(x,y)\in(\C\cup\{\infty\})^2: \gamma(x,y)=0 \}$ has genus  $1$ in the discrete setting (requiring the use of elliptic functions), instead of~$0$ in the present paper.

\subsection{Universal covering of the doubly punctured sphere}
 The uniformization~\eqref{eq:uniformization} allows us to constructively and explicitly identify $\cS$, the Riemann surface of genus $0$ defined in~\eqref{eq:definition_Riemann_sphere_2} by the cancellation of the kernel, with $\C\cup\{\infty\}$.
Let us consider $\C^*\equiv\cS\setminus \{0,\infty\}$, which is mapped by~\eqref{eq:uniformization} to the finite points of the surface.
This surface is then homeomorphic to a doubly punctured sphere, or equivalently  a cylinder, and may be considered as an infinite vertical strip whose opposite edges are identified (Figure~\ref{fig:informaluniversalcovering}, left).
Informally, the universal covering $\widehat{\s}$ of $\cS\setminus \{0,\infty\}$ consists of infinitely many such strips glued together and covering the complex plane (Figure~\ref{fig:informaluniversalcovering}, right). 

\begin{figure}[hbtp]
  \includegraphics[height=3cm]{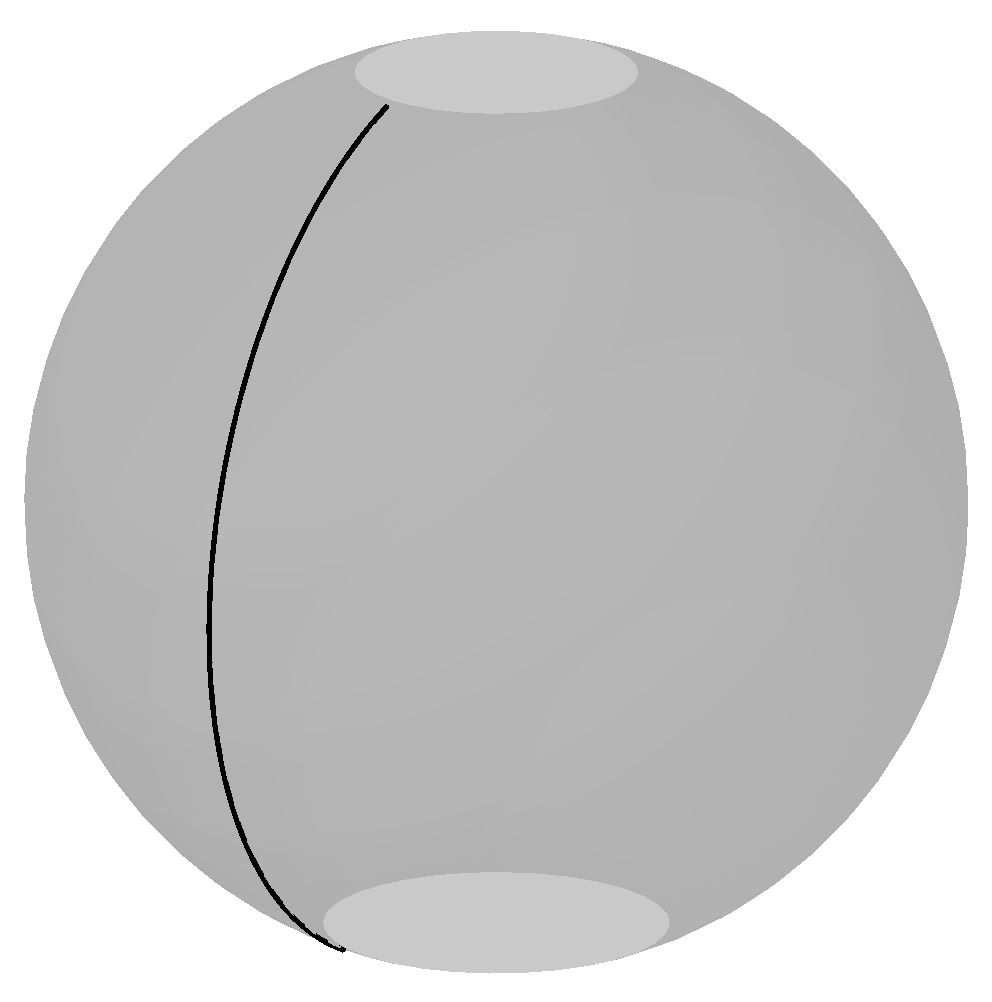}  \hspace{2mm}
  \includegraphics[height=3cm]{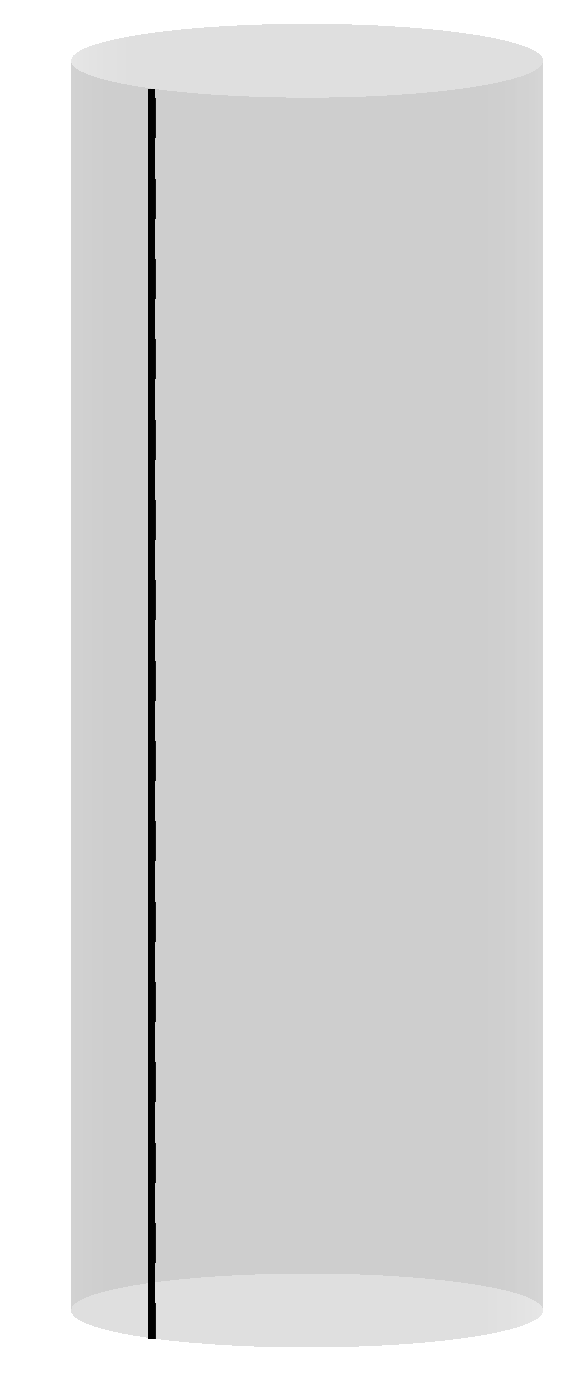} \hspace{2mm} 
  \includegraphics[height=4cm]{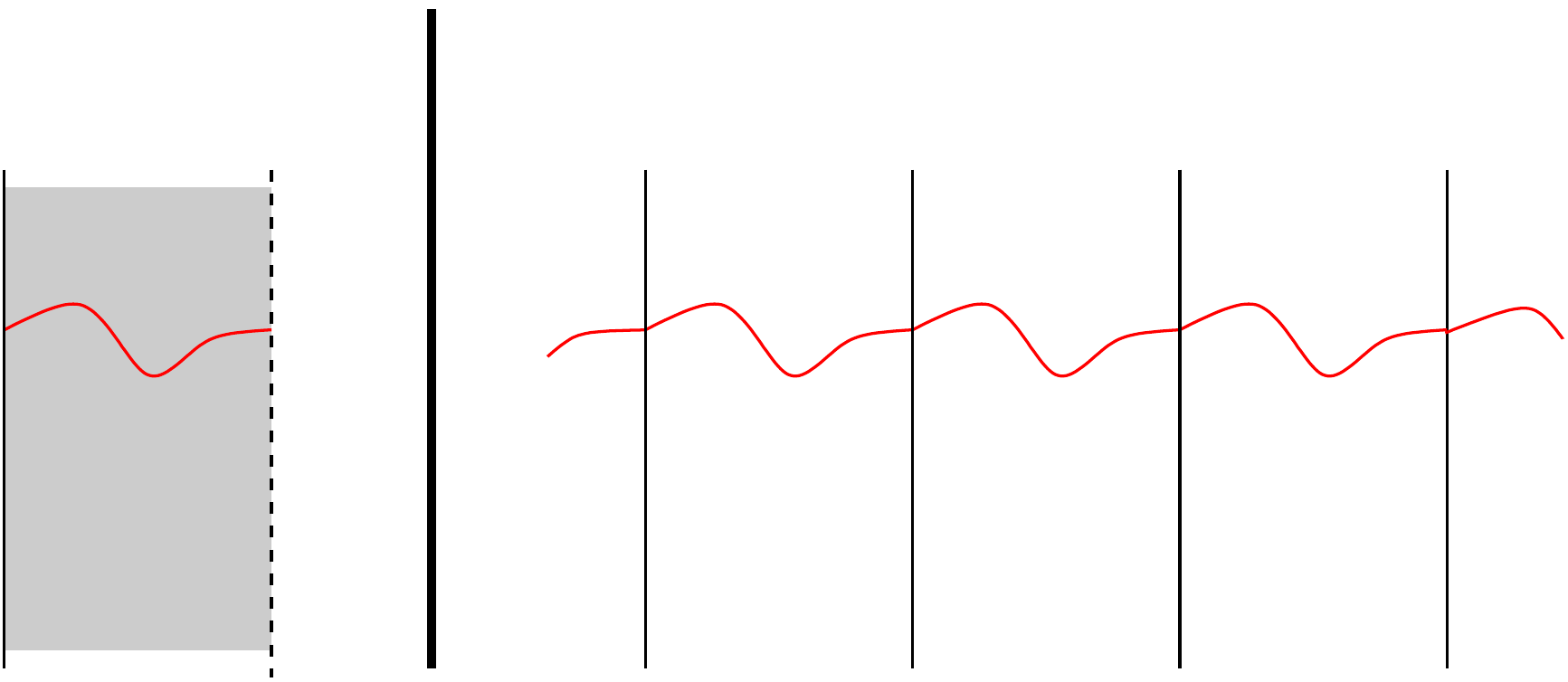}
  \caption{Left: Three representations of $\cS\setminus\{0, \infty\}$. Right:  The universal covering $\widehat{\s}$ of  $\mathcal S\setminus\{0, \infty\}$ is the complex plane.}
  \label{fig:informaluniversalcovering}
\end{figure}

More precisely, let us define the map $\lambda$ by
\[
  \begin{array}{lrcl}
    \lambda: & \widehat{\s}\equiv\C & \longrightarrow &   \C^* \equiv \s \setminus \{0,\infty\}
    \\
             & \omega & \longmapsto &  \lambda(\omega):=e^{i\omega}.
  \end{array}
\] 
This is a $2\pi$-periodic, non-branching covering map from $\C$ to $\mathcal S\setminus\{0,\infty\}$.
Every segment of the form $[a+ib, a+2\pi+ib]$, with $a$ and $b$ real,
is projected onto a closed curve of $\s \setminus \{0,\infty\}$ homologous to a curve going around the cylinder.

Given $s\in\s \setminus\{0, \infty\}$ and $S \subset \s\setminus\{0, \infty\}$, we will use the notation $\widehat{s}$ and $\widehat{S}$ for  their preimages by $\lambda$ in some prescribed vertical strip of width $2\pi$ (which is often taken to be  $\{ 0 \leqslant\Re \omega <2\pi\}$, but not always).
In particular, given that  $\lambda(\pi+\theta)= s_0=-e^{i\theta}$,
$\lambda(\pi)=s_1^-{=-1}$,
$\lambda(\pi+\beta)=s_2^-{=-e^{i\beta}}$ (see Lemma~\ref{lem:s0} and~\eqref{encombre}),  we will write:
\[
  \hat s_0=\pi+\theta , 
  \qquad\hat{s}_1^-=\pi, 
  \qquad\hat{s}_2^-=\pi+\beta.
\]

Every conformal automorphism $\chi$ of $\s \setminus \{0,\infty\}$ may be lifted
to a conformal automorphism $\widehat{\chi}=\lambda^{-1}\chi\lambda$ of the universal covering $\C$. The function $\lambda^{-1}$ being multivalued, this continuation is uniquely defined if we fix the image by $\widehat{\chi}$
of a given point $\omega_0\in\C$.

Recall the definitions~\eqref{eq:elements_group} of the maps $\xi$, $\eta$ and $\zeta$.  Recall in particular that $\xi$ fixes $1$ and $-1$, while $\eta$ fixes $\pm e^{i\beta}$. Let us define $\wxi$ (resp.\ $\weta$) by choosing its fixed point to be $\hat{s}_1^-=\pi$ (resp.\ $\hat{s}_2^-=\pi+\beta$). 
Using~\eqref{eq:elements_group} we have
\beq\label{wxi-def}
\wxi(\omega)=-\omega+2\pi \quad \text{and}\quad\weta(\omega)=-\omega+2(\pi+\beta).
\eeq
These are central symmetries of respective centers $\hat{s}_1^-$ and $\hat{s}_2^-$. It follows that $\weta\,\wxi$ and $\wxi\,\weta$ are just translations by $2\beta$ and $-2\beta$:
\beq\label{eta-xi-hat}
\weta\,\wxi(\omega)=\omega+2\beta\quad \text{and}\quad
\wxi\,\weta(\omega)=\omega-2\beta.
\eeq

\subsection{Where is $\Re x(e^{i\om})$ negative?}
\label{lifted-aut}

The initial domain of definition of the Laplace transform $\phi_2(x)$ is $\{x\in\C: \Re x\leqslant 0\}$.
Returning to the uniformization~\eqref{eq:uniformization} of the curve $\gamma(x,y)=0$ by the variable $s$, we 
define $\Delta_1 := \{s\in\s \setminus\{0, \infty\}:\Re x(s) \leqslant 0 \}$ and we introduce its lifting
\beq\label{Delta1-def}
\wD_1   := \{\omega\in\C : 
0\leqslant \Re \omega <2\pi  \text{ and } \Re x(\lambda\omega) \leqslant   0
\}.
\eeq
The goal of this subsection is to study  this lifted convergence domain.
We denote by $\wmI_{x}$ the curve where the value $x(e^{i\omega})$ is purely imaginary:
\[
  \wmI_{x}:=\{ \omega \in \C :  0 \leqslant\Re \omega <2\pi \text{ and } 
  \Re x(\lambda \omega) =0\}.
\]
The following lemma is illustrated in Figure~\ref{fig:universalcovering}, which we have completed by more examples in Figure~\ref{fig:curvesI-cases}.   

\begin{lem} \label{lem:Delta1}
  The curve $\wmI_{x}$ consists of two connected branches, with vertical asymptotes at $\Re \omega \in \{\pi/2, 3\pi/2\}$.
  Denoting  $\omega=u+iv$, with $u$ and $v$ real, these branches lie in two disjoint  vertical strips and are defined  by the equation:
  \beq\label{curveI-equation}
  \cosh v= -\frac{\cos \theta}{\cos u} ,
  \quad   \text {for }
  \begin{cases}
    u \in (\pi/2, \pi-\theta] \cup [\pi+\theta, 3\pi/2)& \text {if } \theta < \pi/2,\\
    u \in [\pi-\theta, \pi/2) \cup ( 3\pi/2, \pi+\theta]& \text {if } \theta > \pi/2.
  \end{cases}
  \eeq
  The case $\theta=\pi/2$ is degenerate, with $\wmI_{x}$ consisting of two vertical lines at abscissas $\pi/2$ and~$3\pi/2$.
  
  We denote by  $\wmI_{x}^-$ the rightmost branch, which  goes through $\hat s_0= \pi+\theta$, and by  $\wmI_{x}^+$ the leftmost one, which goes through $\hat s_0':=\wxi( \hat s_0)=\pi-\theta$.   The automorphism $\wxi$ defined in~\eqref{wxi-def} exchanges the branches $\wmI_{x}^+$ and~$\wmI_{x}^-$. 
  The notation $\wmI_{x}^\pm$ comes from the fact that $\Re y(\lambda\omega)$ is positive (resp.\ non-positive) on $\wmI_{x}^+$ (resp.\ $\wmI_{x}^-$).
  Finally,  the domain $\wD_1$   lies between the two branches of~$\wmI_x$.
  \label{lem:equationI-curve}
\end{lem}

\begin{proof}
  We work with the normal form~\eqref{param:normal} of the parametrization. It follows from the expression of $\xn(s)$ that, for  $\omega=u+iv$, 
  \beq\label{x-omega}
  \Re\, \xn(\lambda \om) = \cos \theta +\cos u \cosh v .
  \eeq
 An elementary study then   establishes the description~\eqref{curveI-equation}. 
 
 Since $x(\xi s)=x(s)$, the curve $\wmI_{x}$ is fixed by the automorphism $\wxi$. Since $\wxi$ swaps the two vertical strips that contain  the  branches of $\wmI_{x}$, we conclude that it exchanges these two branches.

  Let us now justify the notation $\wmI_{x}^\pm$. First, we derive for $\yn(\lambda \omega)$  the following counterpart of~\eqref{x-omega}: 
  \[
    \Re \,\yn(\lambda \omega)= \cos(\beta-\theta)+ \cos(u-\beta)\cosh v.
  \]
  By combining this equation with~\eqref{curveI-equation}, we see that
  on the curve $\wmI_x$, the value $\Re \yn(\lambda \omega)$ has the sign of
  \[
    \cos(\beta-\theta)- \cos(u-\beta) \frac {\cos\theta}{\cos u}=\sin \beta\, \frac{ \sin(\theta-u)}{\cos u},
  \]
  that is, the sign of   $\sin(\theta-u)/\cos u$. The result follows by considering separately the two vertical strips of~\eqref{curveI-equation} and the 
  three cases $\theta< \pi/2$,  $\theta> \pi/2$ and $\theta=\pi/2$.

 The point $\om=\pi$ always lies between the two branches of $\wmI_x$. Given that $x(\lambda\pi)=x(-1)=x^-<0$, we conclude that  the domain $ \wD_1 $ is the area lying between the two branches of $\wmI_x$.
\end{proof}

\begin{figure}[htbp]
  \centering
  \includegraphics[scale=2,trim= 8mm 5mm 8mm 5mm,clip]{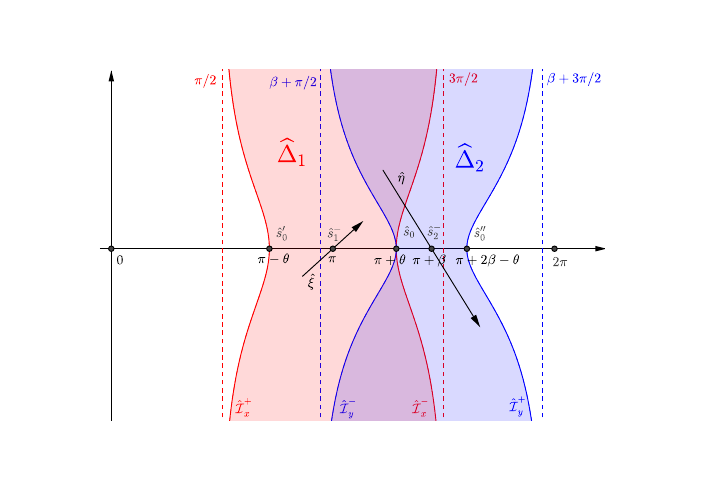}
  \caption{ In red (left), the domain $\wD_1$, and its boundary $\wmI_x$ shown with its asymptotes. In blue (right), $\wD_2$ and its boundary.
    In this example, $\theta<\pi/2$ and $ \beta-\theta<\pi/2$. The cases  ($\theta>\pi/2$, $\beta-\theta<\pi/2$) and  ($\theta<\pi/2$, $\beta-\theta>\pi/2$) are shown in Figure~\ref{fig:curvesI-cases}. Given that $\beta<\pi$, it is not possible to have $\theta\ge\pi/2$ and $\beta-\theta\ge\pi/2$.}
  \label{fig:universalcovering}
\end{figure}

\begin{figure}[htb]
  \centering
  \includegraphics[width=60mm,trim= 8mm 5mm 8mm 5mm,clip]{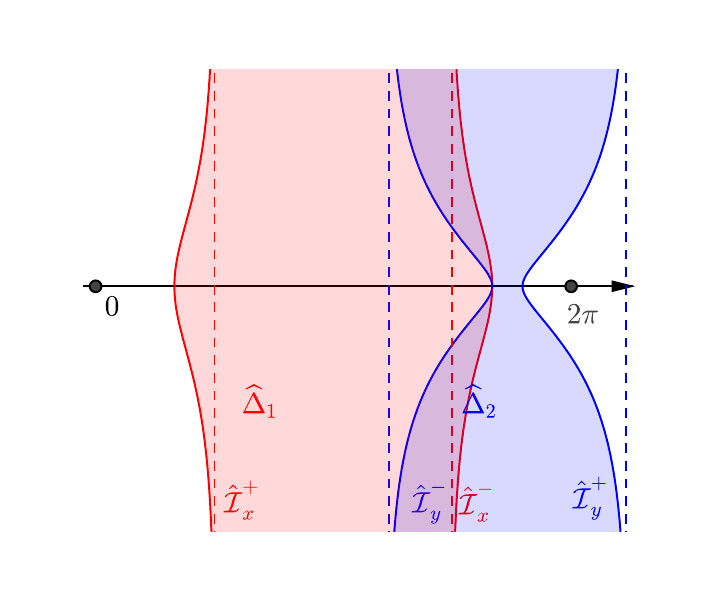} \hskip 10mm
  \includegraphics[width=60mm,trim= 8mm 5mm 8mm 5mm,clip]{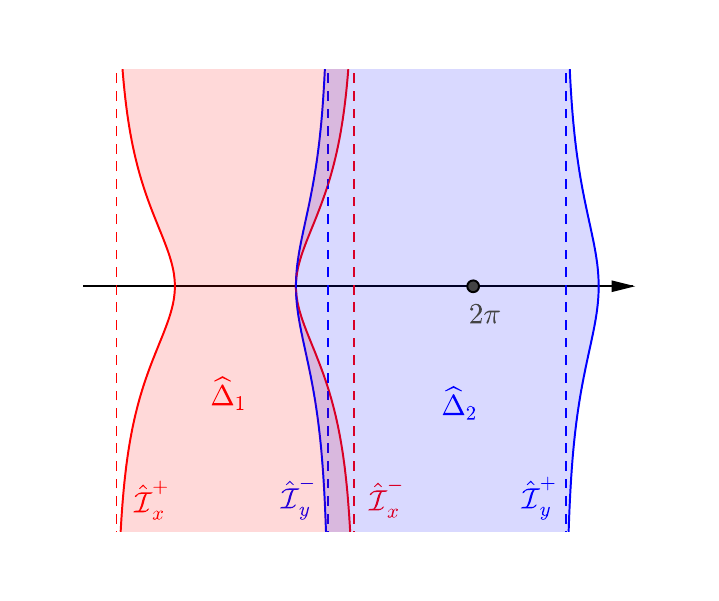}
  \caption{Some examples of the domains $\wD_1$ (red) and $\wD_2$ (blue),
    for various values of $\theta$ and $\beta$: left,    $\theta>\pi/2$ and $\beta-\theta<\pi/2$; right, $\theta<\pi/2$ and $\beta-\theta>\pi/2$. }
  \label{fig:curvesI-cases}
\end{figure}

We now want to determine where $\Re y(e^{i\om}) \le 0$. We first define $\Delta_2 = \{s\in\s\setminus\{0, \infty\}:\Re y(s) \leqslant 0 \}$. Given that $\yn(\lambda\om)$ is obtained from $\xn(\lambda (\om-\beta))$ by replacing $\theta$ by $\beta-\theta$ (see~\eqref{param:normal}), it makes sense to consider the following lifting of $\Delta_2$:
\beq\label{Delta2-def}
\wD_2
:= \{\omega\in\C : \beta\leqslant \Re \omega <\beta+2\pi  \text{ and } \Re y(\lambda\omega) \leqslant   0
\}.
\eeq
 We define $\wmI_{y}$ as~$\wmI_{x}$, but again in the translated strip:
\[
  {    \wmI_{y}:=\{ \omega \in \C :  \beta \leqslant\Re \omega <\beta +2 \pi \text{ and } \Re y(\lambda\omega) =0  \}.}
\]

The counterpart of Lemma~\ref{lem:equationI-curve} reads as  follows; see again Figures~\ref{fig:universalcovering} and~\ref{fig:curvesI-cases} for various illustrations.

\begin{lem}\label{lem:Delta2}
  The curve $\wmI_{y}$ consists of two branches, with vertical asymptotes at $\Re \omega \in \{\beta+\pi/2,\beta+3\pi/2\}$. Denoting $\omega=u+iv$, with $u$ and $v$ real, these two branches lie in two disjoint vertical strips and are defined by:
  \[
    \cosh v= -\frac{\cos (\beta-\theta)}{\cos( u-\beta)} ,
    \quad   \text {for }
    \begin{cases}
      u \in (\beta+ \pi/2, \pi+\theta] \cup [\pi+2\beta-\theta, \beta + {3\pi}/{2})& \text {if } \beta-\theta < \pi/2,\\
      u \in [\pi+\theta, \beta + \pi/2) \cup ( \beta+3\pi/2, \pi+2\beta-\theta]& \text {if } \beta-\theta > \pi/2.
    \end{cases}
  \]
  The case $\beta -\theta=\pi/2$ is degenerate, with $\wmI_{y}$ consisting of two vertical lines at abscissas $\beta+\pi/2$ and $\beta+3\pi/2$.
  The curve $\wmI_{y}$ is fixed by the automorphism $\weta$, which swaps its two branches. We denote by  $\wmI_{y}^-$ the leftmost branch, which  goes through $\hat s_0= \pi+\theta$, and by  $\wmI_{y}^+$ the rightmost one, which goes through $\hat s_0'':=\weta( \hat s_0)=\pi+2\beta-\theta$. 
  The value of  $\Re x(\lambda \omega)$ is positive on $\wmI_{y}^+$ and  non-positive on $\wmI_{y}^-$.
  The domain $\wD_2$ lies between $\wmI_{y}^-$  and $\wmI_{y}^+$. 
  \label{lem:equationI-curve2}
\end{lem} 
The proof is a straightforward adaptation of  the one of Lemma~\ref{lem:equationI-curve}.
 Finally, the following lemma  is illustrated in Figure~\ref{fig:universalcovering}, and can be checked further on  the examples of Figure~\ref{fig:curvesI-cases}.

\begin{lem}\label{lem:branches}
  The domain $\wD_2$ contains the branch $\wmI_{x}^-$, while the domain $\wD_1$ contains the branch $\wmI_{y}^-$. In particular, the set  $\wD_1 \cap \wD_2$, which is bounded on the left by $\wmI_{y}^-$ and on the right by $\wmI_{x}^-$, has a non-empty interior.
\end{lem}

\begin{proof}
  By Lemma~\ref{lem:equationI-curve},  the value of $\Re y(\lambda \omega)$ is {non-positive} in $\wmI_{x}^-$: hence this 
  branch is included in $\wD_2$ by definition~\eqref{Delta2-def} of this domain.
  The second property   follows similarly from Lemma~\ref{lem:equationI-curve2}.
\end{proof}

\subsection{Lifting and meromorphic continuation of \texorpdfstring{$\varphi_1$}{phi1} and \texorpdfstring{$\varphi_2$}{phi2} to the universal covering}

As the Laplace transform $x \mapsto \varphi_2(x)$ is analytic in the interior of  $\{x\in\C: \Re x\leqslant 0\}$ and continuous on its boundary,  we can lift it
to the set $\wD_1 $ defined by~\eqref{Delta1-def} by setting:
\[
  \psi_2(\omega) :=  \varphi_2(x(e^{i\omega})), \quad \forall\omega\in\wD_1.
\]
Analogously, we define the lifting of $\phi_1$ by
\beq \label{psi1} 
\psi_1(\omega)
:= \varphi_1(y(e^{i\omega})), \quad \forall\omega\in\wD_2. 
\eeq 
These maps are analytic in the interiors of their domains $\wD_1$ and $\wD_2$, and continuous on the boundaries of these domains. 
For $\omega$ in $\wD_1 \cap \wD_2$ (which is non-empty by Lemma~\ref{lem:branches}), the main functional equation~\eqref{eq:functional_equation} yields
\beq
\label{eq:func-eq-universal-cov}
\gamma_1(x(\lambda\omega),y(\lambda\omega))\psi_1(\omega)+
\gamma_2(x(\lambda\omega),y(\lambda\omega))\psi_2(\omega)=0,
\eeq
where we recall that $\lambda \om$ stands for $e^{i\om}$.

We can now  extend meromorphically $\psi_1$ and $\psi_2$ to the interior of $\wD:=\wD_1 \cup \wD_2$ by means of the formulas 
\begin{align}
  \psi_1(\omega)&=-\frac{\gamma_2}{\gamma_1}(x(\lambda\omega),y(\lambda\omega)) \psi_2(\omega)
                  \quad\text{if } \omega\in  \wD_1, \label{psi1-ext}\\ 
  \psi_2(\omega)&=-\frac{\gamma_1}{\gamma_2}(x(\lambda\omega),y(\lambda\omega)) \psi_1(\omega)
                  \quad\text{if } \omega\in  \wD_2, \label{psi2-ext}
\end{align} 
see~\cite[Lem.~3]{franceschi_explicit_2017} or~\cite[Lem.~6]{franceschi_asymptotic_2016}. Note that~\eqref{eq:func-eq-universal-cov} guarantees that the values of $\psi_1$ given by~\eqref{psi1} and~\eqref{psi1-ext} actually coincide on $\wD_1 \cap \wD_2$. A similar statement holds for $\psi_2$. We finally extend $\psi_1$ and $\psi_2$ to  the boundary of $\wD$ by continuity.  The fact that this extension is only meromorphic, rather than analytic, comes from the divisions by $\gamma_1$ and $\gamma_2$, which may create poles. The functional equation~\eqref{eq:func-eq-universal-cov} now  holds on the whole interior of $\wD$, and also on its boundary by continuity.

In order to extend $\psi_1$ and $\psi_2$ to $\C\equiv \widehat \s$, we will need the following  lemma, which states in particular that the complex plane is completely covered by translates of the set $\wD$ by shifts of~$2\beta$. We recall that a translation by $2\beta$ is precisely the effect of $\weta \,\wxi$ (see~\eqref{eta-xi-hat}).

\begin{lem} \label{lemrev}
  Recall that $\wD=\wD_1\cup\wD_2$, where $\wD_1$ and $\wD_2$ are defined by~\eqref{Delta1-def} and~\eqref{Delta2-def} respectively. The set $\wD$ is bounded by $\wmI_{x}^+$ (on the left) and $\wmI_{y}^+$ (on the right). Moreover, $\wmI_{x}^+ +2\beta \subset \wD_2$ and $\wmI_{y}^+-2\beta \subset \wD_1$. This implies that
  \[
    \C= \underset{n\in\Z}{\bigcup} (\wD +2n\beta)
    =\underset{n\in\Z }{\bigcup} (\weta\,\wxi)^n \wD.
  \]
  {Moreover, $\wD \cap \weta \wD= \wD_2$, and $\wD \cap \wxi                   \wD= \wD_1$.}
\end{lem}

\begin{proof}
  {The first statement follows from the discussion in Section~\ref{lifted-aut}}, see Figures~\ref{fig:universalcovering} and~\ref{fig:curvesI-cases}.
  Now
  \begin{align*}
    \wmI_{x}^+ +2\beta &= \weta \wxi \wmI_{x}^+\\
                       & = \weta \wmI_{x}^- \qquad \text {since } \wxi \text{ exchanges the two branches of } \wmI_x \text{ (Lemma~\ref{lem:Delta1})}\\
                       &\subset \weta \wD_2 \qquad \text {since } \wmI_{x}^- \subset \wD_2\text{ (Lemma~\ref{lem:branches})}\\
                       &\subset \wD_2  \qquad \hskip 2mm \text {since } \wD_2 \text { is left invariant by } \weta \text{ (Lemma~\ref{lem:Delta2})}.
  \end{align*}
  A similar argument proves that $\wxi\, \weta \wmI_{y}^+ = \wmI_{y}^+ -2\beta
  \subset \wD_1$.
  These two   properties are illustrated in Figures~\ref{fig:shiftdelta} and~\ref{fig:shift-cases}. They imply that $\C$ can be covered by translates of $\wD$ by multiples of $2\beta$.

  Let us now prove that $\wD \cap \weta \wD= \wD_2$. Since $\weta \wD_2= \wD_2$, we have $\wD_2\subset \wD \cap \weta \wD$. Now  take $\omega \in \wD_1 \setminus \wD_2$. This point thus lies on the left of the curve $\wmI_y^-$. Then $\weta \omega$ lies on the right of $\weta \wmI_y^-= \wmI_y^+$, and thus cannot be in $\wD$.  Equivalently, $ \om$ cannot be in $\weta \wD$, and thus $\wD \cap \weta \wD$ is reduced to $\wD_2$. The final statement of the lemma is proved  similarly.
\end{proof}

\begin{figure}[htbp]
  \centering
  \includegraphics[scale=2,trim= 8mm 5mm 8mm 5mm,clip]{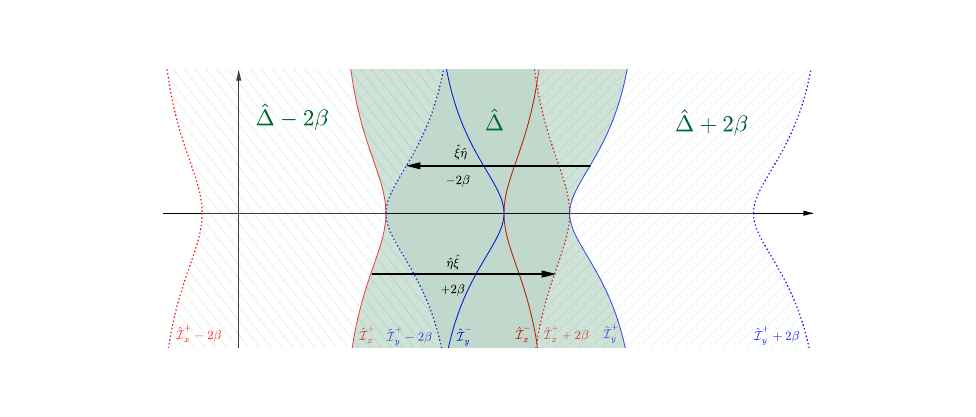}
  \caption{The set $\wD$, in grey/green, and its translates by $\pm 2\beta$ (dashed areas). Here $\theta<\pi/2$ and $\beta-\theta<\pi/2$, as in Figure~\ref{fig:universalcovering}. The other two cases are illustrated in Figure~\ref{fig:shift-cases}.}
    \label{fig:shiftdelta}
\end{figure}

\begin{figure}[htb]
  \centering
  \includegraphics[scale=1.2,trim= 6mm 5mm 6mm 5mm,clip]{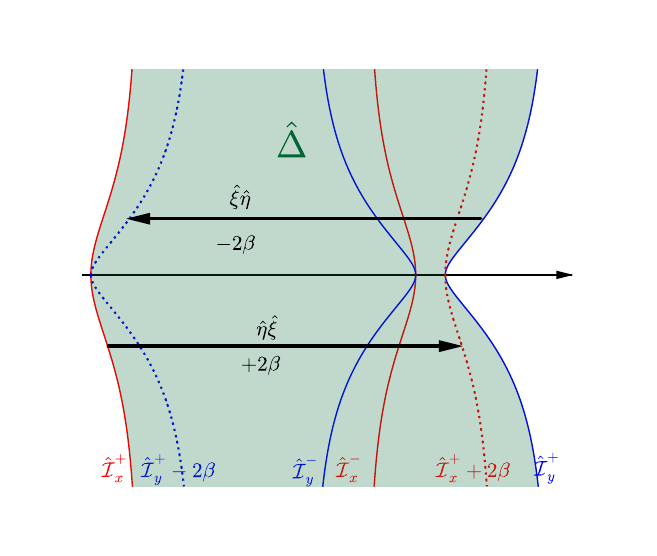}  \hskip 10mm
  \includegraphics[scale=1.2,trim= 6mm 5mm 6mm 5mm,clip]{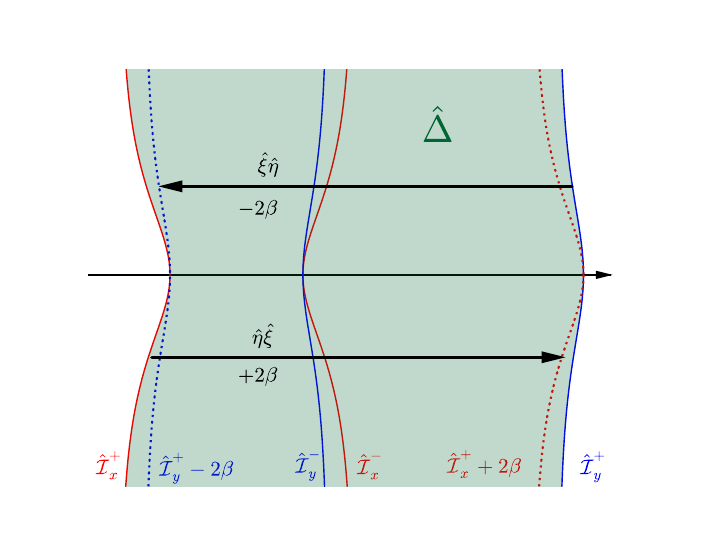}
  \caption{The set $\wD$ and, in dotted lines, the translates of the branches  $\wmI_x^+$ (red) and $\wmI_y^+$ (blue) by $2\beta$ and $-2\beta$, respectively. In all cases, these translated branches fit in the domain $\wD$. On the left,
      $\theta>\pi/2$ and $\beta-\theta<\pi/2$; on the right, $\theta<\pi/2$ and $\beta-\theta>\pi/2$. }
  \label{fig:shift-cases}
\end{figure}

 To lighten notation in the functional equation~\eqref{eq:func-eq-universal-cov}, we will denote 
\[
  \wgamma_1(\omega):=\gamma_1
  (x(\lambda\omega),y(\lambda\omega))\quad \text{and}\quad\wgamma_2(\omega):=\gamma_2
  (x(\lambda\omega),y(\lambda\omega)).
\]

\begin{thm} \label{thmpro}
  The functions $\psi_1$ and $\psi_2$, which are so far defined on the interior of $\wD$,  can be continued meromorphically to the whole of $\C$. For all $\omega\in\C$, these functions satisfy
  \allowdisplaybreaks \begin{numcases}{}
    \wgamma_1(\omega)\psi_1(\omega)+\wgamma_2(\omega)\psi_2(\omega)
    =0, & \text{(functional equation)} 
    \label{equa} \\
    \psi_1(\weta\omega)={\psi_1(2\pi+2\beta-\omega)}=\psi_1(\omega) , & \text{(invariance of 
      $\psi_1$ by $\weta$)} 
    \label{inveta} \\
    \psi_2(\wxi\omega)={\psi_2(2\pi-\omega)}=\psi_2(\omega),  & \text{(invariance of 
      {$\psi_2$} by $\wxi$)} 
    \label{invxi}\\
    \psi_1(\weta\wxi\omega) ={ \psi_1(\omega+2\beta)}
    =\dfrac{\wgamma_2}{\wgamma_1}(\wxi\omega) 
    \dfrac{\wgamma_1}{\wgamma_2}(\omega) \psi_1 (\omega), & \text{
      (shift of $2\beta$)} 
    \label{pro+trans}\\
    \psi_2 (\wxi\weta\omega) ={ \psi_2(\omega-2\beta)}
    =\dfrac{\wgamma_1}{\wgamma_2}(\weta\omega)
    \dfrac{\wgamma_2}{\wgamma_1}(\omega) \psi_2 (\omega), &  \text{
      (shift of $-2\beta$)}. \label{pro-trans}
  \end{numcases}
\end{thm}

Note that~\eqref{pro+trans} can be rewritten as
\begin{align}
  \psi_1(\omega+ 2\beta)& = \frac{\gamma_2(x(e^{-i\omega}), y(e^{-i\omega}))}{\gamma_1(x(e^{-i\omega}), y(e^{-i\omega}))}
                          \frac{\gamma_1(x(e^{i\omega}), y(e^{i\omega}))}{\gamma_2(x(e^{i\omega}), y(e^{i\omega}))}  \psi_1(\omega),\nonumber
  \\
                        & = E(e^{{i\omega}}) \psi_1(\omega),\label{prop-psi1-formula}
\end{align}
where $E(s)$ is defined by~\eqref{eq:definition_E_V0}. This is the formula announced in Proposition~\ref{prop:psi1}. 

\begin{proof}[Proof of Theorem~\ref{thmpro}]
  We have constructed $\psi_1$ and $\psi_2$  meromorphically inside
  $\wD$. Our first task will be to prove the identities~\eqref{equa}--\eqref{pro-trans} where they are well defined. We have already seen  that the functional equation~\eqref{equa} holds in $\wD$, by construction of $\psi_1$ and $\psi_2$ (see~\eqref{psi1-ext} and~\eqref{psi2-ext}).
 
  Let us now prove that the invariance formula~\eqref{inveta} holds where it is well defined, that is, for $\omega \in\wD \cap \weta\wD =\wD_2$ (Lemma~\ref{lemrev}).   For $\omega \in \wD_2$, the value $\psi_1(\omega)$ only depends on $y(\lambda \omega)$ (see~\eqref{psi1}), and $y\circ \lambda$ is invariant by $\weta$. Hence~\eqref{inveta} holds. In the same way, we prove that~\eqref{invxi} holds for  $\omega \in\wD \cap \wxi\wD = { \wD}_1$.

  Let us now  establish the translation formula~\eqref{pro+trans} where it is well defined, that is,   for $\omega \in\wD\cap \wxi\weta\wD = \wD\cap (\wD -2 \beta)$. This is the area located between the branches $\wmI_x^+$ and $\wmI_y^+-2\beta$ (see Figures~\ref{fig:shiftdelta} and~\ref{fig:shift-cases}). In particular, $\omega\in \wD_1$ (Lemma~\ref{lemrev}), hence $\wxi\omega\in {\wxi \wD_1} = \wD_1$ as well. We have:
  \begin{align*}
    \psi_1(\weta\wxi\omega) &= \psi_1(\wxi\omega) & \text{by~\eqref{inveta}, given that } \weta\wxi\omega \in \wD \text{ and } \wxi \omega \in  \wD,\\
                            &=- \frac{\wgamma_2(\wxi\omega)}{\wgamma_1(\wxi\omega)} \psi_2 ({\wxi}\omega)& \text{by~\eqref{equa}, given that } \wxi \omega \in  \wD,\\
                            &=- \frac{\wgamma_2(\wxi\omega)}{\wgamma_1(\wxi\omega)} \psi_2 (\omega)& \text{by~\eqref{invxi}, given that } \wxi \omega \in  \wD\text{ and } \omega \in  \Delta,\\
                            &=\frac{\wgamma_2(\wxi\omega)}{\wgamma_1(\wxi\omega)} 
                              \frac{\wgamma_1(\omega)}{\wgamma_2(\omega)} \psi_1 (\omega) & \text{by~\eqref{equa}, given that } \omega \in  \wD.
  \end{align*}
  This completes the proof of~\eqref{pro+trans}.  We prove~\eqref{pro-trans} in a similar fashion for $\omega \in\wD\cap \weta\wxi\wD$.
  
  \medskip
  We will now continue  $\psi_1$ meromorphically on successive translates of $\wD$ by multiples of $\pm 2\beta$, using the translation formulas~\eqref{pro+trans} and~\eqref{pro-trans}. Let us first define $\psi_1$ on $\weta{\wxi} \wD= \wD +2\beta$. 
  For $\omega \in   \wD $   we set 
  \[
    \psi_1(\omega+2\beta) =\psi_1(\weta{\wxi}\omega) 
    :=
    \frac{\wgamma_2(\wxi\omega)}{\wgamma_1(\wxi\omega)} 
    \frac{\wgamma_1(\omega)}{\wgamma_2(\omega)} \psi_1 (\omega).
  \]
   Since~\eqref{pro+trans} holds on $\wD\cap \wxi \weta\wD$,  this is consistent with the already defined  values of $\psi_1$.  We thus obtained a  meromorphic extension of $\psi_1$ on $\wD\cup \weta\wxi\wD$ (recall that $\wD\cap \weta{\wxi}\wD=\wD \cap (\wD   +2\beta)$ has a non-empty interior).
  Let us now define $\psi_1$ on ${\wxi}\weta \wD=\wD-2\beta$.  For 
  $\omega \in  \wD-2\beta$, we set:
  \[
    \psi_1 (\omega):=
    \frac{\wgamma_1(\wxi\omega)}{\wgamma_2(\wxi     \omega)} 
    \frac{\wgamma_2(\omega)}{\wgamma _1(\omega)} 
    \psi_1(\weta{\wxi}\omega) .
  \]
  Again, the fact that~\eqref{pro+trans} holds on $\wD\cap \wxi\weta\wD$ guarantees that we have indeed an   extension of~$\psi_1$.

  With the same translation procedure we now propagate the construction of $\psi_1$ to 
  \[
    ({\wxi}\weta)^n \wD \cup\cdots \cup{\wxi}\weta \wD \cup \wD \cup \weta{\wxi} \wD \cup\cdots \cup (\weta {\wxi   })^n \wD
    =
    \underset{k\in \llbracket  -n, n  \rrbracket}{\bigcup} (\wD    +2k\beta)
  \]
  for all $n\in\mathbb{N}$.
  Lemma~\ref{lemrev} guarantees that we finally cover the whole of $\C$.

  We continue $\psi_2$ meromorphically to $\C$ using a similar procedure, based now on~\eqref{pro-trans}.
  The principle of analytic/meromorphic continuation implies that the equations~\eqref{equa}--\eqref{invxi}  are satisfied on the whole of~$\C$. 
\end{proof}

In Section~\ref{subsec:funcequ}, we defined a lifting of $\phi_1$ to the $s$-plane by $\widetilde{\phi}_{1}(s):=\phi_1(y(s))$.  first in the set $\Delta_2:=\{s: \Re y(s)\le 0\}$.
\begin{cor}
  \label{cor:meromorphy_phi_1_tilde}
  The function $\widetilde{\phi}_{1}$  may be continued meromorphically on the slit plane $\C\setminus  e^{i\beta} \RR_+$.
\end{cor}

\begin{proof}
  Let us observe  from~\eqref{psi1}  that $\widetilde{\phi}_{1} (e^{i\omega})=\psi_1(\omega)$ for $\omega\in \wD_2$. The domain $\wD_2$ contains in particular a neighbourhood of $\pi+\beta +i\RR$.  Let $\log$ be the determination of the complex logarithm in the slit plane $\C\setminus e^{i\beta}\,\RR_+$ 
  that satisfies $\log(e^{i\omega})=i\omega$ when $\Re \omega= \pi+\beta$.
  In a neighbourhood of $\arg s=\pi+\beta$, we now have
  \[
    \widetilde{\phi}_{1} (s)=\psi_1(- i \log s).
  \] 
  But    Theorem~\ref{thmpro} states that $\psi_1$ can be continued meromorphically to $\mathbb{C}$.  Then the above formula   allows us to continue $\widetilde{\phi}_{1}$  meromorphically to $\C\setminus e^{i\beta}\,\RR_+$.
\end{proof}

\section{The case $\alpha_1=\alpha_2=0$: the inverse Laplace transform (arXiv version only)}
\label{app:alg-inverse}

In this section we re-prove Proposition~\ref{prop:specialcase}, which gives the expression of the density of the stationary distribution of the SRBM when $\alpha_1=\alpha_2=0$, this time by inverting the Lapace transform $\phi(x,y)$.

\medskip\paragraph{{\bf Expression of $\phi(x,y)$.}} Our first task is to establish the expression~\eqref{phi-alg} for the bivariate transform $\phi(x,y)$. It is obtained by combining:
\begin{itemize}
\item the functional equation~\eqref{eq:functional_equation}, which expresses $\phi$ in terms of $\gamma_1 \phi_1 + \gamma_2 \phi_2$,
\item the values of $\phi_1(y)$ and $\phi_2(x)$ that we have obtained in~\eqref{phi1-very-simple} and~\eqref{phi2-very-simple},
\item the expressions~\eqref{masses} and~\eqref{gamma_i:normal} of $\phi_i(0)$ and $\gamma_i(x,y)$, for $i=1,2$,
\item the angle relations~\eqref{angles-00}, which will be used repeatedly in our calculations below. 
\end{itemize}
We work with the normal variables $\xn$ and $\yn$, observing that $\tx:=\sqrt{1-x/x^+}=\sqrt{1-\xn/\xn^+}$ and analogously for $\ty $. The kernel $\gamma(x,y)$, written  in terms of $\tx$ and $\ty $,  factors as:
\[
  \gamma(x,y)= \kappa\, D(\tx ,\ty ) \widetilde D(\tx, \ty ),
\]
where $\kappa$ does not involve $\tx$ and $\ty $, 
\[
  D(\tx,\ty ):=\tx^2\sin^2\delta+\ty ^2\sin^2\vareps-2\tx \ty \sin\delta\sin \vareps \cos(\delta+\vareps)-\sin^2(\delta+\vareps),
\]
and $\widetilde D(\tx, \ty )$ is obtained by changing the sign of the term in $\tx \ty $ in the above expression. Then the factor $\widetilde D(\tx, \ty )$ is seen  to simplify with the numerator of $\gamma_1 \phi_1 + \gamma_2 \phi_2$. The calculations can be done by hand, but we have used {\sc Maple} to be on the safe side, and to share our session with the readers; it is available on the first author's \href{https://www.labri.fr/perso/bousquet/publis.html}{webpage}~\cite{mbm-web}.

\medskip
\paragraph{{\bf Domain of analyticity.}} Clearly, $\phi(x,y)$ is meromorphic away from the cuts $x\in [x^+, +\infty)$ and $y\in [y^+, +\infty)$, but we need  to determine where it can be defined analytically. 
Observe that the denominator $D(\tx,\ty )$ satisfies $D(1,1)=-4\sin\delta \sin \vareps \cos(\delta+\vareps)>0$ (because of the angle conditions~\eqref{conds-new}), so that  $\phi$ is analytic in a neighborhood of $(x,y)=(0,0)$ (corresponding to $(\tx,\ty )=(1,1)$).  Let us first focus on real values of  $(x,y)$, and examine the domain
\beq\label{domain-def}
\mathcal D:=  \left\{(x,y )\in (-\infty, x^+) \times (-\infty, y^+): D\left(\sqrt{1-x/x^+}, \sqrt{1-y/y^+}\right) >0\right\},
\eeq
in which $\phi(x,y)$ can  be defined as a real analytic
function. Studying $\mathcal D$  boils down to studying (now in
variables $\tx$ and $\ty $) the intersection $\mathcal E$
of the quadrant $\RR_+^2$ with the exterior of the ellipse $D(\tx,\ty )=0$. This ellipse admits a rational parametrization (in fact inherited from the parametrization~\eqref{eq:uniformization} of $\gamma(x,y)=0$):
\beq\label{D-param}
\tx \sin \delta= \frac 1 2 \left(v+ \frac 1 v\right), \qquad
\ty  \sin \vareps =  \frac 1 2 \left( \frac v{e^{i(\delta+\vareps)}} +\frac{e^{i(\delta+\vareps)}}v\right).
\eeq
The real points are obtained when $v=e^{i\om}$ is on the unit circle, so that
\[
  \tx \sin \delta= \cos \om, \qquad  \ty  \sin \vareps = \cos(\delta+\vareps-\om).
\] 
The set $\mathcal E$, and its image $\mathcal D$ by the map $(\tx, \ty )\mapsto( x^+(1-\tx^2),y^+(1-\ty ^2))$, are shown in Figure ~\ref{fig:domain} for some specific values of the angles $\delta$ and $\vareps$ (to be more precise, the second figure uses normal variables and shows the image of $\mathcal E$ by $(\tx, \ty )\mapsto( \xn^+(1-\tx^2),\yn^+(1-\ty ^2))$).

\begin{figure}[htb]
  \centering
  \includegraphics[height=40mm]{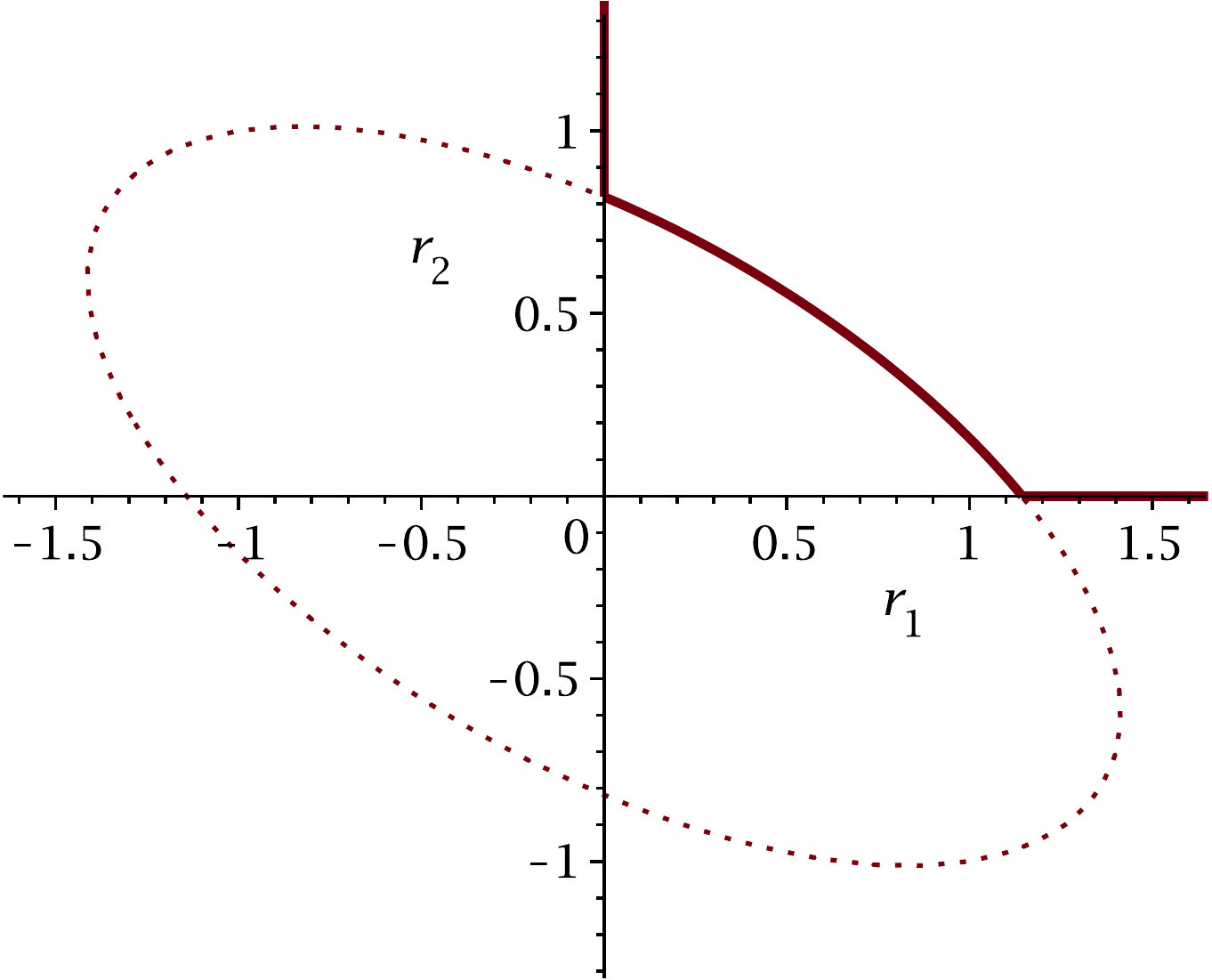}   \hskip 22mm
  \includegraphics[height=40mm]{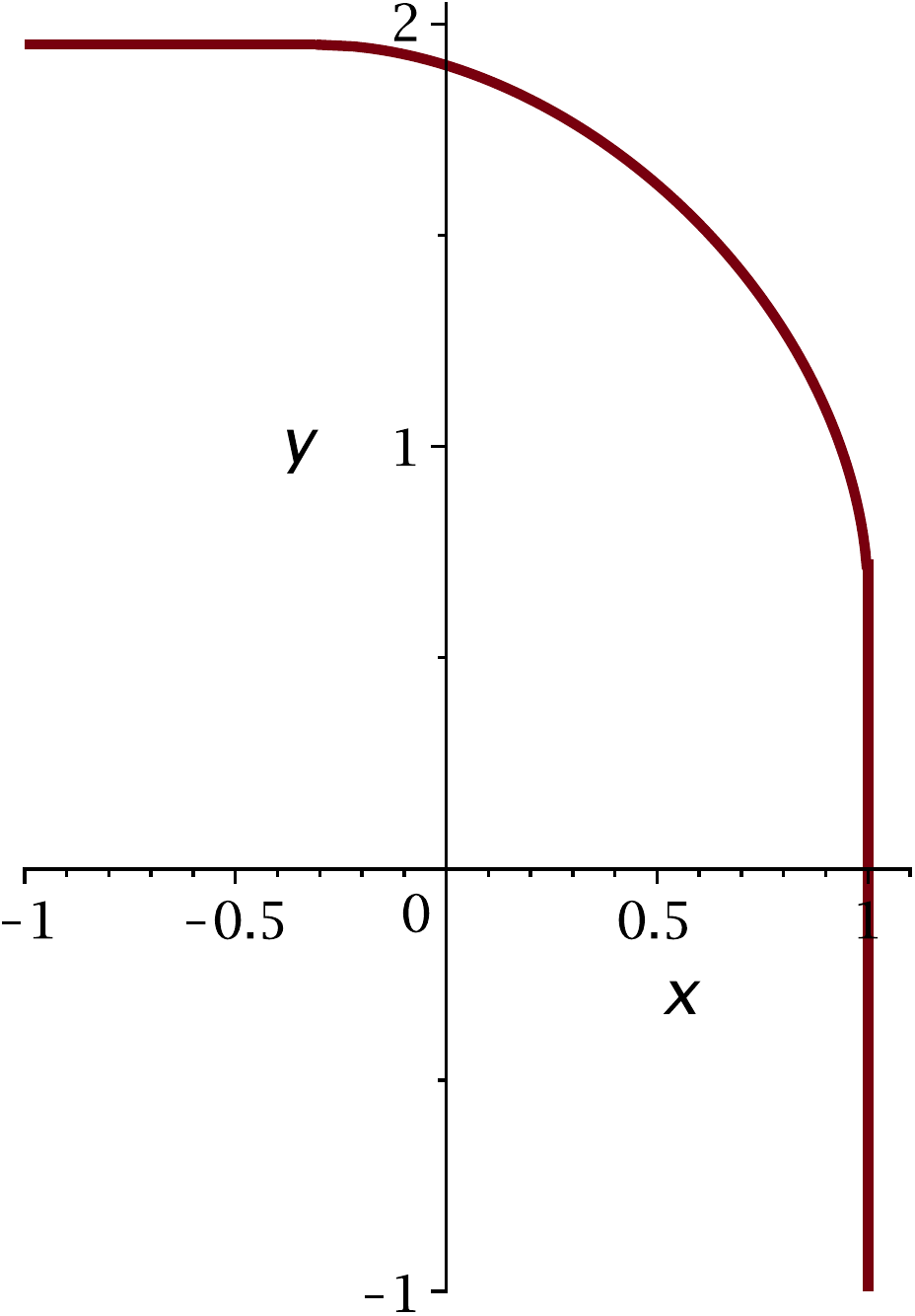}
  \caption{Left: above the thick line, the domain $\mathcal E= \{(\tx, \ty ) \in \RR_+^2: D(\tx,\ty )>0\}$, for $\delta=3\pi/4$ and $\vareps=11\pi/20$. The image of this domain by $(\tx, \ty ) \mapsto (\xn^+(1-\tx^2), \yn^+(1-\ty ^2))$ is shown on the right,  below the thick curve.}
  \label{fig:domain}
\end{figure}

\begin{lem}\label{lem:domain}
  The Laplace transform $\phi(x,y)$ is real analytic on the domain $\mathcal D$  defined by~\eqref{domain-def}, and  analytic on
  \[
    \mathcal D ^+:=
    \{(x,y) \in \C^2: (\Re(x) , \Re(y) ) \in   \mathcal D \}.
  \]
\end{lem}
\begin{proof}
  The part of the statement dealing with real values of $x$ and $y$ follows from the definition of $\mathcal D$. It remains to prove that $D(\sqrt{1-x/x^+}, \sqrt{1-y/y^+})$ does not vanish on $\mathcal D^+$. So assume that  $\Re(x)=a$ and $\Re(y)=b$, with  $(a,b)\in \mathcal D$. Write
  \[
    {1-x/x^+}  =u_1^2+iv_1,  \qquad  {1-y/y^+}=u_2^2+iv_2,
  \]
  where $u_1=\sqrt{1-a/x^+}$,  $u_2=\sqrt{1-b/y^+}$ and $v_1,v_2$ are real. Then
  \begin{multline*}
    D(\sqrt{1-x/x^+}, \sqrt{1-y/y^+})=\\
    ( u_1^2+iv_1)\sin^2\delta +(u_2^2+iv_2)\sin^2\vareps -2 \sqrt{u_1^2+iv_1}\sqrt{u_2^2+iv_2}\sin \delta \sin \vareps \cos(\delta+\vareps) -\sin^2(\delta+\vareps),
  \end{multline*}
  and has real part
  \begin{multline*}
    u_1^2\sin^2\delta +u_2^2\sin^2\vareps -2 \Re\left(\sqrt{u_1^2+iv_1}\sqrt{u_2^2+iv_2}\right)\sin \delta \sin \vareps \cos(\delta+\vareps) -\sin^2(\delta+\vareps) \\
    \ge u_1^2\sin^2\delta +u_2^2\sin^2\vareps -2 u_1u_2\sin \delta \sin \vareps \cos(\delta+\vareps) -\sin^2(\delta+\vareps)= D(u_1,u_2)>0.
  \end{multline*}
  The lower bound follows from the fact that   $ \cos(\delta+\vareps)<0$ and  $\Re\left(\sqrt{u_1^2+iv_1}\sqrt{u_2^2+iv_2}\right) \ge u_1 u_2$.
  Hence the denominator of $\phi(x,y)$ does not vanish on $\mathcal D^+$, and the lemma is proved.    
\end{proof}

\paragraph{\bf{Computation of the inverse Laplace transform by residues.}}
Let us now go back to the inverse Laplace transform. We have
\[
  p_0(z_1,z_2)= \lim_{T\rightarrow \infty}
  -  \frac 1{4\pi^2} \int_{a-iT}^{a+iT} \int_{b-iT}^{b+iT} e^{-xz_1
    -yz_2} \phi(x,y)\, dx \, dy,
\]
for $(a,b)$ in the domain $\mathcal D$.
We will  lighten notation by writing our integrals on full lines $\Re(x)=a$, $\Re(y)=b$, but this should be understood as a shortcut for the above limit. Also, in order to work with the normal variables $\xn$ and $\yn$ defined by~\eqref{xy:normal}, we consider instead the function $\widetilde p_0(z_1, z_2)$ already introduced in~\eqref{tilde-p0}:
\[ 
  \widetilde p_0(z_1, z_2)=p_0\left( \frac{\det \Sigma}{\sqrt{\Delta \sigma_{22}}} \, z_1, \frac{\det \Sigma}{\sqrt{\Delta \sigma_{11}}} \, z_2\right)
  = -   \frac{\Delta\sqrt{\sigma_{11}\sigma_{22}}}{4\pi^2 \det^2 \Sigma} \int_{\Re(\xn)=\an} \int_{\Re(\yn)=\bn} e^{-\xn  z_1 -\yn z_2} \phi_n(\xn,\yn)\, d\xn \, d\yn,
\]
where
\[
  \phi_n(\xn,\yn):=\phi(x,y)= \kappa_0 \frac {\tx+r_2}{\tx r_2D(\tx,r_2)} ,
\]
with $\tx=     \sqrt{1-\xn/\xn^+}$ and  $r_2=     \sqrt{1-\yn/\yn^+}$.  In sight of~\eqref{param:normal} and~\eqref{angles-00}, it is natural to choose
\[
  \an= {\cos \theta}   = -{\cos(2\delta)},   \qquad
  \bn= {\cos (\beta-\theta)}  =-{\cos(2\vareps)}, 
\]
but we need to check that $(\an,\bn)$ is in the domain of analyticity
of $\phi_n$.  By comparison with~\eqref{xyn-pm}, we see that  $\an<\xn^+$ and $\bn<\yn^+$, so that we are away from the cuts. Moreover, the numbers
\[
  \tx^{(0)}:=\sqrt{1-\an/\xn^+} = \frac 1 {\sqrt 2 \sin \delta} \quad \text{and } \quad 
  \ty^{(0)}:=\sqrt{1-\bn/\yn^+} = \frac 1 {\sqrt 2 \sin \vareps}
\]
satisfy $D(\tx^{(0)}, \ty^{(0)})=-\cos(\delta+\vareps)(1-\cos(\delta+\vareps)) >0$ (because of~\eqref{conds-new}). Hence the integration contour is, as it should,  in the domain of analyticity of $\phi_n$ (see Lemma~\ref{lem:domain}).

Let us now parametrize the lines $\Re(\xn)=\an$ and $\Re(\yn)=\bn$ by $\xn=\xn(iu_1)$ and $\yn=\yn(ie^{i\beta}u_2)$, where $\xn(s)$ and $\yn(s)$ are defined by~\eqref{param:normal} and  $u_1$ and $u_2$ range from $0$ to $+\infty$ (more precisely, from $1/T$ to $T$, with $T\rightarrow \infty$). That is,
\[
  \xn= \frac 1 2 \left( 2\cos \theta +i u_1 - \frac i {u_1}\right),
  \qquad
  \yn= \frac 1 2 \left( 2\cos (\beta-\theta) +i u_2 - \frac i {u_2}\right).
\]
Then, using~\eqref{xyn-pm}, we find 
\[
  1- \xn/\xn^+= \frac i{4\sin ^2\delta} \frac{(1-iu_1)^2 }{u_1},
\]
so that,  the following equation holds (recall that  $\delta \in (0,\pi)$),
\beq\label{r1-u1}
\tx:= \sqrt{1- \xn/\xn^+} = \frac 1{2\sin \delta} \left( e^{-i\pi/4} \sqrt {u_1} + e^{i\pi/4}/\sqrt {u_1}\right).
\eeq
Analogously,
\beq\label{r2-u2}
\ty:= \sqrt{1- \yn/\yn^+} = \frac 1{2\sin \vareps} \left( e^{-i\pi/4} \sqrt {u_2} + e^{i\pi/4}/\sqrt {u_2}\right).
\eeq
With the help  of~\eqref{D-param}, we can now write $D(\tx,\ty )$ in factorized form:
\beq\label{D-fact-u}
D(\tx,\ty )= \frac{(\sqrt{u_1}+\sqrt{u_2}e^{i\beta/2}) (\sqrt{u_1}+\sqrt{u_2}e^{-i\beta/2}) (\sqrt{u_1}\sqrt{u_2}+ie^{i\beta/2}) (\sqrt{u_1}\sqrt{u_2}+ie^{-i\beta/2}) }{4i{u_1}{u_2}},
\eeq
where we recall that $\beta/2=\delta+\vareps-\pi$. At this stage, due to the change of variables $(\xn,\yn ) \mapsto (u_1, u_2)$, we have
\begin{multline}\label{p0-tilde}
  \widetilde p_0(z_1,z_2)= \frac{\Delta \sqrt{\sigma_{11}\sigma_{22}}\kappa _0}{16\pi^2 \det^2 \Sigma}e^{-z_1 \xn^+ -z_2\yn^+} \times \\
  \iint_{\RR_+^2} \frac{e^{z_1 \xn^+ \tx^2+z_2 \yn^+  \ty ^2}}{D(\tx,\ty )} 
  \cdot \frac{\tx+\ty }{\tx \ty }  \left(1+ \frac 1 {u_1^2}\right) \left(1+ \frac 1 {u_2^2}\right) du_1 du_2,
\end{multline}
where $\tx$ and $\ty $ are given by~\eqref{r1-u1} and~\eqref{r2-u2}, and both integrals are in fact from $1/T$ to $T$, with $T\rightarrow \infty$.

\begin{figure}[htb]
  \centering
  \scalebox{0.55}{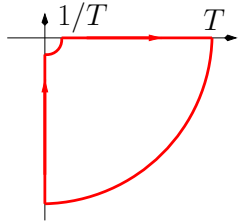}
  \caption{The integration contour (in $u_1$, then $u_2$) used to compute the integral $\widetilde p_0(z_1, z_2)$ by residues.}
  \label{fig:contour}
\end{figure}

The idea is now to deform the integration contour so that $\xn$ surrounds the cut $[\xn^+, +\infty)$, and analogously for $\yn$. This means that $u_1$ (resp.\ $u_2$) should now be of the form $-i v_1$ (resp.\ $-iv_2$), for $v_1, v_2$ in $\RR_+$. We will see that the final deformed integral is zero, but during the deformation we will capture some  residues of $\phi$. Let us fix $u_2$ in $\RR_+$, and begin with the transformation of~$u_1$. More formally, we apply the residue theorem to the contour shown in Figure~\ref{fig:contour}. A standard argument shows that the two arcs do not contribute in the limit $T\rightarrow \infty$, and  we obtain:
\beq\label{int-res}
\int_{0}^\infty \frac{e^{z_1 \xn^+ \tx^2}}{D(\tx,\ty )}
  \cdot \frac{\tx +\ty }{\tx \ty }\left(1+ \frac 1 {u_1^2}\right) du_1
=
-i \int_{0}^\infty \frac{e^{z_1 \xn^+ \tx^2}}{D(\tx,\ty )}
  \cdot \frac{\tx +\ty }{\tx \ty } \left(1- \frac 1 {v_1^2}\right) dv_1 -2i\pi \res(u_2) ,
\eeq
where now, on the right-hand side,
\beq\label{r1-new}
\tx=-\frac i{2\sin \delta} \left( \sqrt {v_1} -1/\sqrt {v_1}\right),
\eeq
and $\res(u_2)$ is the sum of residues corresponding to poles of the integrand (of the first integral) lying in the region shown in Figure~\ref{fig:contour}. Given the factorization~\eqref{D-fact-u}, and the fact that $\beta\in(0, \pi)$, the only possible pole  in this region is obtained for $\sqrt{u_1}=-ie^{i\beta/2}/\sqrt{u_2}$, which is there only  when $\beta \ge \pi/2$.  Let us temporarily assume that $\beta< \pi/2$, so that $\res(u_2)=0$. Let $I$ denote  the double integral in~\eqref{p0-tilde}.   Then it follows from~\eqref{int-res} that
\[
  I= -i \int_0^\infty e^{z_1 \xn^+ \tx^2} \left(1- \frac 1 {v_1^2}\right)\left(
    \int_0^\infty  \frac{e^{z_2 \yn^+ \ty ^2}}{D(\tx,\ty )} \cdot\frac{\tx+\ty }{\tx \ty }\left(1+ \frac 1 {u_2^2}\right)   du_2 \right) dv_1,
\]
where  $\tx$ is given by~\eqref{r1-new} and $\ty $ by~\eqref{r2-u2}. We now proceed with the second deformation, where~$v_1$ is fixed in $\RR_+$ and $u_2$ becomes $-iv_2$. We use again the contour of Figure~\ref{fig:contour}, now in the variable~$u_2$.
The integrand of the inner integral has exactly one pole in the relevant region, obtained for $\sqrt{u_2}=-ie^{i\beta/2}/\sqrt{u_1}=e^{i(\beta/2-\pi/4)}/\sqrt{v_1}$, or $u_2=-ie^{i\beta}/v_1$. The deformed  integral
\[
  \iint_{\RR_+^2}\frac{e^{z_1 \xn^+ \tx^2+z_2 \yn^+ \ty ^2}}{D(\tx,\ty )} 
  \cdot  \frac{\tx+\ty }{\tx \ty }\left(1- \frac 1 {{v_1}^2}\right) \left(1- \frac 1 {{v_2}^2}\right) dv_1 dv_2,
\]
where now
\[
  \ty =-\frac i{2\sin \vareps} \left( \sqrt {v_2} -1/\sqrt {v_2}\right),
\]
is zero, because the transformation $(v_1, v_2)  \mapsto (1/v_1,1/v_2)$ replaces this integral by its opposite. Hence
\beq \label{I-expr2}
I=  -i (-2i\pi) \int_{0}^\infty {e^{z_1 \xn^+ \tx^2}}
\left(1- \frac 1 {v_1^2}\right)
\Res_{u_2} \left(\frac{e^{z_2 \yn^+ \ty ^2}}{D(\tx,\ty )} \cdot  \frac{\tx+\ty }{\tx \ty }   \left(1+ \frac 1 {u_2^2}\right) \right)dv_1,
\eeq
where the residue is taken at  $u_2=u_2^{(0)}:=-ie^{i\beta}/v_1$. In the term
\beq\label{expr-integrand}
\frac{e^{z_2 \yn^+ \ty ^2}}{D(\tx,\ty )} \cdot  \frac{\tx+\ty }{\tx \ty }   \left(1+ \frac 1 {u_2^2}\right),
\eeq
the polar part comes from the denominator $D(\tx,\ty )$, and more precisely from
the term
\[
  \frac 1 {\sqrt{u_1}\sqrt{u_2} + ie^{i\beta/2}}=\frac{\sqrt{u_1}\sqrt{u_2} - ie^{i\beta/2}}{u_1u_2-e^{i\beta}},
\]
with $u_1=-iv_1$ as before, which has residue $2e^{i\beta}/v_1$  at $u_2=u_2^{(0)}$.  The non-singular  parts of~\eqref{expr-integrand} simply need to be evaluated at $u_2=u_2^{(0)}$. We find that at this point:
\allowdisplaybreaks
\begin{align*}
  \ty &= \frac i{2\sin \vareps} (e^{-i\beta/2}  \sqrt{v_1}-e^{i\beta/2}/\sqrt{v_1}),\\
  \tx+\ty &= - \frac {i\sin(\delta+\vareps)}{2\sin  \delta \sin\vareps} (e^{-i\delta}\sqrt{v_1} - e^{i\delta}/\sqrt{v_1}),\\
  \sqrt{u_1}+ \sqrt{u_2} e^{i\beta/2}&= e^ {i\beta/2-i\pi/4} (e^{-i\beta/2}\sqrt{v_1}+e^{i\beta/2} /\sqrt{v_1}),\\
  \sqrt{u_1}+ \sqrt{u_2} e^{-i\beta/2}&=e^{-i\pi/4}( \sqrt{v_1}+1/\sqrt{v_1}),\\
  \sqrt{u_1} \sqrt{u_2}+i e^{-i\beta/2}&=2\sin (\beta/2),\\
  4iu_1u_2&= -4i e^{i\beta},\\
  \frac 1 {\tx}\left( 1-\frac 1 {v_1^2}\right) &= \frac{2i\sin \delta }{v_1} (\sqrt {v_1} + 1/\sqrt{v_1}),\\
  \frac 1 {\ty }\left( 1+\frac 1 {u_2^2}\right) &=2i e^{-i\beta} v_1 \sin \vareps  (e^{-i\beta/2} \sqrt{v_1} +e^{i\beta/2} /\sqrt{v_1}).
\end{align*} 
When putting together all these contributions in~\eqref{I-expr2},  several simplifications occur, and one ends up with the following remarkably compact expression:
\beq\label{I-final}
I=
16 \pi i e^{z_1+z_2} \int_0^\infty \frac{e^{-\frac 1 2 (\bar zv_1+ z/v_1)}}
{v_1^{3/2}} \left(v_1e^{-i\delta }- e^{i\delta}\right) dv_1,
\eeq
where $z= (z_1+z_2e^{i\beta})$ and $\bar z$ is the complex conjugate of $z$. This integral can finally be computed exactly:
\[
  I=   16 \pi i \sqrt {2\pi}\,  e^{z_1+z_2- |z|} \left( \frac {e^{-i\delta}}{\sqrt {\bar z}} - \frac{e^{i\delta}}{\sqrt z}\right)
  =   32 \sqrt {2}\pi^{3/2}  e^{z_1+z_2- |z|}\,  \frac{\sin \om}{\sqrt{|z|}},
\]
where $\om = \arg ({e^{i\delta}}/{\sqrt z})= \delta- (\arg z)/2>0$.  One finds
\[
  |z|= \sqrt{z_1^2+z_2^2+2z_1 z_2\cos \beta },
\]
\[
  \cos(2\om)= \Re\left( \frac{\bar z}{|z|}e^{2i\delta}\right)
  = \frac 1 {|z|} \left( z_1 \cos(2\delta)+z_2\cos(2\delta-\beta)\right).
\]
Recalling the angle identities~\eqref{angles-00}, this gives
\[
  \sin^2\om= \frac 1 2 \left( 1-\cos(2\om)\right)=
  \frac 1{2|z|}
  \left( |z|+z_1\cos\theta+z_2\cos(\beta-\theta)\right).
\]
We now return to~\eqref{p0-tilde}, and recall that $I$ is the double integral in this expression. The exponential term in $\widetilde p_0(z_1,z_2)$ is
\[
  e^{-z_1\xn^+ -z_2\yn^+ +z_1 +z_2 -|z|}
  =e^{-z_1\cos \theta -z_2\cos(\beta-\theta)-|z|}
  =e^{-2|z|\sin^2\om}.
\]
Upon noticing that
\[
  \sin \om = \sin (\delta- a/2) = \sin( (\theta+\pi-a)/2)=\cos((\theta-a)/2),
\]
we have at last obtained  the expression of Proposition~\ref{prop:specialcase}... but only in the case $\beta<\pi/2$!

Fortunately, the calculation is not harder when $\beta >\pi/2$. In sight of~\eqref{int-res}, the double integral in~\eqref{p0-tilde} splits into a double integral and a simple one:
\[
  I=  -i  \int_{\RR_+^2} \frac{e^{z_1 \xn^+ \tx^2+z_2  \yn^+ \ty ^2}}{D(\tx,\ty )}
  \cdot \frac{\tx +\ty }{\tx \ty }\left(1- \frac 1 {v_1^2}\right) \left(1+ \frac 1 {u_2^2}\right) dv_1 du_2
  -2i\pi 
  \int_0^\infty e^{z_2  \yn^+ \ty    ^2}\res(u_2) du_2.
\]
In each  integral, we  perform the deformation of $u_2$ into $-i v_2$, for $v_2\in \RR_+$. In the double integral, no residue arises during the deformation, and the resulting integral in $v_1$ and $v_2$ is zero as argued above. The second integral (in $u_2$ only), once deformed,  is a variant of the one occurring in~\eqref{I-final}. One recovers the same expression as in the case $\beta <\pi/2$.

We leave the case  $\beta=\pi/2$ to the enthusiasts.

\bibliographystyle{abbrv} 
\bibliography{bib-SRBM}

\end{document}